\documentclass[12pt]{amsart}
\usepackage{geometry}   %????????
\usepackage[colorlinks,citecolor = red, linkcolor=blue,hyperindex]{hyperref}
\usepackage{euscript,eufrak,verbatim, mathrsfs}
\usepackage[psamsfonts]{amssymb}
\usepackage[usenames]{color}
\usepackage{graphicx}

 \usepackage[all, cmtip]{xy}

\usepackage{upref, xcolor}
\usepackage{amsfonts,amsmath,amstext,amsbsy, amsopn,amsthm}
\usepackage{enumerate}

\usepackage{url}
\usepackage{bookmark}

 \usepackage{euscript}
\usepackage{helvet}         % selects\textbf{\textbf{•}} Helvetica as sans-serif font
\usepackage{courier}        % selects Courier as typewriter font
\usepackage{multicol}        % used for the two-column index
\usepackage{enumitem} 

\newtheorem{theorem}{Theorem}[section]
\newtheorem*{theorem*}{Theorem B} 
\newtheorem{lemma}[theorem]{Lemma}

\newtheorem{proposition}[theorem]{Proposition}
\newtheorem{corollary}[theorem]{Corollary}
\newtheorem{definition}[theorem]{Definition}
\newtheorem*{definition*}{Definition}
\newtheorem*{remark*}{Remark}

\newtheorem*{observation*}{Observation}

\newtheorem*{assumption*}{Assumption}
\newtheorem*{question*}{Question}
\newtheorem{remark}[theorem]{Remark}

\newcommand{\R}{\mathbb{R}}
\newcommand{\N}{\mathbb{N}}
\newcommand{\Z}{\mathbb{Z}}

\newcommand{\C}{\mathbb{C}}
\newcommand{\E}{\mathbb{E}}
\newcommand{\PP}{\mathbb{P}}
\newcommand \Prob {{\mathbb P}}

\newcommand \eps {{\varepsilon}}
\newcommand \la {{\lambda}}

\newcommand \EE {{\mathbb{E}}}
\newcommand \KK {{\mathbb{K}}}
\newcommand \RR {{\mathbb{R}}}

\newcommand{\Conf}{\mathrm{Conf}}

\newcommand{\spann}{\mathrm{span}}

\newcommand{\Var}{\mathrm{Var}}
\newcommand{\Cov}{\mathrm{Cov}}
\newcommand{\supp}{\mathrm{supp}}
\newcommand{\sgn}{\mathrm{sgn}}

\newcommand{\Pf}{\mathrm{Pf}}

\newcommand{\const}{\mathrm{const}}

\newcommand{\sine}{\mathrm{sine}}
\newcommand{\Sine}{\mathrm{Sine}}

\newcommand{\an}{\text{\, and \,}}

\newcommand{\ch}{\mathbb{I}}

\DeclareMathOperator{\Def}{Def}

\begin{document}

%\title{Number rigidity for stationary Pfaffian processes}
\title{On Number Rigidity for Pfaffian Point Processes}

\author%[authorlabel1]
{Alexander I. Bufetov}
\address%[authorlabel1]
{Alexander I. BUFETOV: 
Aix-Marseille Universit\'e, Centrale Marseille, CNRS, Institut de Math\'ematiques de Marseille, UMR7373, 39 Rue F. Joliot Curie 13453, Marseille, France;
Steklov Mathematical Institute of RAS, Moscow, Russia}
\email{bufetov@mi.ras.ru, alexander.bufetov@univ-amu.fr}

\author%[authorlabel1]
{Pavel P. Nikitin}
\address%[authorlabel1]
{Pavel P. Nikitin: 
St. Petersburg Department of V.A.Steklov Institute of Mathematics of the Russian Academy of Sciences, 27 Fontanka, 191023, St. Petersburg, Russia;
St. Petersburg State University, St. Petersburg, Russia}
\email{pnikitin0103@yahoo.co.uk}

\author%[authorlabel1]
{Yanqi Qiu}
\address%[authorlabel1]
{Yanqi QIU: Institute of Mathematics, Academy of Mathematics and Systems Science, Chinese Academy of Sciences, Beijing 100190, China}
\email{yanqi.qiu@hotmail.com}

\begin{abstract}
Our first result states that the orthogonal and symplectic Bessel processes are rigid in the sense of Ghosh and Peres. 
Our argument in the Bessel case proceeds by an estimate of the variance of additive statistics in the spirit of Ghosh and Peres.
Second, a sufficient condition for number rigidity of stationary Pfaffian processes, relying on the Kolmogorov criterion for interpolation of stationary processes and applicable, in particular, to pfaffian sine-processes, is given in terms of the asymptotics of the spectral measure for additive statistics. 
\end{abstract}


\maketitle

\setcounter{equation}{0}

\section{Introduction}
\subsection{Point processes and number rigidity}

Let $\Conf(\R)$ be the set of  non-negative integer-valued  Radon measures on the real line $\R$.   Elements of $\Conf(\R)$ are called (locally finite) configurations on $\R$.  The space $\Conf(\R)$ is a Polish space equipped with the  vague topology generated by the maps:
\[
\xi \mapsto \int_\R f d\xi \text{\, for compactly supported continuous functions $f: \R \rightarrow \C$. }
\]
By definition, a point process on $\R$ is a  Borel probability measure on $\Conf(\R)$.

A configuration $\xi \in \Conf(\R)$ is called simple if $\xi (\{x\}) \in \{0, 1\}$ for all $x \in \R$.  A point process $\PP$ is called simple, if $\PP$-almost every configuration is simple. 

Given an element $\xi \in \Conf(\R)$ and any Borel subset $S\subset \R$, we denote $\xi|_S$ the restriction of the measure $\xi$ on $S$.

\begin{definition}[Ghosh \cite{Ghosh-sine}, Ghosh-Peres\cite{Ghosh-rigid}]\label{defn-rig}
A point process $\PP$ on $\R$ is called \emph{number rigid} if for any bounded Borel subset $B \subset \R$, there exists a Borel function $F_{B}: \Conf(\R)\rightarrow \Z$ such that 
\[
\xi(B) = F_{B} (\xi|_{\R\setminus B}), \text{\, for $\PP$-almost every $\xi \in \Conf(\R)$.}
\]
\end{definition}

\subsection{Pfaffian point processes}

Recall that for a simple point process $\PP$ on $\R$, the \emph{$k$-point correlation function} $\rho_\PP^{(k)}$ of $\PP$ with respect to the Lebesgue measure, if it exists,   is the non-negative function $\rho_\PP^{(k)}: \R^k \rightarrow \R$ such that for any continuous compactly supported function $\varphi: \R^k \rightarrow \C$, we have 
\[
\int\limits_{\Conf(\R)} \sum_{x_1, \dots, x_k \in \mathcal{X}}^*  \varphi(x_1, \dots, x_k) \PP(d \mathcal{X}) =   \int\limits_{\R^k} \varphi(x_1, \dots, x_k) \rho_\PP^{(k)} (x_1, \dots, x_k) dx_1 \cdots dx_k,
\]
where $\sum\limits^{*}$ denotes the sum over all ordered $k$-tuples of {\it distinct} points $(x_1, \dots, x_k) \in \mathcal{X}^k$. 

A simple point  process $\PP$ on $\R$ is said to be a  \emph{Pfaffian point process}  if there exists a matrix kernel  $K: \R \times \R  \rightarrow \C^{2 \times 2}$ such that  for all positive integers $k$, the  $k$-point correlation functions of $\PP$ exist and  have the form
\[
\rho_\PP^{(k)} (x_1, \cdots, x_k) =  \Pf[K(x_i, x_j)J]_{1\le i,j\le k}. 
\]
Here $\Pf(A)$ is the Pfaffian of an antisymmetric matrix and the matrix kernel $K$ must satisfy the condition
\begin{align}\label{anti-cond}
(K(x,y)J)^t = - K(y,x)J, \text{\, where\,} J = 
\begin{bmatrix}
0  & 1\\
-1 & 0
\end{bmatrix}
\end{align}
to ensure that the $2k \times 2k$ matrix $[K(x_i, x_j)J]_{1\le i,j\le k}$ is antisymmetric.  In this situation, we say that the point process $\PP$ is the Pfaffian point process induced by the matrix kernel $K$ and is denoted $\PP_K$. 

We can write the matrix kernel $K$ as 
\begin{align}\label{def-matrix-K}
K(x,y) = 
\begin{bmatrix}
K_{11}(x,y) & K_{12}(x,y) \\
K_{21}(x,y) & K_{22}(x,y) 
\end{bmatrix},
\end{align}
where the entries $K_{ij}: \R \times \R \rightarrow \C$ are scalar functions and  then the condition \eqref{anti-cond} says that \begin{equation}\label{eq:skew-symm}
K_{22}(x,y) = K_{11}(y,x), \  K_{12}(x,y) = - K_{12}(y,x), \ K_{21}(x,y) = - K_{21}(y,x). 
\end{equation}

We recall the general structure of the Pfaffian kernels for $\beta = 4$ (symplectic) ensembles and $\beta = 1$ (orthogonal) ensembles and their scaling limits:
\begin{gather}
\KK_4(x,y) =
\frac12\begin{bmatrix}
			K_4(x,y)																				& -\int_x^y K_4(x,t)dt	\\
\frac{\partial }{\partial x} K_4(x,y)		&				K_4(y,x)
\end{bmatrix},		\label{eq:kernel_4}				\\
\KK_1(x,y) =
\begin{bmatrix}
			K_1(x,y)																				& -\int_x^y K_1(x,t)dt	- 1/2 \sgn(x-y)\\
\frac{\partial }{\partial x} K_1(x,y)		&				K_1(y,x)
\end{bmatrix},		\label{eq:kernel_1}	
\end{gather}
for some particular $K_1(x,y)$, $K_4(x,y)$. In the integrable case a kernel has the following form
\begin{equation*}%\label{integ}
K_{\beta}(x,y)=\displaystyle \frac{A(x)B(y) - B(x)A(y)}{x-y} + C(x)D(y).
\end{equation*}

Section~\ref{sec:Bessels} is devoted to the Pfaffian Bessel processes. 
Recall that a classical polynomial $\beta$-ensemble is defined by the probability density function
$$
\const(\beta, w_\beta) \prod_{i=1}^N w_\beta(x_i) \prod_{1\le i<j\le N} |x_i - x_j|^\beta,
$$
where $w_\beta$ corresponds to one of the classical weights (Hermite, Laguerre, Jacobi polynomials) and $N\in \N$. If $\beta=2$ then the corresponding processes are determinantal,
and if $\beta = 1$ or $\beta=4$ then the corresponding processes are Pfaffian (orthogonal for $\beta = 1$ and symplectic for $\beta = 4$).
Pfaffian Bessel processes arise as the scaling limits of Laguerre and Jacobi Pfaffian ensembles in the  {\it hard edge scaling limit}, see~\cite{Forrester-log}, 7.7.1 and 7.9.1 and~\eqref{eq:Laguerre-kernel} below. For the determinantal Bessel processes, see \cite{TW-Bessel}. 

The symplectic Bessel kernel $\KK^{Bessel}_{4,s}(x,y)$ is given by the formula
\begin{gather*}
\KK^{Bessel}_{2,s}(x,y) =  \frac{x^{1/2}J_{s+1}(x^{1/2})J_s(y^{1/2}) - y^{1/2}J_{s+1}(y^{1/2})J_s(x^{1/2})}{2(x-y)},\\
K^{Bessel}_s(x,y) =  2 \left(\frac{x}{y}\right)^{1/2} \KK^{Bessel}_{2, 2s-1}(4x, 4y) -
			 \frac{J_{2s-1}(2y^{1/2})}{2y^{1/2}} \int_0^{x^{1/2}} J_{2s-1}(2t)dt,  \\
\KK^{Bessel}_{4,s}(x,y) =
\begin{bmatrix}
			K_s^{Bessel}(x,y)														& 	\int_y^x K_s^{Bessel}(x,t)dt	\\
\frac{\partial }{\partial x} K_s^{Bessel}(x,y)		&		K_s^{Bessel}(y,x)
\end{bmatrix},  
\end{gather*}
where $s>0$. Regarding the formula for $K^{Bessel}_s(x,y)$, see Proposition~\ref{prop:Bessel-4_formula}.

The orthogonal Bessel kernel $\KK^{Bessel}_{1,s}(x,y)$ is given by the formula
\begin{gather*}
%\KK^{Bessel}_{2, s}(x,y) =  \frac{x^{1/2}J_{s+1}(x^{1/2})J_s(y^{1/2}) - y^{1/2}J_{s+1}(y^{1/2})J_s(x^{1/2})}{2(x-y)},	\\
K_{1,s}(x,y) =  \left(\frac{x}{y}\right)^{1/2} \KK^{Bessel}_{2, s+1}(x, y) +
			 \frac{J_{s+1}(y^{1/2})}{4y^{1/2}} \int_{x^{1/2}}^\infty J_{s+1}(t)dt,  \\
\KK^{Bessel}_{1,s}(x,y) =
\begin{bmatrix}
				K_{1,s}(x,y)												& 	-\int_x^y K_{1,s}(x,t)dt	- 1/2 \sgn(x-y)\\
\frac{\partial }{\partial x} K_{1,s}(x,y)		&		K_{1,s}(y,x)
\end{bmatrix}. 
\end{gather*}
\begin{theorem}\label{thm:Bessel-rigidity}
\begin{enumerate}[label=(\roman*)]
\item\label{thm-item:symplect-Bessel-rigidity}
The symplectic Bessel process is number rigid.
\item\label{thm-item:orthog-Bessel-rigidity}
The orthogonal Bessel process is number rigid.
\end{enumerate}
\end{theorem}
The number rigidity of the Pfaffian Bessel processes is proved in subsections~\ref{sec:symplect-Bessel},~\ref{sec:orthog-Bessel}.  We use the following sufficient condition for the number rigidity of a point process due to Ghosh and Peres.

\begin{proposition} [Ghosh \cite{Ghosh-sine}, Ghosh and Peres \cite{Ghosh-rigid}]\label{prop:G-P_ridigity-cond}
	Let $M$ be a complete metric separable space.  Let $\mathbb{P}$ be a Borel probability measure on $\mathrm{Conf}(M)$.  Assume that for any $\varepsilon >0$ and any bounded subset $B \subset M$ there exists a bounded measurable function $f$ of bounded support such that $f\equiv 1$ on $B$ and $\mathrm{Var}_{\mathbb{P}} S_f < \varepsilon$, 
where $S_f(X) = \sum_{x\in X} f(x)$, $X\in \mathrm{Conf}(M)$. Then the measure $\mathbb{P}$ is number rigid.
\end{proposition}

We give an explicit formula for the variance of an additive functional of a Pfaffian point process in~\eqref{eq:Pfaff-var_general}, subsection~\ref{sec:variance}. 
Preliminary integral estimates are discussed in subsection~\ref{sec:prelim-estimates}. A difference with the determinantal case can be seen as follows. Consider a point process on $\mathbb R$ with the first correlation function $\rho^{(1)}(x)$, second corrleation function $\rho^{(2)}(x,y)$ and truncated second correlation function $\rho^{(2, T)}(x,y)=\rho^{(2)}(x,y)-\rho^{(1)}(x)\rho^{(1)}(y)$. For a determinantal process governed by an orthogonal projection, we have
\begin{equation}\label{eq:int_rho-trunc}
\int_\RR \rho^{(2, T)}(x,y) dy = -\rho^{(1)}(x),
\end{equation}
 In a discussion of {\it perfect screening} in~\cite{Forrester-log}, 14.1, p.660,  Forrester writes that property~\eqref{eq:int_rho-trunc} should ``remain valid in the thermodynamic limit''. We show that~\eqref{eq:int_rho-trunc} is not valid for the Pfaffian Bessel processes, see \eqref{no-scr-4}, \eqref{no-scr-1}. Nonetheless, a weaker integral property holds (see Proposition~\ref{prop:int-rho-2_average-cond} and Proposition~\ref{prop:int-rho-2_average-cond_beta=1}) and suffices for our purposes.

We next give a general sufficient condition for the number rigidity of a stationary point process convenient for working with stationary  Pfaffian  processes.
Let $\PP$ be a stationary point process on $\R$ admitting  the first  and the second correlation functions $\rho_\PP^{(1)}$ and $\rho_\PP^{(2)}$. 
The first correlation function is a constant, and we set
\[
\rho = \rho_\PP^{(1)}(x).
\]
Denote also
$$
 F(x) = \rho_\PP^{(2)}(x, 0)- \rho^2.
 $$
 \begin{proposition}\label{g-stat}
 Assume that there exists $C> 0$ such that the Fourier transform $\widehat{F}$ of $F$ satisfies
\begin{align}\label{ass-str}
0 \le \widehat{F}(\lambda)  + \rho \le C |\lambda| \text{ \, for all $\lambda \in \R$.}
\end{align}
Then the point process $\PP$ is number rigid.
\end{proposition} 

For example, 
Pfaffian sine processes arise as {\it bulk scaling limits} of the Pfaffian Gaussian ensembles, see~\cite{Forrester-log}, 7.6.1. and 7.8.1.
Let
$$
 K_{\sine,2}(x,y) = S(x-y), \quad S(x) :  = \frac{\sin(\pi x)}{ \pi x}
$$
be the standard sine-kernel. Then, the orthogonal sine process or $\Sine_1$-process, is the Pfaffian point process on $\R$ with a matrix correlation kernel 
\[
\KK_{\sine,1}  (x,y) =  \left[\begin{array}{cc}  S (x-y)&     IS(x-y)  - \varepsilon(x-y) \\  S'(x-y) & S (x-y)  \end{array}\right],
\]
where 
\begin{align}\label{def-S}
IS(x) : = \int_0^x S(t)dt \an \varepsilon(x) := \frac{1}{2}\sgn(x),
\end{align}
and the symplectic sine process, the $\Sine_4$-process, is the Pfaffian point process on $\R$ with a matrix correlation kernel 
\[ 
\KK_{\sine,4}  (x,y) =  \frac{1}{2}\left[\begin{array}{cc}  S (x-y)&     IS(x-y) \\  S'(x-y) & S (x-y)  \end{array}\right].
\]
In these cases, the quantity $\widehat{F}(\lambda)-\rho=\widehat{F}(\lambda)-\widehat{F}(0)$ is computed in Forrester \cite{Forrester-log}.
For the orthogonal sine process, Forrester \cite[formula (7.136)]{Forrester-log} gives
\[
\widehat{F}(\lambda)  - \widehat{F}(0) =\left\{ \begin{array}{cc}  2 | \lambda| -  | \lambda| \log ( 1 + 2 | \lambda|) & \text{if $| \lambda| \le 1$} \vspace{2mm}\\ 2 -  | \lambda| \log \frac{2|\lambda| + 1}{2 | \lambda| -1}& \text{ if $| \lambda| \ge 1$} \end{array}\right. .
\] 
For the symplectic sine process, Forrester \cite[formula (7.95)]{Forrester-log} gives
\[
\widehat{F}(\lambda)  - \widehat{F}(0) =\left\{ \begin{array}{cc}  \frac{|\lambda|}{2}- \frac{|\lambda|}{4}  \log \Big| 1 - 2 |\lambda|\Big|&  \text{ if $| \lambda| \le 1$} \vspace{2mm}\\ \frac{1}{2}& \text{ if $| \lambda| \ge 1$} \end{array}\right. .
\]
In both the orthogonal and symplectic case, we have 
\[
0 \le \widehat{F}(\lambda)  - \widehat{F}(0)  \le C| \lambda|,
\]
and Proposition \ref{g-stat} yields 
\begin{proposition}\label{prop-sine-1}
The orthogonal sine process is number rigid. 
\end{proposition}
\begin{proposition}\label{prop-sine-4}
The symplectic sine process is number rigid. 
\end{proposition}
\begin{remark*} 
For the general $\Sine_\beta$ processes, rigidity is due to Chhaibi and Najnudel \cite{rena}, and Propositions \ref{prop-sine-1}, \ref{prop-sine-4} of course follow from their result. Their argument is quite different. It would be interesting to obtain a spectral asymptotics at zero for general $\Sine_{\beta}$ processes. 
\end{remark*}

\begin{remark*} 
The  \emph{soft edge scaling limit} yields Pfaffian Airy kernels, and  it would be interesting to prove the rigidity of the corresponding point processes.
\end{remark*}

\section{Pfaffian Bessel point processes}\label{sec:Bessels}

%Pfaffian Bessel processes are not stationary, and we plan to use the following sufficient condition for the rigidity of a point process due to Ghosh and Peres. 

%Thus our first task is to obtain a reasonable formula for the variance of an additive functional of a Pfaffian point process.

\subsection{Variance of an additive functional}\label{sec:variance}

For a general point process we have the following formula for the variance of $S_f(X) = \sum_{x\in X} f(x)$, $X\in \mathrm{Conf}(M)$.
\begin{multline*}
\Var(S_f) = \EE_{\Prob_{\KK}}(|S_f|^2) - |\EE_{\Prob_\KK}(S_f)|^2 = \\ 
  = \EE_{\Prob_{\KK}}\biggl(\sum_{x\in X} |f(x)|^2\biggr) + \EE_{\Prob_{\KK}}\biggl(\sum_{x,y\in X, x\neq y} f(x)\overline{f(y)}\biggr)  - 
		\biggl|\EE_{\Prob_{\KK}}\biggl(\sum_{x\in X} f(x)\biggr)\biggr|^2 = \\ 
	=  \int_\RR|f(x)|^2 \rho^{(1)}_{\Prob_\KK}(x) dx + \int_{\RR^2}f(x)\overline{f(y)}\rho^{(2)}_{\Prob_\KK}(x,y) dxdy - \biggl|\int_\RR f(x)\rho^{(1)}_{\Prob_\KK}(x)dx \biggr|^2  =  \\
	=  \int_\RR|f(x)|^2 \rho^{(1)}_{\Prob_\KK}(x) dx + \int_{\RR^2}f(x)\overline{f(y)}\rho^{(2, T)}_{\Prob_\KK}(x,y) dxdy =  \\
	=  \int_\RR|f(x)|^2 \biggl( \rho^{(1)}_{\Prob_\KK}(x) + \int_{\RR^2}\rho^{(2, T)}_{\Prob_\KK}(x,y) dy\biggr) dx - \frac12\int_{\RR^2} |f(x)-f(y)|^2 \rho^{(2, T)}_{\Prob_\KK}(x,y) dxdy,
\end{multline*}
where $\rho^{(2, T)}_{\Prob_\KK}(x,y) = \rho^{(2)}_{\Prob_\KK}(x,y) - \rho^{(1)}_{\Prob_\KK}(x)\rho^{(1)}_{\Prob_\KK}(y)$ is the \emph{truncated second correlation function} of $\Prob_\KK$. If we additionally have~\eqref{eq:int_rho-trunc}, then we finally obtain the formula
\begin{equation}\label{eq:gen-var}
\Var(S_f) = -\frac12 \int_{\RR^2} |f(x)-f(y)|^2 \rho^{(2, T)}_{\Prob_\KK}(x,y) dxdy.
\end{equation}

For a Pfaffian point process induced by a matrix kernel $\KK$, by definition, we have
$$
\rho^{(1)}_{\Prob_\KK}(x) = \KK_{1,1}(x,x), \quad \rho^{(2, T)}_{\Prob_\KK}(x,y) = - \det \KK(x,y).
$$
 Thus using~\eqref{eq:gen-var} we obtain that if the condition~\eqref{eq:int_rho-trunc} holds, then 
\begin{equation}\label{eq:Pfaff-var_perf-screening}
\Var(S_f) = \frac12 \int_{\RR^2} |f(x)-f(y)|^2 \det \KK(x,y) dxdy
\end{equation}
 and in general case, we have
\begin{multline}\label{eq:Pfaff-var_general}
\Var(S_f) = \int_\RR|f(x)|^2 \biggl( \KK_{1,1}(x,x) - \int_{\RR} \det \KK(x,y) dy\biggr) dx  +  \\
		+\frac12 \int_{\RR^2} |f(x)-f(y)|^2 \det \KK(x,y) dxdy.
\end{multline}

We recall once again the formulas for the kernels for $\beta = 4$ (symplectic ensembles) and $\beta = 1$ (orthogonal ensembles):
\begin{gather*}
\KK_4(x,y) =
\frac12\begin{bmatrix}
			K_4(x,y)																				& -\int_x^y K_4(x,t)dt	\\
\frac{\partial }{\partial x} K_4(x,y)		&				K_4(y,x)
\end{bmatrix},		%\label{eq:kernel_4}				
\\
\KK_1(x,y) =
\begin{bmatrix}
			K_1(x,y)																				& -\int_x^y K_1(x,t)dt	- \frac{1}{2} \sgn(x-y)\\
\frac{\partial }{\partial x} K_1(x,y)		&				K_1(y,x)
\end{bmatrix}.%		\label{eq:kernel_1}	
\end{gather*}

\begin{remark}\label{rem:skew-symm+det-symm}
We will consider Pfaffian sine and Bessel processes, they arise as limits of the Pfaffian polynomial ensembles. Kernels of these polynomial ensembles are skew-symmetric by construction, therefore the limit kernels are also skew-symmetric (it is not obvious from the definition of the Pfaffian Bessel kernel).
We note also that $\det K(x,y)$ is a symmetric function if $K(x,y)J$ is skew-symmetric.
\end{remark}

\begin{proposition}\label{prop:proj_K-1_K-4}
If $K(x,y)$ is a projection, $\frac{\partial }{\partial x} K(x,y)$ is skew-symmetric function, and for any $x\in \RR$ we have
$$
\lim_{y\to\pm\infty} K(y,x) = 0, \quad \lim_{x\to\pm\infty} K(y,x) \int_x^y K(x,t)dt = 0,
$$ 
then the condition~\eqref{eq:int_rho-trunc} holds for the kernels $\KK_4(x,y)$ and $\KK_1(x,y)$.
\end{proposition}
\begin{proof}
For $\KK_4(x,y)$ we have 
$$
\det \KK_4(x,y) = \frac14 \biggl( K(x,y)K(y,x) + \frac{\partial }{\partial x} K(x,y) \int_x^y K(x,t)dt \biggr),
$$
and we use the equality $\frac{\partial }{\partial x} K(x,y) = -\frac{\partial }{\partial y} K(y,x)$ to write
\begin{multline*}
\int_\RR \biggl( \frac{\partial }{\partial x} K(x,y) \int_x^y K(x,t)dt \biggr) dy = 
-\int_\RR \biggl( \frac{\partial }{\partial y} K(y,x) \int_x^y K(x,t)dt \biggr) dy =  \\
-K(y,x) \int_x^y K(x,t)dt \biggl\vert_{-\infty}^{\infty} +
\int_\RR K(y,x)K(x,y) dy = \int_\RR K(x,y)K(y,x) dy,
\end{multline*}
and finally
$$
\int_\RR \det \KK_4(x,y) dy = \frac12 \int_\RR K(x,y)K(y,x) dy = \frac12 K(x,x).
$$

For $\KK_1(x,y)$ we have 
\begin{multline*}
\det \KK_1(x,y) = \biggl( K(x,y)K(y,x) + \frac{\partial }{\partial x} K(x,y) \int_x^y K(x,t)dt \biggr) + \\
	+ \frac{1}{2}\sgn(x-y) \frac{\partial }{\partial x} K(x,y) =  4\det \KK_4(x,y) - \frac{1}{2} \sgn(x-y) \frac{\partial }{\partial y} K(y,x).
\end{multline*}
Integrating the last term we obtain
\begin{multline*}
- \frac12 \int_\RR \sgn(x-y) \frac{\partial }{\partial y} K(y,x) dy = - \frac12\int_{-\infty}^x \frac{\partial }{\partial y} K(y,x) dy +  \\
+ \frac12\int_x^{\infty} \frac{\partial }{\partial y} K(y,x) dy = - K(x,x),
\end{multline*}
thus
$$
\int_\RR \det \KK_1(x,y) dy = 2K(x,x) - K(x,x) = K(x,x).
$$
\end{proof}

For a Pfaffian process with a kernel of the form $\KK_1(x,y)$ or $\KK_4(x,y)$, let us define {\it the defect} of the process as the difference
$$
\Def_\KK(x) := \int_\RR K(x,y)K(y,x) dy - K(x,x).
$$
We have the following corollary from the proof of the previous proposition; it will be useful for the analysis of the Bessel processes later.
\begin{corollary}\label{cor:int-defect}
For  $\KK = \KK_4$ we have
\begin{align*}
-  \int_\RR \rho^{(2, T)}_{\Prob_\KK}(x,y) dy -  \rho^{(1)}_{\Prob_\KK}(x) 
		=	 \frac{1}{4} \Big[ 2\Def_{\KK}(x) - K(y,x) \int_x^y K(x,t)dt \biggl\vert_{y=-\infty}^{y=\infty}\Big].
\end{align*}
For $\KK = \KK_1$, if for any $x\in \RR$ we have
$$
\lim_{y\to\pm\infty} K(y,x) = 0,
$$  then 
\begin{align*}
-  \int_\RR \rho^{(2, T)}_{\Prob_\KK}(x,y) dy - \rho^{(1)}_{\Prob_\KK}(x)   
		=	2\Def_{\KK}(x) - K(y,x) \int_x^y K(x,t)dt \biggl\vert_{y=-\infty}^{y=\infty}.
\end{align*}
\end{corollary}

\subsection{Preliminary propositions}\label{sec:prelim-estimates}

We plan to estimate the variance~\eqref{eq:Pfaff-var_general} term by term. From the formulas in the previous section we see that it leads to the estimates of the integrals of the type
\begin{equation}\label{eq:int-type}
\int_{\RR_+^2} |f(x)-f(y)|^2 \Pi(x,y) dxdy,
\end{equation}
where for most of the summands we will have $\Pi(x,y) = A(x,y)/(x-y)^2$, or $\Pi(x,y) = B(x,y)/(x-y)$ or $\Pi(x,y) = C(x)D(y)$ for some reasonable functions $A(x,y)$, $B(x,y)$, $C(x), D(y)$. The corresponding integrals usually will be not absolutely convergent, and we will use the oscillation of $\Pi(x,y)$ to obtain a required estimate. 
In this subsection sufficient conditions on a kernel $\Pi(x,y)$ in these three cases are stated in Propositions~\ref{prop:rig-det}, \ref{prop:Pi/(x-y)_rig} and \ref{prop:split-var_rig} 
respectively.

\begin{remark}\label{rem:RR}
We have noted that $\det K(x,y)$ is symmetric for a skew-symmetric kernel, therefore in our case it is sufficient to estimate all the integrals only for $y\ge x$. For the general case one can split the domain of integration into two parts and change the variables.

Moreover, in the present paper we consider the case of $\RR$ only for the sine process; the proof of the rigidity in this case is much simpler than for Bessel process, and we can also use the stationarity to give an alternative proof as presented in section \ref{sec-stat}. Thus we decided to present the detailed proof for the case of the processes on $\RR_+$.
Analogous propositions are true when we consider $\RR$ instead of $\RR_+$, with almost the same arguments.
\end{remark}

We will use the family of functions that were used in \cite{AB_det-rigidity} for the determinantal point processes. Take $R>0$, $T>R$ and set 
\begin{equation*}
\varphi^{(R, T)}(x) = \begin{cases}
1,	\quad		x \le R; 		\\
1-\displaystyle\frac{\log(x-R+1)}{\log(T-R+1)}  \  \mathrm{if} \  R\leq x\leq T;\\
0, \quad		T \le x.
\end{cases}
\end{equation*}

We split the domain of integration in the following way:
\begin{gather*}
D = \{(x,y) | y\ge x \ge 0\} \subset \RR_+\times \RR_+,\quad D = D_{>R} \sqcup D_{<R} \sqcup \{(x,y) |R > y \ge x\},  \\
D_{>R} = \{(x,y) | y\ge x \ge R\}, \quad D_{<R} = \{(x,y) | y\ge R > x\}.
\end{gather*}
The difference $\varphi^{(R, T)}(x) - \varphi^{(R, T)}(y)$ is zero on $\{(x,y) |R > y \ge x\}$, therefore it is sufficient to consider the two other domains $D_{>R}$ and $D_{<R}$.

First we restate Proposition 1.1 from \cite{AB_det-rigidity} in our case as follows.

\begin{proposition}\label{prop:rig-det}
Let $\Pi(x,y)$ be a function on $D$.
\begin{enumerate}[label=(\roman*)]  %--  ?? turned off enumitem. Turn it on when there is a connection!!
\item \label{prop:gen-rig_it-1}
	Assume that there exists $\alpha\in (0, 1/2)$ and, for any $R>0$, a constant $\const(R)>0$ such that for any $y\ge x\geq R$ we have
	$$
	|\Pi(x,y)|\leq \const(R)\cdot \frac{({x}/{y})^{2\alpha}+ ({y}/{x})^{2\alpha}}{(x-y)^2}.
	$$
	Then we have
	\begin{equation*}%\label{eq:var-estimate_(x-y)^2_x>R}
	\int_{D_{>R}}|\varphi^{(R, T)}(x) - \varphi^{(R, T)}(y)|^2 |\Pi(x,y)| dxdy \xrightarrow[T\to \infty]{} 0.
	\end{equation*}
\item \label{prop:gen-rig_it-2}
	Assume that there exists $\varepsilon>0$, and, for any $R>0$, a constant $\const(R)>0$ such that for any $y\ge R$ we have 
	$$
	\int_{-R}^R|\Pi(x,y)|dx\leq \frac{\const(R)}{y^{1+\varepsilon}}.
	$$
	Then we have
	\begin{equation*}%\label{eq:var-estimate_(x-y)^2_x<R}
	\int_{D_{<R}}|\varphi^{(R, T)}(x) - \varphi^{(R, T)}(y)|^2 |\Pi(x,y)| dxdy \xrightarrow[T\to \infty]{} 0.
	\end{equation*}
\end{enumerate} 
\end{proposition}

\begin{remark}\label{rem:bounded-kernel}
The result of Proposition~\ref{prop:rig-det} is also valid for $\Pi(x,y) = \Pi_1(x,y)/(x-y)^2$, once $\Pi_1(x,y)$ is bounded. The assumption \ref{prop:gen-rig_it-2}  of Proposition \ref{prop:rig-det} doesn't hold,
but we can see from the proof of the Proposition~\ref{prop:rig-det} that what we actually need is the decreasing bound for the integral
\begin{multline*}
\biggr| \int\limits_R^{\infty} \int\limits_0^R (\varphi^{(R,T)}(x)-1)^2 \Pi(x,y) dy dx \biggl| \le
\frac{\const(R)}{(\log T)^{2}} \int\limits_R^{\infty} \int\limits_0^R \frac{\log^2(x-R+1)}{(x-y)^2} dy dx \le \\
\le \frac{\const(R)}{(\log T)^{2}} \int\limits_R^{\infty} \frac{\log^2(x-R+1)}{x(x-R)} dx \le \frac{\const(R)}{(\log T)^{2}}.
\end{multline*}
\end{remark}

Recall the following well-known lemma (Abel-Dirichlet's test)
\begin{lemma}
Let $f(x)$, $g(x)$ be two functions on $[a, b)$, where $g(x)$ is monotonic, differentiable, $\lim_{x\to b} g(x) = 0$, and there exists a constant $M>0$ such that for any $c\in [a,b)$ we have
$$
\biggl| \int_a^c f(x) dx \biggr| \le M.
$$
Then
\begin{equation}\label{eq:Abel-Dirichlet}
\left| \int_a^b f(x)g(x) dx \right|  \le   M\cdot |g(a)|.
\end{equation}
\end{lemma}

We will also use the following simple analog of Riemann-Lebesgue lemma:
\begin{lemma}\label{lem:sin-la-x-int-by-parts}
Let $f(x)$ be a differentiable function on $[a,b]$, then for every $\alpha\in\R$ we have
$$
\biggl| \int_a^b f(x) \sin(m x^{\alpha} + n)x^{\alpha-1} dx \biggr|  \le  \frac{\const}{m} \left(\max_{a\le x\le b}|f(x)| + \int_a^b |f'(x)| dx \right).
$$
\end{lemma}

\begin{proposition}\label{prop:Pi/(x-y)_rig}
Let $\Pi(x,y)$ be a function on $D$. We set $\la = y/x$ and $\bar\Pi(x, \la) = \Pi(x, \la x)$.  
\begin{enumerate}[label=(\roman*)]  %--  ?? turned off enumitem. Turn it on when there is a connection!!
\item \label{eq:var-estimate_x>R}
	Assume that there exist constants $1\ge\eps_1 > \eps_4 \ge 0$,  $1 >\eps_2 > \eps_5 \ge 0$, $\eps_3 > \eps_6 \ge 0$, $\la_1 > 1$, $\la_2 < 1$, a positive function $\psi$, 
	$\psi(T) = o(\log^2(T))$ when $T\to\infty$, and, for any $R>0$, a constant $\const(R)>0$ such that the following holds:
	\begin{gather} \label{eq:main-prop-cond1}
	\max_{a,b>R} \biggl| \int_a^b \bar{\Pi}(x, \la) dx \biggr| \leq \const(R)\psi(\max(a,b))\biggl(\frac{\la^{\eps_4}}{|\la-1|^{\eps_1}} + \frac{\la^{\eps_5}}{|\la-\la_1|^{\eps_2}}
		+ \frac{\la^{\eps_6}}{|\la-\la_2|^{\eps_3}} \biggr),		\\
	| \Pi(x,y) | \le \const(R) \text{  for  } y>x>R.		\label{eq:main-prop-cond1-2}
	\end{gather} 
	Then we have
	\begin{equation}\label{eq:main-prop-estim1}
	\int_{D_{>R}}|\varphi^{(R, T)}(x) - \varphi^{(R, T)}(y)|^2 \frac{\Pi(x,y)}{y-x} dxdy \xrightarrow[T\to \infty]{} 0.
	\end{equation}
	
\item \label{eq:var-estimate_x>R_eps_2=1}
If $\eps_2 = 1$ and $\psi(T) = o(\log(T))$ when $T\to\infty$, then \eqref{eq:main-prop-estim1} also holds.
	
\item \label{eq:var-estimate_x<R}
	Assume that there exist constants $\eps_7 > -1$, $\eps_8>0$, and, for any $R>0$, a constant $\const(R)>0$ such that the following holds:
	\begin{equation} \label{eq:main-prop-cond2}
	|\Pi(x, y) | \le \const(R) x^{\eps_7}  y^{-\eps_8} \text{  for  }x<R, y>R.		
	\end{equation} 
	Then we have
	\begin{equation}\label{eq:main-prop-estim2}
	\int_{D_{<R}}|\varphi^{(R, T)}(x) - \varphi^{(R, T)}(y)|^2 \frac{\Pi(x,y)}{y-x} dxdy \xrightarrow[T\to \infty]{} 0.
	\end{equation}
\end{enumerate}
\end{proposition}

\begin{proof}

We start with the proof of \ref{eq:var-estimate_x>R}. We split the domain $D_{>R}$ into three parts: 
\begin{multline*}
D_{>R} = \{x,y\in D: R\le x\le y<T\} \sqcup \{x,y\in D: R\le x< T\le y\}\ \sqcup  \\
	\sqcup \{x,y\in D: T\le x\le y\}.
\end{multline*}
Note that the integral is zero on $\{x,y\in D: T\le x\le y\}$. 

{\it The First Case}: $R\le x\le y<T$.

\begin{multline}\label{eq:prop:Pi/(x-y)_rig_step-1_1}
\frac{\mathrm {const}}{\log^2(T-R+1)} \left| \int\limits_R^T\int\limits_x^T
\biggl(\log(x-R+1) -\log(y-R+1)\biggr)^2 \frac{\Pi(x,y)}{x-y} dydx  \right|  \le \\
\frac{\const(R)}{(\log T)^{2}} \left| \int\limits_R^T\int\limits_1^{T/x} 
\log^2(\la)\frac{\bar\Pi(x, \la)}{\la-1} d\la dx \right|  +  \\
+ \frac{\const(R)}{(\log T)^{2}} \left| \int\limits_R^T\int\limits_1^{T/x} 
\log(\la) \biggl(\log(1-(R-1)x^{-1}) -\log(1-(R-1)(\la x)^{-1})\biggr) \frac{\bar\Pi(x, \la)}{\la-1} d\la dx \right|  +  \\
+ \frac{\const(R)}{(\log T)^{2}} \left| \int\limits_R^T\int\limits_x^T 
\biggl(\log(1-(R-1)x^{-1}) -\log(1-(R-1)(y)^{-1})\biggr)^2 \frac{\Pi(x, y)}{y-x} dy dx \right|.
\end{multline}
where we have used a simple estimate $\log(T-R+1) \ge \mathrm{const}\log(T)$ for $T$ sufficiently large.

For the first term we change the order of integration to obtain
\begin{multline}\label{eq:step-1_last}
\frac{\const(R)}{(\log T)^{2}} \left|\int\limits_R^T\int\limits_1^{T/x} 
\log^2(\la)\frac{\bar{\Pi}(x, \la)}{\la-1} d\la dx \right|   = 	\frac{\const(R)}{(\log T)^{2}} \left|
\int\limits_1^{T/R} \frac{\log^2(\la)}{\la-1} \biggl( \int\limits_R^{T/\la} \bar{\Pi}(x, \la) dx \biggr) d\la \right|  \le		\\
\frac{\const(R)\psi(T)}{(\log T)^{2}} \int\limits_1^{T/R} \frac{\log^2(\la)}{\la-1} \biggl(\frac{\la^{\eps_4}}{|\la-1|^{\eps_1}} + \frac{\la^{\eps_5}}{|\la-\la_1|^{\eps_2}}
		+ \frac{\la^{\eps_6}}{|\la-\la_2|^{\eps_3}} \biggr) d\la  \xrightarrow[T\to \infty]{} 0.
\end{multline}

We have $| \log(1-(R-1)y^{-1}) - \const(R)y^{-1}| \le \const(R)y^{-2}$, therefore for the second term in \eqref{eq:prop:Pi/(x-y)_rig_step-1_1} we obtain
\begin{multline*}
\frac{\const(R)}{(\log T)^{2}} \left| \int\limits_R^T\int\limits_1^{T/x} 
\log(\la) \biggl(\log(1-(R-1)x^{-1}) -\log(1-(R-1)(\la x)^{-1})\biggr) \frac{\bar\Pi(x, \la)}{\la-1} d\la dx \right|  \le  \\
\frac{\const(R)}{(\log T)^{2}} \int\limits_1^{T/R} \frac{\log(\la)}{\la-1} \biggl| \int\limits_R^{T/\la} \log(1-(R-1)x^{-1}) \bar{\Pi}(x, \la) dx \biggr| d\la   +  \\
\frac{\const(R)}{(\log T)^{2}} \int\limits_1^{T/R} \frac{\log(\la)}{\la(\la-1)} \biggl| \int\limits_R^{T/\la} \frac{\bar{\Pi}(x, \la)}x dx \biggr| d\la   +  
\frac{\const(R)}{(\log T)^{2}} \int\limits_1^{T/R} \int\limits_R^{T/\la} \frac{\log(\la)}{x^2\la^2(\la-1)} dx d\la  \le  \\
\frac{\const(R)\psi(T)}{(\log T)^{2}} \int\limits_1^{T/R} \frac{\log(\la)}{\la-1} \biggl(\frac{\la^{\eps_4}}{|\la-1|^{\eps_1}} + \frac{\la^{\eps_5}}{|\la-\la_1|^{\eps_2}}
		+ \frac{\la^{\eps_6}}{|\la-\la_2|^{\eps_3}} \biggr)  d\la + \frac{\const(R)}{(\log T)^{2}}  \xrightarrow[T\to \infty]{} 0,
\end{multline*}
where we have used~\eqref{eq:Abel-Dirichlet} and~\eqref{eq:main-prop-cond1} to estimate the first two terms, and~\eqref{eq:main-prop-cond1-2} to bound $|\Pi(x, y)|$ in the last term.

We have
$$
\Biggl| \log\biggl(\frac{1-(R-1)x^{-1}}{1-(R-1)y^{-1}}\biggr) \Biggr|  \le  \frac{\const(R)(y-x)}{xy},\ \text{ for }\ y>x>R,
$$
therefore for the last term in \eqref{eq:prop:Pi/(x-y)_rig_step-1_1} we obtain
\begin{multline*}
\frac{\const(R)}{(\log T)^{2}} \left| \int\limits_R^T\int\limits_x^T 
\biggl(\log(1-(R-1)x^{-1}) -\log(1-(R-1)(y)^{-1})\biggr)^2 \frac{\Pi(x, y)}{y-x} dy dx \right|  \le  \\
\le \frac{\const(R)}{(\log T)^{2}} \left| \int\limits_R^T\int\limits_x^T \frac{(y-x)|\Pi(x, y)|}{x^2y^2} dy dx \right| \xrightarrow[T\to \infty]{} 0,
\end{multline*}
where we have used~\eqref{eq:main-prop-cond1-2} once again.

{\it The Second Case}: $R\le x< T\le y$.

We need to estimate the integral 
\begin{multline*}
{\const(R)} \left\rvert \int\limits_T^{\infty} \int\limits_R^T \left(1-\frac{\log(x-R+1)}{\log(T-R+1)}\right)^2 
		 \frac{\Pi(x,y)}{y-x} dx dy \right\rvert \le \\
\le \frac{\const(R)}{\log^2(T)} \left\rvert \int\limits_1^{T/R} \left( \int\limits_{T/\la}^T \log^2\biggl(\frac{x-R+1}{T-R+1}\biggr)
		 \bar\Pi(x,\la) dx \right) \frac{d\la}{\la-1} \right\rvert + \\
		+ \frac{\const(R)}{\log^2(T)} \left\rvert \int\limits_{T/R}^{\infty} \left( \int\limits_R^T \log^2\biggl(\frac{x-R+1}{T-R+1}\biggr)
		 \bar\Pi(x,\la) dx \right) \frac{d\la}{\la-1} \right\rvert.
\end{multline*}
We use \eqref{eq:Abel-Dirichlet} to bound both summands from above:
\begin{multline*}
	\frac{\const(R)}{\log^2(T)} \left\rvert \int\limits_1^{T/R} \left( \int\limits_{T/\la}^T \log^2\biggl(\frac{x-R+1}{T-R+1}\biggr)
	 \bar\Pi(x,\la) dx \right) \frac{d\la}{\la-1} \right\rvert \le		\\
	\frac{\const(R)\psi(T)}{\log^2(T)} \left\rvert \int\limits_1^{T/R} \log^2\biggl(\frac{T/\la-R+1}{T-R+1}\biggr) 
			\cdot \biggl(\frac{\la^{\eps_4}}{|\la-1|^{\eps_1}} + \frac{\la^{\eps_5}}{|\la-\la_1|^{\eps_2}}
		+ \frac{\la^{\eps_6}}{|\la-\la_2|^{\eps_3}} \biggr) \frac{d\la}{\la-1} \right\rvert \le		\\
	\le	\frac{\const(R)\psi(T)}{\log^2(T)} \left\rvert \int\limits_1^{T/R} \log^2\la 
			\cdot \biggl(\frac{\la^{\eps_4}}{|\la-1|^{\eps_1}} + \frac{\la^{\eps_5}}{|\la-\la_1|^{\eps_2}}
		+ \frac{\la^{\eps_6}}{|\la-\la_2|^{\eps_3}} \biggr) \frac{d\la}{\la-1} \right\rvert 
	\xrightarrow[T\to \infty]{} 0.
\end{multline*}
and
\begin{multline*}
		\frac{\const(R)}{\log^2(T)} \left\rvert \int\limits_{T/R}^{\infty} \left( \int\limits_R^T \log^2\biggl(\frac{x-R+1}{T-R+1}\biggr)
		 \bar\Pi(x,\la) dx \right) \frac{d\la}{\la-1} \right\rvert \le		\\
		\frac{\const(R)\psi(T)}{\log^2(T)} \left\rvert \int\limits_{T/R}^{\infty} \log^2(T-R+1) \cdot \biggl(\frac{\la^{\eps_4}}{|\la-1|^{\eps_1}} + \frac{\la^{\eps_5}}{|\la-\la_1|^{\eps_2}}
		+ \frac{\la^{\eps_6}}{|\la-\la_2|^{\eps_3}} \biggr)  \frac{d\la}{\la-1} \right\rvert \xrightarrow[T\to \infty]{} 0.
\end{multline*}

If $\eps_2 = 1$ and $\psi(T) = o(\log(T))$ then we split the interval $[1,T/R]$ into two parts: $P_1=[1, \la_1-T^{-1}]\cup[\la_1+T^{-1},T/R]$ and $P_2=[\la_1-T^{-1}, \la_1+T^{-1}]$.
For the first part we have
\begin{equation}%\label{eq:case_ii-1}
\int_{P_1} \frac{\log^j(\la)}{\la-1} \biggl(\frac{\la^{\eps_4}}{|\la-1|^{\eps_1}} + \frac{\la^{\eps_5}}{|\la-\la_1|^{\eps_2}}
		+ \frac{\la^{\eps_6}}{|\la-\la_2|^{\eps_3}} \biggr) d\la  \le  \const\log(T),	\quad	j\in\Z_+,
\end{equation}
and for the second part we combine the estimate~\eqref{eq:main-prop-cond1-2} with an obvious estimate
\begin{equation}%\label{eq:case_ii-2}.
\int\limits_R^T\int_{P_2} F(x,\la) d\la dx  \le \const(R),
\end{equation}
where $F(x,\la)$ is bounded on $P_2$, $\vert F(x,\la) \vert  \le  \const(R) $. The cases~\ref{eq:var-estimate_x>R} and~\ref{eq:var-estimate_x>R_eps_2=1} are fully proved.

Now we prove \ref{eq:var-estimate_x<R}. We split the domain $D_{<R}$ into two parts:
\begin{gather*}
D_{<R} = \{x,y\in D: 0\le x < R\le y<T\} \sqcup \{x,y\in D: 0\le x<R;	T\le y\}.
\end{gather*}

{\it The Third Case}: $0\le x < R\le y<T$.

We should estimate the integral 
\begin{multline*}
\mathrm {const} \left| \int\limits_R^{\infty} \int\limits_0^R (\varphi^{(R,T)}(y)-1)^2 \frac{\Pi(x,y)}{y-x} dx dy \right|  \le %\\
\frac{\const(R)}{(\log T)^{2}} \int\limits_R^{T} \int\limits_0^R \log^2(y-R+1) \frac{|\Pi(x,y)|}{y-x} dx dy  	\le		\\
\end{multline*}
\begin{multline*}
\frac{\const(R)}{(\log T)^{2}} \int\limits_R^{T} \int\limits_0^R \log^2(y-R+1) \frac{ x^{\eps_7}  y^{-\eps_8} }{y-x} dx dy  \le \\
\frac{\const(R)}{(\log T)^{2}} \int\limits_R^{T} \frac{ \log^2(y-R+1)  y^{-\eps_8} }{y-R} dy   \le   \frac{\const(R)}{(\log T)^{2}}.
\end{multline*}
%\begin{multline*}
%\le \frac{\const(R)}{(\log T)^{2}} \left| \int\limits_0^{R/T} \frac{d\la}{1-\la} \biggl( \int\limits_R^T 					\log^2(y-R+1) \Pi(\la,y) dy \biggr) \right| +		\\
%+		\frac{\const(R)}{(\log T)^{2}} \left| \int\limits_{R/T}^1 \frac{d\la}{1-\la} \biggl( \int\limits_R^{R/\la} 	\log^2(y-R+1) \Pi(\la,y) dy \biggr) \right| \le	\\
%\le \frac{\const(R)}{(\log T)^{2}} \left| \int\limits_0^{R/T} \frac{\log^2(T-R+1)}{1-\la}  \biggl(\frac{C_1}{|\la-1|^{\eps_1}} + \frac{C_2}{|\la-\la_2|^{\eps_2}} \biggr) d\la \right| +	 	\\
%+		\frac{\const(R)}{(\log T)^{2}} \left| \int\limits_{R/T}^1 \frac{\log^2(R/\la-R+1)}{1-\la}  \biggl(\frac{C_1}{|\la-1|^{\eps_1}} + \frac{C_2}{|\la-\la_2|^{\eps_2}} \biggr) d\la \right| \le	\\ 
%\frac{\const(R)}{T} + \frac{\const(R)}{(\log T)^{2}}.
%\end{multline*}

{\it The Fourth Case}: $0\le x<R;	T\le y$.

We should consider the integral 
\begin{multline*}
\const(R) \left| \int\limits_T^\infty \int\limits_0^R \frac{\Pi(x,y)}{y-x} dx dy \right| 			\le
\const(R) \int\limits_T^\infty \int\limits_0^R \frac{ x^{\eps_7}  y^{-\eps_8} }{y-x} dx dy	\le \\
\const(R) \int\limits_T^\infty \frac{ y^{-\eps_8} }{y-R} dy \le \frac{\const(R)}{T^{\eps_8}}.
\end{multline*}

The proposition is proved completely.

\end{proof}

\begin{remark}
The assumptions of Proposition~\ref{prop:Pi/(x-y)_rig}~\ref{eq:var-estimate_x>R} obviously hold if 
\begin{align*}\label{eq:main-prop-cond1-simple}
\left| \bar{\Pi}(x, \la) \right| &\leq \frac{\const(R)}{x} \biggl(\frac{\la^{\eps_4}}{|\la-1|^{\eps_1}} + \frac{\la^{\eps_5}}{|\la-\la_1|^{\eps_2}}
		+ \frac{\la^{\eps_6}}{|\la-\la_2|^{\eps_3}} \biggr),		\\
	| \Pi(x,y) | &\le \const(R),	\quad y>x>R.
\end{align*}
We use this fact many times below.
\end{remark}

In some of the applications below the condition~\eqref{eq:main-prop-cond2} doesn't hold. In this case we will apply the following obvious corollary from the proof of the previous theorem.

\begin{corollary}\label{cor:no_y^-eps}
Let $\Pi(x,y)$ be a function on $D$. If we have
\begin{align*}
\frac1{(\log T)^{2}} \left| \int\limits_R^{T} \int\limits_0^R \log^2(y-R+1) \frac{\Pi(x,y)}{y-x} dx dy  \right|  & \xrightarrow[T\to \infty]{} 0,  \\
\left| \int\limits_T^\infty \int\limits_0^R \frac{\Pi(x,y)}{y-x} dx dy \right|  & \xrightarrow[T\to \infty]{} 0,
\end{align*}
then the convergence~\eqref{eq:main-prop-estim2} holds.
\end{corollary}

We use similar but simpler arguments to show the convergence to zero of the required integrals in case when the variables are split:

\begin{proposition}\label{prop:split-var_rig}
Let $\Pi(x,y)=\Pi_1(x) \cdot \Pi_2(y)$ be a function on $D$. 
\begin{enumerate}[label=(\roman*)]  %--  ?? turned off enumitem. Turn it on when there is a connection!!
\item
	Assume that there exists a positive function $\tilde\psi$, $\tilde\psi(T) = o(\log(T))$ when $T\to\infty$, 
	and, for any $R>0$, a constant $\const(R)>0$ such that for $m\in \{0,1,2\}$, we have
	\begin{gather*}
	\max_{R<a,b<T} \Biggl| \int_a^b \log^m (y-R+1) \Pi_2(y) dy\Biggr| \le \const(R) \tilde\psi^m(T),	\\
	\Biggl| \int_R^T \log^m(x-R+1) \Pi_1(x) dx\Biggr| \le  \const(R) \tilde\psi^m(T), \quad
	\Biggl| \int_T^\infty \Pi_2(y) dy \Biggr| \xrightarrow[T\to \infty]{} 0,		\\
	|\Pi_1(x)| \le \const(R), x \ge R,	\quad		|\Pi_2(y)| \le \const(R), y \ge R.
	\end{gather*} 
	Then we have
	\begin{equation*}\label{eq: var-estimate_split-vars-1}
	\int_{D_{>R}}|\varphi^{(R, T)}(x) - \varphi^{(R, T)}(y)|^2 \Pi(x,y) dxdy \xrightarrow[T\to \infty]{} 0.
	\end{equation*}
\item
	Assume additionally that $\Pi_1(x)$ is integrable on $[0,R]$ for any $R>0$.
	%there additionally exists a constant $\const(R)>0$ such that 
	%\begin{gather*}
	%\Biggl| \int_0^R \Pi_1(x) dx\Biggr|  \le  \const(R).
	%\end{gather*} 
	Then we have
	\begin{equation*}\label{eq: var-estimate_split-vars-2}
	\int_{D_{<R}}|\varphi^{(R, T)}(x) - \varphi^{(R, T)}(y)|^2 \Pi(x,y) dxdy \xrightarrow[T\to \infty]{} 0.
	\end{equation*}
\end{enumerate}
\end{proposition}

\begin{proof}
For $R\le x\le y<T$ we should estimate the integral
\begin{multline*}
\frac{\mathrm {const}}{\log^2(T-R+1)} \left| \int\limits_R^T\int\limits_x^T
\biggl(\log(x-R+1) -\log(y-R+1)\biggr)^2 \Pi(x,y) dydx  \right|  \le \\
\frac{\const(R)}{(\log T)^{2}} \sum_{m=0}^2 \Biggl| \int_R^T \log^m(x-R+1) \Pi_1(x) dx\Biggr| \times \\
\times \max_{R<a,b<T} \Biggl| \int_a^b \log^{2-m} (y-R+1) \Pi_2(y) dy\Biggr| \xrightarrow[T\to \infty]{} 0.
\end{multline*}

For $R\le x <T \le y$ we have
\begin{multline*}
\left| \int\limits_R^T\int\limits_T^\infty
\biggl(\frac{\log(x-R+1)}{\log(T-R+1)} - 1 \biggr)^2 \Pi(x,y) dydx  \right|  \le \\
2 \sum_{m=0}^2 \Biggl| \int_R^T \frac{\log^m(x-R+1)}{\log^m(T-R+1)} \Pi_1(x) dx\Biggr| \times 
\Biggl| \int_T^\infty \Pi_2(y) dy\Biggr| \xrightarrow[T\to \infty]{} 0.
\end{multline*}

For $x<R\le y<T$ we have
\begin{multline*}
\Biggl| \int_R^T \int_0^R|\varphi^{(R, T)}(x) - \varphi^{(R, T)}(y)|^2 \Pi(x,y) dx dy \Biggr|  =  \Biggl| \int_0^R \Pi_1(x) dx \Biggr| \times \\
\times  \Biggl| \int_R^T \frac{\log^2(y-R+1)}{\log^2 (T-R+1)} \Pi_2(y) dy \Biggr|\xrightarrow[T\to \infty]{} 0. 
\end{multline*}

And for $x<R$, $T<y$ we obtain an estimate
\begin{multline*}
\Biggl| \int_T^\infty \int_0^R|\varphi^{(R, T)}(x) - \varphi^{(R, T)}(y)|^2 \Pi(x,y) dx dy \Biggr| = \Biggl| \int_0^R \Pi_1(x) dx \Biggr| %\times \\
%\times  
\Biggl| \int_T^\infty \Pi_2(y) dy \Biggr|\xrightarrow[T\to \infty]{} 0. 
\end{multline*}
\end{proof}

\begin{corollary}\label{cor:split-var_rig-simple}
Assume that $\Pi_1(x) = f_1(x)g_1(x) + h_1(x)$, $\Pi_2(y) = f_2(y)g_2(y) + h_2(y)$, and there exist constants $R>0$ and $\eps_1 > -1, \eps_2 > 1$ such that
\begin{itemize}
\item $\max_{a,b>R} \Biggl| \int_a^b f_i(x) dx\Biggr| \le \const(R)$, 	for $i\in\{1,2\}$;
\item $g_i(x)\log^2(x-R+1)$ are decreasing to zero for $x$ sufficiently large, $i\in\{1,2\}$;
\item $|h_i(x)| \le \const(R) x^{-\eps_2}$ for $x > R$, $i\in\{1,2\}$;
\item $|\Pi_1(x)| \le \const(R), x \ge R$,	$|\Pi_2(y)| \le \const(R), y \ge R$;
\item $|\Pi_1(x)| \le \const(R)x^{\eps_1}$ for $x < R$.
\end{itemize}
Then we have
\begin{equation*}
\int_D |\varphi^{(R, T)}(x) - \varphi^{(R, T)}(y)|^2 \Pi_1(x)\Pi_2(y) dxdy \xrightarrow[T\to \infty]{} 0.
\end{equation*}
\end{corollary}

\begin{proof}
Let $g_2(y)\log^m(t-R+1)$ be decreasing for $y>R_1>R$, $m\in\{0,1,2\}$. We have
\begin{multline*}
\max_{R<a,b<T} \Biggl| \int_a^b \log^m (y-R+1) \Pi_2(y) dy\Biggr| \le  \\
\le  \const(R) + \max_{R_1<a,b<T} \Biggl| \int_a^b f_1(y) \left( g_2(y)\log^m (y-R+1) \right) dy\Biggr| +  \\
+  \const(R) \int_R^{\infty} \log^m (y-R+1) y^{-\eps_2} dy \le  \const(R) \biggl( 1 + |g_2(R_1)\log^m(R_1-R+1)| \biggr),
\end{multline*}
where we have used \eqref{eq:Abel-Dirichlet} for the second term. We have an identical estimate for $\Pi_1(x)$.

We also have
$$
\Biggl| \int_T^\infty \Pi_2(y) dy \Biggr|  \le  \const(R)|g_2(T)| + T^{1-\eps_2}  \xrightarrow[T\to \infty]{} 0,
$$
and
$$
\Biggl| \int_0^R \Pi_1(x) dx\Biggr|  \le  \const(R) \Biggl| \int_0^R x^{\eps_1} dx\Biggr|  \le   \const(R).
$$
Thus we can apply Proposition~\ref{prop:split-var_rig} to obtain the required estimate.
\end{proof}

From the proofs of Proposition~\ref{prop:split-var_rig} and Corollary~\ref{cor:split-var_rig-simple} we obtain also the corresponding one-dimensional result.
\begin{corollary}\label{cor:split-var_one-dim}
Assume that $\Pi_1(x) = f(x)g(x) + h(x)$ and there exist constants $R>0$ and $\eps > 1$ %, \eps_2 > 1$ 
such that
\begin{itemize}
\item $\max_{a,b>R} \Biggl| \int_a^b f_(x) dx\Biggr| \le \const(R)$;
\item $g(t)\log^2(t-R+1)$ is decreasing to zero for $t$ sufficiently large;
\item $|h(x)| \le \const(R) x^{-\eps}$ for $x \ge R$;
\item $|\Pi_1(x)| \le \const(R)$ for $x \ge R$.%;
%\item $|\Pi_1(x)| \le \const(R)x^{\eps_1}$ for $x < R$. ?? 
\end{itemize}
Then we have
\begin{equation*}
\int_0^\infty \biggl(|\varphi^{(R, T)}(x) |^2 - 1 \biggr) \Pi_1(x) dx \xrightarrow[T\to \infty]{} 0.
\end{equation*}
\end{corollary}

\subsection{Symplectic Bessel process}\label{sec:symplect-Bessel}

We will use the following estimate for the Bessel function for small $x$: 
\begin{equation}\label{eq:Bessel-small}
J_s(x) 		=	 \frac{(x/2)^s}{\Gamma(s+1)} + O(x^{s+2})
\end{equation}
(cf. e.g. 9.1.10 in Abramowitz and Stegun \cite{AS}) and the asymptotic expansion 
\begin{equation}\label{eq:Bessel-large}
J_s(x)																= 		\sqrt{\frac 2{\pi x}}\cos(x-s\pi/2-\pi/4)		+ O(x^{-3/2})
\end{equation}
of the Bessel function of a large argument (cf. e.g. 9.2.1 in Abramowitz and Stegun \cite{AS}). 
From the relation
\begin{equation}\label{eq:Bessel-der}
J'_s(x)		= \pm \frac sx J_s(x) \mp J_{s\pm 1}(x),  
\end{equation}
(cf. e.g. 9.1.27 in Abramowitz and Stegun \cite{AS}) we obtain
\begin{equation}\label{eq:Bessel-small-der}
J'_s(x) 		=	 \frac{s(x/2)^{s-1}}{2\Gamma(s+1)} + O(x^{s+2})
\end{equation}
for small $x$ and
\begin{equation}\label{eq:Bessel-large-der}
J'_s(x)															= - \sqrt{\frac 2{\pi x}}\sin(x-s\pi/2-\pi/4)		+ O(x^{-3/2}) 
\end{equation}
for $x \to \infty$. Also integrating the asymptotic expansion we have
\begin{equation}\label{eq:Bessel-large-int}
\int_{x}^\infty J_s(t)dt		=   \sqrt{\frac 2{\pi x}}\sin(x-s\pi/2-\pi/4)		+ O(x^{-3/2}).
\end{equation}

Forrester \cite[p.~312, (7.109)]{Forrester-log} gives the following definition of the hard edge scaling limit (scaling limit of the Laguerre symplectic ensemble in our case):
\begin{multline}\label{eq:Laguerre-kernel}
\tilde{S}_4(x,y) = \frac12\left(\frac{x}{y}\right)^{1/2}K^{(L)}_{2N}(x,y)  +  \\
  +  \frac{(2N)!y^{(s-1)/2}e^{-y/2}L^s_{2N}(y)}{4\Gamma(s+2N)}\int_x^\infty t^{(s-1)/2}e^{-t/2}L^s_{2N-1}(t)dt,
\end{multline}
\begin{gather*}
K^{hard\ edge}_s(X,Y)  =  \lim_{N\to\infty} \frac1{4N} S_4\left(\frac{X}{4N},\frac{Y}{4N}\right), \quad S_4(x,y) = 2\tilde{S}_4(2x,2y)\biggl\vert_{s \to 2s-1},  \\
\KK^{hard\ edge}_{4,s}(x,y) =
\begin{bmatrix}
			K_s^{hard\ edge}(x,y)																	& 	\int_y^x K_s^{hard\ edge}(x,t)dt	\\
\frac{\partial }{\partial x} K_s^{hard\ edge}(x,y)		&		K_s^{hard\ edge}(y,x)
\end{bmatrix},  
\end{gather*}
where $L^s_N(x)$ is the $N$-th Laguerre polynomial and $K^{(L)}_{2N}(x,y)$ is the corresponding Christoffel--Darboux kernel.

\begin{proposition}\label{prop:Bessel-4_formula}
The (hard edge) scaling limit of the Laguerre symplectic ensemble is defined by the following kernel:
\begin{gather*}
K^{Bessel}_s(x,y) =  2 \left(\frac{x}{y}\right)^{1/2} \KK^{Bessel}_{2, 2s-1}(4x, 4y) - \frac{J_{2s-1}(2y^{1/2})}{2y^{1/2}} \int_0^{x^{1/2}} J_{2s-1}(2t)dt,  \\
\KK^{Bessel}_{4,s}(x,y) =
\begin{bmatrix}
			K_s^{Bessel}(x,y)														& 	\int_y^x K_s^{Bessel}(x,t)dt	\\
\frac{\partial }{\partial x} K_s^{Bessel}(x,y)		&		K_s^{Bessel}(y,x)
\end{bmatrix}.  
\end{gather*}
\end{proposition}
\begin{proof}[First proof]
We have
$$
e^{-x/2}x^{s/2}L^s_N(x) \sim N^{s/2}J_s\left(2(Nx)^{1/2}\right) \text{  for  } N\to\infty
$$
and
\begin{equation}\label{eq:Laguerre-zero-int}
\int_0^\infty t^{(s-1)/2}e^{-t/2}L^s_{2N-1}(t)dt = 0,
\end{equation}
and the required formula follows.
\end{proof}
\begin{proof}[Second proof]
Forrester \cite[p.~312, (7.111)]{Forrester-log} gives
$$
K^{hard\ edge}_s(x,y) =  2\KK^{Bessel}_{2, 2s}(4x,4y) - \frac{J_{2s-1}(2y^{1/2})}{2y^{1/2}} \int_0^{x^{1/2}} J_{2s+1}(2t)dt.
$$
Therefore it is sufficient to check that
\begin{multline*}
 2\KK^{Bessel}_{2, 2s}(4x,4y) - 2\left(\frac{x}{y}\right)^{1/2} \KK^{Bessel}_{2, 2s-1}(4x, 4y)  =  \\
  = \frac{J_{2s-1}(2y^{1/2})}{2y^{1/2}} \left( \int_0^{x^{1/2}} J_{2s+1}(2t)dt - \int_0^{x^{1/2}} J_{2s-1}(2t)dt \right).
\end{multline*} 
From the relations~\eqref{eq:Bessel-der}, \eqref{eq:Bessel-large} we have
\begin{equation}\label{eq:Bessel-recurrence}
\frac {s}{x^{1/2}} J_s(2x^{1/2}) = J_{s-1}(2x^{1/2}) + J_{s+1}(2x^{1/2})
\end{equation}
and
$$
\int_0^{x^{1/2}} J_{2s+1}(2t)dt  - \int_0^{x^{1/2}} J_{2s-1}(2t)dt  =  - J_{2s}(2x^{1/2}).
$$
Regarding the left-hand side,
\begin{multline*}
2\KK^{Bessel}_{2, 2s}(4x,4y) - 2\left(\frac{x}{y}\right)^{1/2} \KK^{Bessel}_{2, 2s-1}(4x, 4y)  =  
\frac1{2(x-y)} \Biggl( - y^{1/2}J_{2s+1}(2y^{1/2})J_{2s}(2x^{1/2}) + \\
+ J_{2s}(2y^{1/2})\biggl(x^{1/2}J_{2s+1}(2x^{1/2}) + x^{1/2}J_{2s-1}(2x^{1/2}\biggr) -  xy^{-1/2}J_{2s}(2x^{1/2})J_{2s-1}(2y^{1/2}) ) \Biggr)  =  \\
	%=  \frac1{2(x-y)} \biggl( - x\frac{J_{2s-1}(y^{1/2})}{y^{1/2}}J_{2s}(x^{1/2}) + sJ_{2s}(2x^{1/2})J_{2s}(y^{1/2}) 
						%- y^{1/2}J_{2s+1}(2y^{1/2})J_{2s}(2x^{1/2}) \biggr) = \\
	= \frac{J_{2s}(x^{1/2})}{2(x-y)} \Biggl( - x\frac{J_{2s-1}(y^{1/2})}{y^{1/2}}  
						+  y^{1/2} \biggl( \frac{2sJ_{2s}(2y^{1/2})}{y^{1/2}} - J_{2s+1}(2y^{1/2}) \biggr)  \Biggr)  =  \\
	= -J_{2s}(2x^{1/2}) \frac{J_{2s-1}(2y^{1/2})}{2y^{1/2}},
\end{multline*}
where we have used \eqref{eq:Bessel-recurrence} several times.
\end{proof}

Unfortunately the condition~\eqref{eq:int_rho-trunc} doesn't hold for the kernel $\KK^{Bessel}_{4,s}(x,y)$. But a weaker condition does hold, and it will be sufficient for our purposes. 

\begin{proposition}\label{prop:int-rho-2_average-cond}
$$
\int_0^{\infty} \Biggl( \int_0^{\infty} \det \KK^{Bessel}_{4,s}(x,y) dy - K^{Bessel}_s(x,x) \Biggr) dx = 0.
$$
\end{proposition}
\begin{proof}
We plan to use Corollary~\ref{cor:int-defect}, and we will first simplify the expressions for the defect $\Def_{\KK^{Bessel}_{4,s}}(x)$ and for the limits
$$
- K_s^{Bessel}(y,x) \int_x^y K_s^{Bessel}(y,t)dt \biggl\vert_{y=0}^{\infty}.
$$
%We have
%\begin{multline*}
%\int_0^{\infty} \det \KK^{Bessel}_{4,s}(x,y) dy  =   \int_0^{\infty} K^{Bessel}_s(x,y) K^{Bessel}_s(y,x) dy  -  \\
%- \int_0^{\infty} \frac{\partial }{\partial x} K_s^{Bessel}(x,y) \int_y^x K_s^{Bessel}(x,t)dt		dy  =  \int_0^{\infty} K^{Bessel}_s(x,y) K^{Bessel}_s(y,x) dy  +  \\
%+ \int_0^{\infty} \frac{\partial }{\partial y} K_s^{Bessel}(y,x) \int_y^x K_s^{Bessel}(x,t)dt		dy  =  2\int_0^{\infty} K^{Bessel}_s(x,y) K^{Bessel}_s(y,x) dy  - \\
%-  K_s^{Bessel}(y,x) \int_x^y K_s^{Bessel}(y,t)dt \biggl\vert_{y=0}^{\infty},
%\end{multline*}
%where we have used the skew-symmetry of the kernel. 
First, we see that 
$$
4 \KK^{Bessel}_{2, 2s-1}(4x, 4y) = \frac14 \int_0^4 J_{2s-1}(\sqrt{ux}) J_{2s-1}(\sqrt{uy}) du
$$ 
is an orthogonal projection onto the subspace of functions $f(x)$ such that has its Hankel transform supported in $[0,4]$. We see from the orthogonal relations (\cite{AS}, 11.4.5) that 
$J_{2s-1}(2\sqrt{x}) \in \mathrm{Ran}(4 \KK^{Bessel}_{2, 2s-1}(4x, 4y))$, and it follows that 
\begin{gather*}
 \int_0^\infty  J_{2s-1}(2y^{1/2}) \cdot 4 \KK^{Bessel}_{2, 2s-1}(4x, 4y) dy = J_{2s-1}(2x^{1/2}).
\end{gather*}
Therefore we have
\begin{multline*}
\int_0^{\infty} K^{Bessel}_s(x,y) 4\left(\frac{y}{x}\right)^{1/2}\KK^{Bessel}_{2, 2s-1}(4x, 4y) dy =
\int_0^{\infty} \biggl( 2 \left(\frac{x}{y}\right)^{1/2} \KK^{Bessel}_{2, 2s-1}(4x, 4y) -  \\
			-  \frac{J_{2s-1}(2y^{1/2})}{2y^{1/2}} \int_0^{x^{1/2}} J_{2s-1}(2t)dt \biggr) 4\left(\frac{y}{x}\right)^{1/2} \KK^{Bessel}_{2, 2s-1}(4x, 4y) dy = 
\\ 
=\frac{1}{2}\int_0^{\infty}  4  \KK^{Bessel}_{2, 2s-1}(4x, 4y)  \cdot 4 \KK^{Bessel}_{2, 2s-1}(4x, 4y) dy  -  \\
			-   \frac{1}{2x^{1/2}} \int_0^{x^{1/2}} J_{2s-1}(2t)dt  \cdot \int_0^\infty  J_{2s-1}(2y^{1/2}) \cdot 4 \KK^{Bessel}_{2, 2s-1}(4x, 4y) dy
\\
=2  \KK^{Bessel}_{2, 2s-1}(4x, 4y) -\frac{1}{2x^{1/2}} \int_0^{x^{1/2}} J_{2s-1}(2t)dt  \cdot  J_{2s-1} (2 x^{1/2}) = K^{Bessel}_s(x,x)
\end{multline*}
and
$$
\int_0^{\infty} 4\left(\frac{x}{y}\right)^{1/2}\KK^{Bessel}_{2, 2s-1}(4x, 4y) \int_0^{y^{1/2}} J_{2s-1}(2t)dt dy = \int_0^{x^{1/2}} J_{2s-1}(2t)dt.
$$
We also have 
$$
2\int_0^{\infty} \frac{J_{2s-1}(2y^{1/2})}{2y^{1/2}} \int_0^{y^{1/2}} J_{2s-1}(2t)dt dy  =
\int_0^{\infty} d\biggl(\int_0^{y^{1/2}} J_{2s-1}(2t)dt \biggr)^2 = \frac14,
$$
where we have used that
\begin{equation}\label{eq:int-Bessel}
\int_0^\infty J_{\mu}(z) dz = 1
\end{equation}
for $\Re\mu>-1$ 

Thus
\begin{multline*}
2 \int_0^{\infty} K^{Bessel}_s(x,y) K^{Bessel}_s(y,x)  dy = \int_0^{\infty} K^{Bessel}_s(x,y) 4\left(\frac{y}{x}\right)^{1/2}\KK^{Bessel}_{2, 2s-1}(4x, 4y) dy  -  \\
-  \frac{J_{2s-1}(2x^{1/2})}{2x^{1/2}} \int_0^{\infty} 4\left(\frac{x}{y}\right)^{1/2} \KK^{Bessel}_{2, 2s-1}(4x, 4y) \int_0^{y^{1/2}} J_{2s-1}(2t)dt dy  +  \\
		+ \frac{J_{2s-1}(2x^{1/2})}{x^{1/2}} \int_0^{x^{1/2}} J_{2s-1}(2t)dt \int_0^{\infty} \frac{J_{2s-1}(2y^{1/2})}{2y^{1/2}} \int_0^{y^{1/2}} J_{2s-1}(2t)dt dy  =  \\
		K^{Bessel}_s(x,x)  - \frac38	\frac{J_{2s-1}(2x^{1/2})}{x^{1/2}} \int_0^{x^{1/2}} J_{2s-1}(2t)dt,
\end{multline*}
and 
$$
2\Def_{\KK^{Bessel}_{4,s}}(x) = - \frac38	\frac{J_{2s-1}(2x^{1/2})}{x^{1/2}} \int_0^{x^{1/2}} J_{2s-1}(2t)dt.
$$
Now since
$$
K_s^{Bessel}(0,x) = 0  \quad\text {and}\quad  \lim_{y\to\infty}K_s^{Bessel}(y,x) = -\frac{J_{2s-1}(2x^{1/2})}{4x^{1/2}},
$$
we have
\begin{multline*}
- K_s^{Bessel}(y,x) \int_x^y K_s^{Bessel}(y,t)dt \biggl\vert_{y=0}^{\infty} =  
\frac{J_{2s-1}(2x^{1/2})}{4x^{1/2}} \int_x^y 2 \biggl( \left(\frac{y}{t}\right)^{1/2} \KK^{Bessel}_{2, 2s-1}(4y, 4t) - \\
			 \frac{J_{2s-1}(2t^{1/2})}{2t^{1/2}} \int_0^{y^{1/2}} J_{2s-1}(2p)dp \biggr) dt \biggl\vert^{y=\infty}  =   
			\frac{J_{2s-1}(2x^{1/2})}{2x^{1/2}} \lim_{y\to\infty} \int_x^y \left(\frac{y}{t}\right)^{1/2} \KK^{Bessel}_{2, 2s-1}(4y, 4t) dt  -  \\
			-  \frac{J_{2s-1}(2x^{1/2})}{8x^{1/2}} \int_{x^{1/2}}^{\infty} J_{2s-1}(2p)dp.			
\end{multline*}
We write
\begin{multline*}
\int_x^y \left(\frac{y}{t}\right)^{1/2} \KK^{Bessel}_{2, 2s-1}(4y, 4t) dt = \int_0^{\infty} \left(\frac{y}{t}\right)^{1/2} \KK^{Bessel}_{2, 2s-1}(4y, 4t) dt - \\
\int_0^x \left(\frac{y}{t}\right)^{1/2} \KK^{Bessel}_{2, 2s-1}(4y, 4t) dt - \int_y^{\infty} \left(\frac{y}{t}\right)^{1/2} \KK^{Bessel}_{2, 2s-1}(4y, 4t) dt,
\end{multline*}
where we have
$$
\lim_{y\to\infty} \biggl| \int_0^x \left(\frac{y}{t}\right)^{1/2} \KK^{Bessel}_{2, 2s-1}(4y, 4t) dt \biggr|  \le  \lim_{y\to\infty} \frac{\const(x)y^{3/4}}{y-x} = 0,
$$
and we use Lemma~\ref{lem:sin-la-x-int-by-parts} to obtain
\begin{multline*}
\lim_{y\to\infty} \biggl| \int_y^{\infty} \left(\frac{y}{t}\right)^{1/2} \KK^{Bessel}_{2, 2s-1}(4y, 4t) dt \biggr|  =  \\
\mathrm {const}\lim_{y\to\infty} \biggl| \int_1^{\infty}  \frac{\cos(\sqrt{y}-s\pi-\pi/4) v^{-1/4}\cos(\sqrt{yv}-s\pi+\pi/4)}{v^{1/2}(v-1)}  -  \\
-  \frac{\cos(\sqrt{y}-s\pi+\pi/4)v^{1/4}\cos(\sqrt{yv}-s\pi-\pi/4)}{v^{1/2}(v-1)}dv \biggr|  =  0.
\end{multline*}
We also have
\begin{multline*}
\int_0^{\infty} \left(\frac{y}{t}\right)^{1/2} \KK^{Bessel}_{2, 2s-1}(4y, 4t) dt = \frac{\sqrt{y}}{4} \int_0^{\infty} \biggl( \int_0^1 J_{2s-1}(2\sqrt{uy}) 
J_{2s-1}(2\sqrt{ut}) du\biggr) \frac{dt}{\sqrt{t}}  =  \\
\frac{\sqrt{y}}{2} \int_0^1 J_{2s-1}(2\sqrt{uy}) \int_0^{\infty} J_{2s-1}(2\sqrt{u}p) dp du =  \frac{\sqrt{y}}{2} \int_0^1 \frac{J_{2s-1}(2\sqrt{uy})}{2\sqrt{u}}du =
\frac12 \int_0^{y^{1/2}} J_{2s-1}(2p)dp,
\end{multline*}
therefore
$$
\lim_{y\to\infty} \int_x^y \left(\frac{y}{t}\right)^{1/2} \KK^{Bessel}_{2, 2s-1}(4y, 4t) dt = \frac12 \lim_{y\to\infty} \int_0^{y^{1/2}} J_{2s-1}(2p)dp = \frac14,
$$
and
$$
- K_s^{Bessel}(y,x) \int_x^y K_s^{Bessel}(y,t)dt \biggl\vert_{y=0}^{\infty} = 
			\frac{J_{2s-1}(2x^{1/2})}{8x^{1/2}}  -  \frac{J_{2s-1}(2x^{1/2})}{8x^{1/2}} \int_{x^{1/2}}^{\infty} J_{2s-1}(2p)dp.			
$$
Write
\begin{multline*}
\int_0^{\infty} \det \KK^{Bessel}_{4,s}(x,y) dy - K^{Bessel}_s(x,x)  =  \\
2 \Def_{\KK^{Bessel}_{4,s}}(x) 	-  K^{Bessel}_s(y,x) \int_x^y K^{Bessel}_s(y,t)dt \biggl\vert_{y=0}^{\infty}  =  \\
=  \frac{J_{2s-1}(2x^{1/2})}{8x^{1/2}} \biggl( -3\int_0^{x^{1/2}} J_{2s-1}(2t)dt  + 1 -  \int_{x^{1/2}}^{\infty} J_{2s-1}(2p)dp \biggr)  =  \\
=  \frac{J_{2s-1}(2x^{1/2})}{16x^{1/2}} - \frac{J_{2s-1}(2x^{1/2})}{4x^{1/2}} \int_0^{x^{1/2}} J_{2s-1}(2t)dt.
\end{multline*}
We directly see that
\begin{equation}\label{no-scr-4}
 \frac{J_{2s-1}(2x^{1/2})}{16x^{1/2}} - \frac{J_{2s-1}(2x^{1/2})}{4x^{1/2}} \int_0^{x^{1/2}} J_{2s-1}(2t)dt\neq 0
 \end{equation}
 and therefore the relation \eqref{eq:int_rho-trunc} does not hold for the kernel $\KK^{Bessel}_{4,s}$. Nonetheless, we have
 
  \begin{multline*}
\int_0^{\infty} \biggl( \frac{J_{2s-1}(2x^{1/2})}{16x^{1/2}} - \frac{J_{2s-1}(2x^{1/2})}{4x^{1/2}} \int_0^{x^{1/2}} J_{2s-1}(2t)dt \biggr) dx  =  \\
= \frac1{16}\int_0^{\infty} \frac{J_{2s-1}(2x^{1/2})}{x^{1/2}} dx -  \frac14 \biggl( \int_0^{x^{1/2}} J_{2s-1}(2t)dt \biggr)^2 \biggl\vert_0^{\infty} =
\frac1{16} - \frac1{16} = 0.
\end{multline*}
\end{proof}

\begin{remark}
The assumptions of Proposition~\ref{prop:proj_K-1_K-4} do hold for the Pfaffian Laguerre kernel $\tilde{S}_4(x,y)$, for the finite $N$ (see the definition of the kernel in \eqref{eq:Laguerre-kernel}), and therefore the condition~\eqref{eq:int_rho-trunc} holds also. In this case, first of all, 
$$
2\tilde{S}_4(x,y) = \left(\frac{x}{y}\right)^{1/2}K^{(L)}_{2N}(x,y) + f_N(x)g_N(y)
$$ 
has a reproducing property: $\Pi(x,y) = \left(\frac{x}{y}\right)^{1/2}K^{(L)}_{2N}(x,y)$ is a projection because $K^{(L)}_{2N}(x,y)$ is a Christoffel--Darboux kernel, $f_N(x) = \int_x^\infty t^{(s-1)/2}e^{-t/2}L^s_{2N-1}(t)dt$ lies in the image of $\Pi(x,y)$ and $g_N(y) = \frac{(2N)!}{2\Gamma(s+2N)} y^{(s-1)/2}e^{-y/2}L^s_{2N}(y)$ is orthogonal to the image of $\Pi(x,y)$. And we also have
$$
\tilde{S}_4(0,x) = 0,\quad \lim_{y\to \infty} \tilde{S}_4(y,x) = 0,  
$$
because the same is true for the kernel $\left(\frac{y}{x}\right)^{1/2}K^{(L)}_{2N}(y,x)$ and because the integral~\eqref{eq:Laguerre-zero-int} is zero.

Neither property holds when we consider the limiting kernel $2K^{Bessel}_s(x,y)$. First, there is no reproducing property: scaling limits for $L^s_{2N-1}(y)$ and $L^s_{2N}(y)$ are the same, therefore the scaling limit of $g_N(y)$ is not orthogonal to the image of the limiting projection. And, second, as we see from \eqref{eq:int-Bessel}, the integral~\eqref{eq:Laguerre-zero-int} is not zero after the scaling limit, and $\lim_{y\to \infty} K^{Bessel}_s(y,x) \neq 0$.
\end{remark}

%We will need the following lemma.
\begin{lemma}\label{lem: J_t J_u J_v}
\begin{enumerate}[label=(\roman*)]  
\item\label{lem: J_t J_u J_v-1} 
The convergence to zero of the integrals~\eqref{eq:main-prop-estim1}, \eqref{eq:main-prop-estim2} holds for
$$
\Pi(x,y) = \biggl(\frac yx\biggr)^{\eps_1/2}  J_t(2\sqrt{y}) J_{t+\eps_1}(2\sqrt{y}) J_v(2\sqrt{x}),
$$
where $\eps_1 \in \{0,1\}$ and $t>-1$, $v>-1$.
\item\label{lem: J_t J_u J_v-2}
The convergence to zero of the integrals~\eqref{eq:main-prop-estim1}, \eqref{eq:main-prop-estim2} holds for
$$
\Pi(x,y) = \biggl(\frac yx\biggr)^{\eps_1/2-1/2}  J_t(2\sqrt{y}) J_{t+\eps_1}(2\sqrt{x}) J_v(2\sqrt{x}),
$$
where $\eps_1 \in \{0,1\}$ and $t>-1$, $v>-1$.
\end{enumerate}
\end{lemma}
\begin{proof}

\ref{lem: J_t J_u J_v-1}. We set $\la = y/x$, we write $\bar\Pi(x, \la) = \Pi(x, \la x)$ as a sum
\begin{multline}\label{eq: J_t J_u J_v-1_3-terms}
\la^{\eps_1/2-1/2}  J_t(2\sqrt{\la x}) J_{t+\eps_1}(2\sqrt{x}) J_v(2\sqrt{x})  =  \biggl( \la^{\eps_1/2}  J_t(2\sqrt{\la}\sqrt{x}) J_{t+\eps_1}(2\sqrt{\la}\sqrt{x}) J_v(2\sqrt{x})  -  \\
  -  \frac{\la^{\eps_1/2-1/2}}{\pi\sqrt{x}} \cos(2\sqrt{\la}\sqrt{x} - t\pi/2-\pi/4)	\cos(2\sqrt{\la}\sqrt{x} - (t+\eps_1)\pi/2-\pi/4)	J_v(2\sqrt{x}) \biggr)  +  \\
	+  \Biggl( \frac{\la^{\eps_1/2-1/2}}{\pi\sqrt{x}} \cos(2\sqrt{\la}\sqrt{x} - t\pi/2-\pi/4)	\cos(2\sqrt{\la}\sqrt{x} - (t+\eps_1)\pi/2-\pi/4)	\times  \\
 \times \biggl( J_v(2\sqrt{x}) -  \frac1{\sqrt{\pi}x^{1/4}} \cos(2\sqrt{x} - v\pi/2-\pi/4) \biggr) \Biggr)  +  \\
+ \frac{\la^{\eps_1/2-1/2}}{\pi\sqrt{x}} \cos(2\sqrt{\la}\sqrt{x} - t\pi/2-\pi/4)	\cos(2\sqrt{\la}\sqrt{x} - (t+\eps_1)\pi/2-\pi/4)	\times  \\
 \times \frac1{\sqrt{\pi}x^{1/4}} \cos(2\sqrt{x} - v\pi/2-\pi/4),
\end{multline}
and we check the convergence to zero of the integrals~\eqref{eq:main-prop-estim1} and \eqref{eq:main-prop-estim2} term by term.

We fix arbitrary $R>0$ and for the first term in~\eqref{eq: J_t J_u J_v-1_3-terms} we obtain
\begin{multline}\label{eq: J_t J_u J_v_1st-term}
\biggl| \la^{\eps_1/2}  J_t(2\sqrt{\la}\sqrt{x}) J_{t+\eps_1}(2\sqrt{\la}\sqrt{x}) J_v(2\sqrt{x})  -  \\
  -  \frac{\la^{\eps_1/2-1/2}}{\pi\sqrt{x}} \cos(2\sqrt{\la}\sqrt{x} - t\pi/2-\pi/4)	\cos(2\sqrt{\la}\sqrt{x} - (t+\eps_1)\pi/2-\pi/4)	J_v(2\sqrt{x}) \biggr|  \le  \\
	\le \frac{\const(R)J_v(2\sqrt{x})}{\sqrt{\la}x},
\end{multline}
for $y = \la x\ge R$, where we have used the estimate~\eqref{eq:Bessel-large} and the fact that $\la\ge 1$. Now we can combine this estimate with~\eqref{eq:Bessel-small} to see
that the assumptions of Proposition~~\ref{prop:Pi/(x-y)_rig} are satisfied for the difference~\eqref{eq: J_t J_u J_v_1st-term}.

As a next step we need to estimate the intergal 
\begin{multline*}
\Biggl| \int_a^b \frac{\la^{\eps_1/2-1/2}}{\pi\sqrt{x}} \cos(2\sqrt{\la}\sqrt{x} - t\pi/2-\pi/4)	\cos(2\sqrt{\la}\sqrt{x} - (t+\eps_1)\pi/2-\pi/4)	\times  \\
 \times \biggl( J_v(2\sqrt{x}) -  \frac1{\sqrt{\pi}x^{1/4}} \cos(2\sqrt{x} - v\pi/2-\pi/4) \biggr) dx \Biggl|  \le  \\
\frac{\la^{\eps_1/2-1/2}}{2\pi} \Biggl| \int_a^b \frac1{\sqrt{x}} \cos(\eps_1\pi/2)	\biggl( J_v(2\sqrt{x}) -  \frac1{\sqrt{\pi}x^{1/4}} \cos(2\sqrt{x} - v\pi/2-\pi/4) \biggr) dx \Biggl|  +  \\
 +  \frac{\la^{\eps_1/2-1/2}}{2\pi} \Biggl| \int_a^b \frac1{\sqrt{x}} \cos(4\sqrt{\la}\sqrt{x} - (2t+\eps_1+1)\pi/2) \times  \\
\times \biggl( J_v(2\sqrt{x}) -  \frac1{\sqrt{\pi}x^{1/4}} \cos(2\sqrt{x} - v\pi/2-\pi/4) \biggr) dx \Biggl|  \le  \frac{\const(R)}{\sqrt{\la}},
\end{multline*}
where we have used estimates~\eqref{eq:Bessel-large} and \eqref{eq:Bessel-large-der} and also~\eqref{eq:Abel-Dirichlet} for non-zero first term, 
and Lemma~\ref{lem:sin-la-x-int-by-parts} for the second term.

Next, for $a, b > R$ for the main part we can write 
\begin{multline*}
\Biggl| \int_a^b \frac{\mathrm{const}\la^{\eps_1/2-1/2}}{x^{3/4}} \cos(2\sqrt{\la}\sqrt{x} - t\pi/2-\pi/4)	\cos(2\sqrt{\la}\sqrt{x} - (t+\eps_1)\pi/2-\pi/4)	\times  \\
 \times \cos(2\sqrt{x} - v\pi/2-\pi/4) \biggr) dx \Biggl|   \le   \frac{\mathrm{const}}{\sqrt{\la}}  + \Biggl| \int_a^b \frac{\mathrm{const}\la^{\eps_1/2-1/2}}{x^{3/4}} \times  \\
\times \cos(4\sqrt{\la}\sqrt{x} - (2t+\eps_1+1)\pi/2) \cos(2\sqrt{x} - v\pi/2-\pi/4) dx \Biggl|  \le  \\
\le  \const(R) \biggl( \frac1{\sqrt{\la}} +  \frac1{2\sqrt{\la}+1} + \frac1{2\sqrt{\la}- 1} \biggr),
\end{multline*}
where we have used~\eqref{eq:Abel-Dirichlet} once again. Therefore we can apply Proposition~\ref{prop:Pi/(x-y)_rig}\ref{eq:var-estimate_x>R} in this case.

If $\eps_1 = 0$ then Proposition~\ref{prop:Pi/(x-y)_rig}\ref{eq:var-estimate_x<R} obviously holds, therefore we can now put $\eps_1 = 1$. We have
\begin{multline*}
 \frac{\const(R)}{(\log T)^{2}} \left| \int\limits_0^R \int\limits_R^{T} \log^2(y-R+1) \frac{\cos(4\sqrt{y} - t\pi) J_v(2\sqrt{x})}
{\sqrt{x}(y-x)} dy dx  \right| \le  \\
\le \frac{\const(R)}{(\log T)^{2}} \int\limits_0^R x^{(v-1)/2} dx \Biggl( \int\limits_R^{R'} \frac{\log^2(y-R+1)}{y-R} dy +
\frac{\sqrt{R'} \log^2(R'-R+1)}{R'-R} \Biggr) \le \frac{\const(R)}{(\log T)^{2}},
\end{multline*}
where the function $\sqrt{y}\log^2(y-R+1)/(y-R)$ is decreasing for $y\ge R'$. And
$$
\left| \int\limits_0^R \int\limits_T^\infty \frac{\cos(4\sqrt{y} - t\pi) J_v(2\sqrt{x})}{\sqrt{x}(y-x)} dy dx  \right| \le  
\frac{\const(R)\sqrt{T}}{T-R} \int\limits_0^R x^{(v-1)/2} dx \xrightarrow[T\to \infty]{} 0,
$$
thus we can use Corollary~\ref{cor:no_y^-eps}.

\ref{lem: J_t J_u J_v-2}. We fix again arbitrary $R>0$ and write
\begin{multline*}
\biggl| \la^{\eps_1/2-1/2}  J_t(2\sqrt{\la}\sqrt{x}) J_{t+\eps_1}(2\sqrt{x}) J_t(2\sqrt{x})  -  \\
  -  \frac{\la^{\eps_1/2-3/4}}{\sqrt{\pi}{x}^{1/4}} \cos(2\sqrt{\la}\sqrt{x} - t\pi/2-\pi/4)	J_{t+\eps_1}(2\sqrt{x})	J_v(2\sqrt{x}) \biggr|  \le  \\
	\le \frac{\const(R) J_{t+\eps_1}(2\sqrt{x}) J_v(2\sqrt{x})}{{\la}^{1/4}x^{3/4}},
\end{multline*}
for $y = \la x\ge R$, where we have used the estimate~\eqref{eq:Bessel-large} and the fact that $\la\ge 1$. Now we can combine this estimate with~\eqref{eq:Bessel-small} to see
that the assumptions of Proposition~~\ref{prop:Pi/(x-y)_rig} are satisfied for the considered difference.

As a next step we need to estimate the intergal 
\begin{multline*}
\Biggl| \int_a^b \frac{\la^{\eps_1/2-3/4}}{\sqrt{\pi}\sqrt{x}} \cos(2\sqrt{\la}\sqrt{x} - t\pi/2-\pi/4)	 \biggl( {x}^{1/4} J_{t+\eps_1}(2\sqrt{x})	J_v(2\sqrt{x})  -  \\
  -  \frac1{{\pi}{x}^{1/4}} \cos(2\sqrt{x} - (t+\eps_1)\pi/2-\pi/4) \cos(2\sqrt{x} - v\pi/2-\pi/4) \biggr) dx \Biggl|  \le  \frac{\const(R)}{\la^{3/4}},
\end{multline*}
where we have used estimates~\eqref{eq:Bessel-large} and \eqref{eq:Bessel-large-der} and then Lemma~\ref{lem:sin-la-x-int-by-parts}.

Next, for $a, b > R$ for the main part we can write 
\begin{multline*}
\Biggl| \int_a^b \frac{\mathrm{const}\la^{\eps_1/2-3/4}}{x^{3/4}} \cos(2\sqrt{\la}\sqrt{x} - t\pi/2-\pi/4)	\cos(2\sqrt{x} - (t+\eps_1)\pi/2-\pi/4)	\times  \\
 \times \cos(2\sqrt{x} - v\pi/2-\pi/4) \biggr) dx \Biggl|   =   \Biggl| \int_a^b \frac{\mathrm{const}\la^{\eps_1/2-3/4}}{x^{3/4}} \cos(2\sqrt{\la}\sqrt{x} - t\pi/2-\pi/4) \times  \\
\times \biggl( \cos(4\sqrt{x} - (t+v+\eps_1+1)\pi/2) + \cos((t-v+\eps_1)\pi/2) \biggr) dx \Biggl|  \le  \\
\le  \frac{\const(R)}{\la^{1/4}}\biggl( \frac1{\sqrt{\la}} +  \frac1{\sqrt{\la}+2} + \frac1{\sqrt{\la}- 2} \biggr),
\end{multline*}
where we have used~\eqref{eq:Abel-Dirichlet} for the last estimate. Therefore we can apply Proposition~\ref{prop:Pi/(x-y)_rig}\ref{eq:var-estimate_x>R} 
and~\ref{eq:var-estimate_x>R_eps_2=1} in this case.

Finally for $x<R$, $y>R$ we have 
$$
\biggl| \frac{\la^{\eps_1/2-3/4}}{\sqrt{\pi}{x}^{1/4}} \cos(2\sqrt{\la}\sqrt{x} - t\pi/2-\pi/4)	J_{t+\eps_1}(2\sqrt{x})	J_v(2\sqrt{x}) \biggl|  \le
\frac{\const(R)x^{(t+v+\eps_1)/2}}{y^{1/4}},
$$
therefore the assumptions of Proposition~\ref{prop:Pi/(x-y)_rig}\ref{eq:var-estimate_x<R} hold for this term and Lemma is fully proved.
\end{proof}

\begin{proof}[Proof of the Theorem~\ref{thm:Bessel-rigidity}\ref{thm-item:symplect-Bessel-rigidity}]

Our plan is to expand the determinant,
$$
\det \KK^{Bessel}_{4,s}(x,y)  =  K^{Bessel}_s(x,y) K^{Bessel}_s(y,x)  - \frac{\partial }{\partial x} K_s^{Bessel}(x,y) \int_y^x K_s^{Bessel}(x,t)dt	
$$
and then to show that all the summands in the formula~\eqref{eq:Pfaff-var_general} for the variance, with $f(x) = \varphi^{(R, T)}(x)$, tend to zero term by term.

$$
\text{\bf The first part,\ } \int_{\RR_+}|\varphi^{(R, T)}(x)|^2 \biggl( K^{Bessel}_s(x,x) - \int_{\RR} \det \KK^{Bessel}_{4, s}(x,y) dy\biggr) dx.
$$
We set
\begin{multline*}
\Pi(x) = K^{Bessel}_s(x,x) - \int_0^{\infty} \det \KK^{Bessel}_{4, s}(x,y) dy  =  \\
  =  \frac{J_{2s-1}(2x^{1/2})}{4x^{1/2}} \int_0^{x^{1/2}} J_{2s-1}(2t)dt  -  \frac{J_{2s-1}(2x^{1/2})}{16x^{1/2}}  =  \\
  =  \frac{J_{2s-1}(2x^{1/2})}{16x^{1/2}}   -  \frac{J_{2s-1}(2x^{1/2})}{4x^{1/2}} \int_{x^{1/2}}^\infty J_{2s-1}(2t)dt,
\end{multline*}
we use Proposition~\ref{prop:int-rho-2_average-cond} to write
$$
\int_0^{\infty}|\varphi^{(R, T)}(x)|^2 \Pi(x) dx  =  \int_0^{\infty} \biggl(|\varphi^{(R, T)}(x)|^2 - 1\biggl) \Pi(x) dx 
$$
and then we use estimates \eqref{eq:Bessel-large} and \eqref{eq:Bessel-large-int} and Corollary~\ref{cor:split-var_one-dim} to see that
$$
\Biggl| \int_0^{\infty}\biggl(|\varphi^{(R, T)}(x)|^2 - 1\biggl) \Pi(x) dx \biggl| \xrightarrow[T\to \infty]{} 0.
$$

$$
\text{\bf The second part,\ } \int_{D} |\varphi^{(R, T)}(x)-\varphi^{(R, T)}(y)|^2 K^{Bessel}_s(x,y) \cdot K^{Bessel}_s(y,x) dx dy.
$$
We have
\begin{multline*}
K^{Bessel}_s(x,y) \cdot K^{Bessel}_s(y,x) = \\
\Biggl( 2 \left(\frac{x}{y}\right)^{1/2} \KK^{Bessel}_{2, 2s-1}(4x, 4y) + \frac{J_{2s-1}(2y^{1/2})}{2y^{1/2}} \biggl( -\frac12 + \int_{x^{1/2}}^\infty J_{2s-1}(2t)dt \biggr) \Biggr) \times		\\
\times  \Biggl( 2 \left(\frac{y}{x}\right)^{1/2} \KK^{Bessel}_{2, 2s-1}(4y, 4x) + \frac{J_{2s-1}(2x^{1/2})}{2x^{1/2}} \biggl( -\frac12 + \int_{y^{1/2}}^\infty J_{2s-1}(2t)dt \biggr) \Biggr)  =  \\
4\left( \KK^{Bessel}_{2, 2s-1}(4x, 4y) \right)^2 	+  \Biggl( 2 \left(\frac{x}{y}\right)^{1/2} \KK^{Bessel}_{2, 2s-1}(4x, 4y) \frac{J_{2s-1}(2x^{1/2})}{2x^{1/2}} \times  \\
\times \int_{y^{1/2}}^\infty J_{2s-1}(2t)dt  +  \frac{J_{2s}(2x^{1/2})J^2_{2s-1}(2y^{1/2})}{4(x-y)} \int_{x^{1/2}}^\infty J_{2s-1}(2t) dt \Biggr)  + \\
+ \frac{y^{1/2}x^{-1/2}J_{2s}(2y^{1/2})J_{2s-1}(2x^{1/2})J_{2s-1}(2y^{1/2})}{4(x-y)} \int_{x^{1/2}}^\infty J_{2s-1}(2t) dt - \\
- \Biggl( \left(\frac{x}{y}\right)^{1/2} \KK^{Bessel}_{2, 2s-1}(4x, 4y) \frac{J_{2s-1}(2x^{1/2})}{2x^{1/2}} + 
							\left(\frac{y}{x}\right)^{1/2} \KK^{Bessel}_{2, 2s-1}(4y, 4x) \frac{J_{2s-1}(2y^{1/2})}{2y^{1/2}} \Biggr)  +  \\
+  \frac{J_{2s-1}(2x^{1/2})}{2x^{1/2}} \int_0^{x^{1/2}} J_{2s-1}(2t)dt \frac{J_{2s-1}(2y^{1/2})}{2y^{1/2}} \int_0^{y^{1/2}} J_{2s-1}(2t)dt =: \\
=: \frac{S_1(x,y)}{(x-y)^2} + \frac{S_2(x,y)}{x-y} + \frac{S_3(x,y)}{x-y} + \frac{S_4(x,y)}{x-y} + S_5(x,y).
\end{multline*}

The integral for the first term, $S_1(x,y)/(x-y)^2 = 4 \biggl( \KK^{Bessel}_{2, 2s-1}(4x, 4y) \biggr)^2$, was estimated in \cite{AB_det-rigidity} by Proposition~\ref{prop:rig-det}.

We use \eqref{eq:Bessel-large} and \eqref{eq:Bessel-large-int} to obtain
\begin{multline*}
\biggl|S_2(x,y)\biggr| = \Biggl| \left( x^{1/2}J_{2s}(2x^{1/2})J_{2s-1}(2y^{1/2}) - y^{1/2}J_{2s}(2y^{1/2})J_{2s-1}(2x^{1/2}) \right)	\times		\\
\frac{J_{2s-1}(2x^{1/2})}{4y^{1/2}} \int_{y^{1/2}}^\infty J_{2s-1}(2t)dt + 
\frac{J_{2s}(2x^{1/2})J^2_{2s-1}(2y^{1/2})}{4} \int_{x^{1/2}}^\infty J_{2s-1}(2t)dt \Biggr| \le \mathrm {const} \frac1{\la^{1/2}x},
\end{multline*}
where we have set, as usual, $\la = y/x$. Thus \ref{eq:main-prop-cond1} is satisfied for $x,y >R$ and any $R > 0$. Also we easily obtain
$$
\biggl|S_2(x,y)\biggr| \le \mathrm{const(R,s)} y^{-1/2} \text{  for  } x<R, y>R,
$$
and we can apply Proposition~\ref{prop:Pi/(x-y)_rig}.

We also have
\begin{multline*}
\Biggl| S_3(x,y)  -  \frac1{2\pi\sqrt{x}}\cos(2y^{1/2}-s\pi-\pi/4)\cos(2y^{1/2}-(2s-1)\pi/2-\pi/4)J_{2s-1}(2x^{1/2}) \times \\
\times \int_{x^{1/2}}^\infty J_{2s-1}(2t)dt \Biggr|  \le  \mathrm {const} \frac{\left|J_{2s-1}(2x^{1/2})\int_{x^{1/2}}^\infty J_{2s-1}(2t)dt\right|}{\sqrt{xy}},
\end{multline*}
thus we can combine this estimate with \eqref{eq:Bessel-small}, \eqref{eq:Bessel-large}, \eqref{eq:Bessel-large-int} to use Proposition~\ref{prop:Pi/(x-y)_rig} in this case. 

For the main term of $S_3(x,y)$ we have
\begin{multline*}
\Biggl| \int_a^b  \frac1{2\pi\sqrt{x}}\cos(2y^{1/2}-s\pi-\pi/4)\cos(2y^{1/2}-(2s-1)\pi/2-\pi/4)J_{2s-1}(2x^{1/2}) \times \\
\times\int_{x^{1/2}}^\infty J_{2s-1}(2t)dt dx \Biggr|  = 
\Biggl| \int_a^b  \frac1{4\pi\sqrt{x}}\cos(4\la^{1/2}x^{1/2}-2s\pi) J_{2s-1}(2x^{1/2}) \times \\
\times \int_{x^{1/2}}^\infty J_{2s-1}(2t)dt dx \Biggr| \le \frac{\mathrm {const}\log(b)}{\sqrt{\la}}
\end{multline*}
by Lemma~\ref{lem:sin-la-x-int-by-parts}, and we can use Proposition~\ref{prop:Pi/(x-y)_rig}\ref{eq:var-estimate_x>R}. We also have 
\begin{multline*}
 \frac{\const(R)}{(\log T)^{2}} \left| \int\limits_0^R \int\limits_R^{T} \log^2(y-R+1) \frac{\cos(4y^{1/2} - 2s\pi) J_{2s-1}(2x^{1/2}) 
\int_{x^{1/2}}^\infty J_{2s-1}(2t)dt}{x^{1/2}(y-x)} dy dx  \right| \le  \\
\le \frac{\const(R)}{(\log T)^{2}} \int\limits_0^R x^{s-1} dx \Biggl( \int\limits_R^{R'} \frac{\log^2(y-R+1)}{y-R} dy +
\frac{\sqrt{R'} \log^2(R'-R+1)}{R'-R} \Biggr) \le \frac{\const(R)}{(\log T)^{2}},
\end{multline*}
where the function $\sqrt{y}\log^2(y-R+1)/(y-R)$ is decreasing for $y\ge R'$. And
\begin{multline*}
\left| \int\limits_0^R \int\limits_T^\infty \frac{\cos(4y^{1/2} - 2s\pi) J_{2s-1}(2x^{1/2}) 
\int_{x^{1/2}}^\infty J_{2s-1}(2t)dt}{x^{1/2}(y-x)} dy dx  \right| \le  \\
\frac{\const(R)\sqrt{T}}{T-R} \int\limits_0^R x^{s-1} dx \xrightarrow[T\to \infty]{} 0,
\end{multline*}
thus we can use Corollary~\ref{cor:no_y^-eps}.

For the term $S_4(x,y)$ we use Lemma~\ref{lem: J_t J_u J_v} for all the four summands.

In the last term $S_5(x,y)$ the variables are split, and we have
\begin{multline*}
S_5(x,y)  =  \frac{J_{2s-1}(2x^{1/2})}{2x^{1/2}} \int_0^{x^{1/2}} J_{2s-1}(2t)dt \times \frac{J_{2s-1}(2y^{1/2})}{2y^{1/2}} \int_0^{y^{1/2}} J_{2s-1}(2t)dt  =  \\
  =  \frac{J_{2s-1}(2x^{1/2})}{2x^{1/2}} \biggl(\frac12 - \int_{x^{1/2}}^\infty J_{2s-1}(2t)dt \biggr) %\times \\
	\times \frac{J_{2s-1}(2y^{1/2})}{2y^{1/2}} \biggl( \frac12 - \int_{y^{1/2}}^\infty J_{2s-1}(2t)dt \biggr).
\end{multline*}
We use the estimates \eqref{eq:Bessel-small}, \eqref{eq:Bessel-large} to obtain
$$
\biggl| \frac{J_{2s-1}(2x^{1/2})}{2x^{1/2}} \int_0^{x^{1/2}} J_{2s-1}(2t)dt \biggr|  \le   \const(R) x^{s-1} \text{ for } x<R.
$$
Moreover,  from the estimates \eqref{eq:Bessel-large}, \eqref{eq:Bessel-large-int} we see that 
\begin{multline}\label{eq:J_int-J_estim}
\Biggl| \frac{J_{2s-1}(2x^{1/2})}{2x^{1/2}} \biggl( \frac12 - \int_{x^{1/2}}^\infty J_{2s-1}(2t)dt \biggr)  -  \frac1{2\sqrt{\pi}x^{3/4}} \cos(2x^{1/2}-(2s-1)\pi/2-\pi/4) \times \\
  \times \biggl(\frac12 - \frac{1}{2\sqrt{\pi}x^{1/4}}\sin(2x^{1/2}-(2s-1)\pi/2-\pi/4) \biggr) \Biggr|  \le  \frac{\const(R)}{x^{5/4}},
\end{multline}
for $x>R$, therefore the conditions of Corollary~\ref{cor:split-var_rig-simple} are satisfied and we have proved the required convergence to zero of the intergal~\eqref{eq:int-type} for $S_5(x,y)$.
Thus the required convergence is proved for the whole first term $K^{Bessel}_s(x,y) \cdot K^{Bessel}_s(y,x)$.

\begin{multline*}
\text{\bf The third part,\ } \\
\int_{D_{>R}} |\varphi^{(R, T)}(x)-\varphi^{(R, T)}(y)|^2 \frac{\partial }{\partial x} K^{Bessel}_s(x,y)\int_x^y K^{Bessel}_s(x,t)dt dx dy.
\end{multline*}

We will use the following notation:
\begin{multline*}
\frac{\partial }{\partial x} K^{Bessel}_s(x,y)  =		\\
			-		\frac{xy^{-1/2}J_{2s}(2x^{1/2})J_{2s-1}(2y^{1/2}) - x^{1/2}J_{2s}(2y^{1/2})J_{2s-1}(2x^{1/2})}{2(x-y)^2} 			+ \\
 \Biggl( \frac{y^{-1/2}J_{2s}(2x^{1/2})J_{2s-1}(y^{1/2}) + x^{1/2}y^{-1/2}J'_{2s}(2x^{1/2})J_{2s-1}(2y^{1/2})}{2(x-y)} - \\
			-		\frac{1/2x^{-1/2}J_{2s}(2y^{1/2})J_{2s-1}(2x^{1/2}) + J_{2s}(2y^{1/2})J'_{2s-1}(2x^{1/2})}{2(x-y)} \Biggr) +		\\
			+	\frac{J_{2s-1}(2y^{1/2})}{2y^{1/2}} \frac{J_{2s-1}(2x^{1/2})}{2x^{1/2}} 		=: 
			\frac{D_1(x,y)}{(x-y)^2} + \frac{D_2(x,y)}{(x-y)} + D_3(x,y).
\end{multline*}
And
\begin{multline*}
\int_x^y K^{Bessel}_s(x,t)dt = \int_0^{y-x} K^{Bessel}_s(x,t+x)dt 	= \\
2 \int_0^{\infty} \left(\frac{x}{t+x}\right)^{1/2} \KK^{Bessel}_{2, 2s-1}(4x, 4(t+x)) dt		- 
2 \int_{y-x}^{\infty} \left(\frac{x}{t+x}\right)^{1/2} \KK^{Bessel}_{2, 2s-1}(4x, 4(t+x)) dt 	+ \\
			+ \int_{x^{1/2}}^{y^{1/2}} J_{2s-1}(2p) dp \int_{x^{1/2}}^\infty J_{2s-1}(2p)dp - \frac12\int_{x^{1/2}}^{y^{1/2}} J_{2s-1}(2p) dp	=:  \\
			\tilde{I}_1(x) + I_2(x,y) + I_3(x,y) + I_4(x,y).
\end{multline*}

We will first separate the main part of $\tilde{I}_1(x)$, because the corresponding integrals are not absolutely convergent. We have
\begin{multline*}
\tilde{I}_1(x) = 2 \int_0^{\infty} \left(\frac{x}{t+x}\right)^{1/2} \Biggl(\frac{x^{1/2}J_{2s}(2x^{1/2})J_{2s-1}(2(t+x)^{1/2})}{4t}		 - \\
- \frac{(t+x)^{1/2}J_{2s}(2(t+x)^{1/2})J_{2s-1}(2x^{1/2})}{4t} \Biggr) dt,
\end{multline*}
and we use \eqref{eq:Bessel-large} to estimate it as follows:
\begin{multline*}
\Biggl| \tilde{I}_1(x) - \frac1{2\pi} \int_0^{\infty} \Biggl(\frac{x^{3/4}\cos(2x^{1/2} - s\pi-\pi/4) \cos(2(t+x)^{1/2} - (2s-1)\pi/2-\pi/4)}{t(t+x)^{3/4}} 	- \\
- \frac{x^{1/4}\cos(2(t+x)^{1/2} - s\pi-\pi/4) \cos(2x^{1/2} - (2s-1)\pi/2-\pi/4)}{t(t+x)^{1/4}} \Biggr) dt \Biggr|		\le   \\ 
\int_0^{1} \frac{\mathrm{const}}{\sqrt{x}} dt + \mathrm{const} x^{-1/4}\int_1^{\infty} \frac{dt}{t(t+x)^{1/4}}  \le  \frac{\mathrm{const}}{\sqrt{x}}.
\end{multline*}
We also have
\begin{multline*}
\Biggl| \int_0^{\infty} \frac{x^{1/4}\cos(2x^{1/2} - s\pi-\pi/4) \cos(2(t+x)^{1/2} - (2s-1)\pi/2-\pi/4)}{t(t+x)^{1/4}} \times  \\
  \times \Biggl( \frac{x^{1/2}}{(t+x)^{1/2}} - 1\Biggr) dt \Biggr|  \le   x^{1/4} \int_0^1 \frac{dt}{(t+x)^{3/4}(x^{1/2} + (t+x)^{1/2})}  +  \\
+  x^{1/4}\Biggl| \int_1^{\infty} \frac{\cos(2(t+x)^{1/2} - (2s-1)\pi/2-\pi/4)}{(t+x)^{3/4}(x^{1/2} + (t+x)^{1/2})} dt \Biggr| \le \frac{\const(R)}{\sqrt{x}},	\quad\text{ for } x>R,
\end{multline*}
where we have used \eqref{eq:Abel-Dirichlet} once again to estimate the second term. Now we estimate the main term:
\begin{multline*}
\frac1{2\pi} \int_0^{\infty} \frac{x^{1/4}}{t(t+x)^{1/4}} \biggl( \cos(2x^{1/2} - s\pi-\pi/4) \cos(2(t+x)^{1/2} - (2s-1)\pi/2-\pi/4) 	- \\
- \cos(2(t+x)^{1/2} - s\pi-\pi/4) \cos(2x^{1/2} - (2s-1)\pi/2-\pi/4) \biggr) dt  =  \\ 
= \frac{-1}{2\pi} \int_0^{\infty} \frac{x^{1/4}}{t(t+x)^{1/4}} \sin(2(t+x)^{1/2}-2x^{1/2}) dt  =  \frac{-1}{\pi} \int_0^{\infty} \frac{(u+1)^{1/2}}{(u+1)^2 - 1} \sin(2x^{1/2}u) du = \\
\frac{-1}{\pi} \int_0^{\infty} \frac{\sin(2x^{1/2}u)}{u} \frac{(u+1)^{1/2}}{u+2} du = \frac{-1}{2\pi} \int_0^{\infty} \frac{\sin(u)}{u}du + O\left(\frac1{\sqrt{x}}\right)  
=  -\frac14 + O\left(\frac1{\sqrt{x}}\right),
\end{multline*}
where we have put $u = (t/x+1)^{1/2}-1$ and then used Lemma~\ref{lem:sin-la-x-int-by-parts} at the last step.

After all, we can write
$$
\tilde I_1(x) = -\frac14 + I_1(x),\quad \int_x^y K^{Bessel}_s(x,t)dt = -\frac14 + I_1(x) + I_2(x,y) + I_3(x,y) + I_4(x,y),
$$
where $|I_1(x)| \le \const(R)/\sqrt{x}$ for $x\ge R$.

Now we write 
\begin{multline*}
\frac{\partial }{\partial x} K^{Bessel}_s(x,y)\int_x^y K^{Bessel}_s(x,t)dt  =  \biggl( \frac{D_1(x,y)}{(x-y)^2} + \frac{D_2(x,y)}{(x-y)} + D_3(x,y) \biggr) \times \\
\times \biggl( -\frac14 + I_1(x) + I_2(x,y) + I_3(x,y) + I_4(x,y) \biggr)  =  \\
=  \Biggl( \frac{D_1(x,y)}{(x-y)^2} \times \biggl( -\frac14 + I_1(x) + I_3(x,y) + I_4(x,y) \biggr) +  \frac{D_2(x,y)}{(x-y)} \times I_2(x,y)\Biggr)  +  \\
+ \Biggl( \frac{D_2(x,y)}{(x-y)} \times \biggl( I_1(x)  + I_3(x,y)\biggr) + D_3(x,y)\times I_2(x,y)\Biggr)  +  \frac{D_1(x,y)}{(x-y)^2} \times I_2(x,y)  +  \\
+ \frac{D_2(x,y)}{(x-y)} \biggl( -\frac14 + I_4(x,y) \biggr)  + D_3(x,y)\times \biggl( -\frac14 + I_1(x) + I_3(x,y) + I_4(x,y)\biggr),
\end{multline*}
and we estimate the summands one by one.

We use \eqref{eq:Bessel-large}, \eqref{eq:Bessel-large-der}, \eqref{eq:Bessel-large-int} and \eqref{eq:Abel-Dirichlet} to obtain the following simple estimates
\begin{multline*}
D_1(x,y) \le \frac{\const(R)}{\la^{1/4}},	\		D_2(x,y) \le \frac{\const(R)}{\sqrt{x}\la^{1/4}}, \	D_3(x,y) \le \frac{\const(R)}{x^{3/2}\la^{3/4}},	\\
	\  I_3(x,y) \le \frac{\const(R)}{\sqrt{x}}, \  I_4(x,y) \le \frac{\const(R)}{x^{1/4}}.
\end{multline*}

We additionally use \eqref{eq:Abel-Dirichlet} to obtain the following estimate:
\begin{multline}\label{eq:I_2}
| I_2(x,y) | =  2 \Biggl| \int_{y-x}^{\infty} \left(\frac{x}{t+x}\right)^{1/2} \Biggl(\frac{x^{1/2}J_{2s}(x^{1/2})J_{2s-1}((t+x)^{1/2})}{2t}		 - \\
- \frac{(t+x)^{1/2}J_{2s}((t+x)^{1/2})J_{2s-1}(x^{1/2})}{2t} \Biggr) dt \Biggr| 		\le  \\
\le  \mathrm{const}\cdot x|J_{2s}(x^{1/2})| \Biggl| \int_{y-x}^{\infty} \frac1{t(t+x)^{1/4}} \frac{J_{2s-1}((t+x)^{1/2})}{(t+x)^{1/4}} dt \Biggr| 	+  \\
+  \mathrm{const}\cdot x^{1/2}|J_{2s-1}(x^{1/2})| \Biggl| \int_{y-x}^{\infty} \frac{(t+x)^{1/4}}{t} \frac{J_{2s}((t+x)^{1/2})}{(t+x)^{1/4}} dt \Biggr| 	\le  \\
\le  \const(R) \frac{x^{3/4}y^{-1/4}+x^{1/4}y^{1/4}}{y-x} 	\le   \const(R) \frac{\sqrt{x}\la^{1/4}}{y-x}.
\end{multline}

From the estimates above we obtain
\begin{multline*}
\biggl| \frac{D_1(x,y)}{(x-y)^2} \times \biggl( -\frac14 + I_1(x) + I_3(x,y) + I_4(x,y) \biggr) +  \frac{D_2(x,y)}{(x-y)} \times I_2(x,y) \biggr| \le 
\frac{\const(R)\la^{-1/4}}{(x-y)^2} \times  \\
\times  \biggl( -\frac14 + \frac{\const(R)}{\sqrt{x}} + \frac{\const(R)}{x^{1/4}} \biggr) + 
\frac{\const(R)}{\sqrt{x}\la^{1/4}} \times \frac{\sqrt{x}\la^{1/4}}{(x-y)^2} \le 
\frac{\const(R)}{(x-y)^2},
\end{multline*}
and we can apply Proposition~\ref{prop:rig-det}\ref{prop:gen-rig_it-1}.

Then we have
\begin{multline*}
\biggl| \frac{D_2(x,y)}{(x-y)} \times \biggl( I_1(x) + I_3(x,y)\biggr) + D_3(x,y)\times I_2(x,y) \biggr| 
\le \frac{\const(R)}{\sqrt{x}\la^{1/4}(y-x)} \times \\
\times \biggl( \frac{\const(R)}{\sqrt{x}} + \frac{\const(R)}{\sqrt{x}} \biggr)  +  
\frac{\const(R)}{x^{3/2}\la^{3/4}} \times \frac{\sqrt{x}\la^{1/4}}{y-x}  \le  \frac{\const(R)}{x\la^{1/4}(y-x)},
\end{multline*}
and we can apply Proposition~\ref{prop:Pi/(x-y)_rig}\ref{eq:var-estimate_x>R}.

For $D_1(x,y)/(x-y)^2 \times I_2(x,y)$ we obtain an estimate
$$
\biggl| \frac{D_1(x,y)}{(x-y)^2} \times I_2(x,y) \biggr|  \le  \frac{\const(R)\la^{-1/4}}{(y-x)^2} \times
\frac{\sqrt{x}\la^{1/4}}{y-x}  \le  \frac{\const(R)\sqrt{x}}{(y-x)^3}.
$$
The corresponding integral diverges in the neighborhood of the point $\la = y/x = 1$, so we split the domain into 
two parts: $y-x \le x^{1/2}$ and $y-x > x^{1/2}$. For the first part we have 
\begin{multline*}
\biggl| \int_0^{y-x} \left(\frac{x}{t+x}\right)^{1/2} \KK^{Bessel}_{2, 2s-1}(4x, 4(t+x)) dt \biggr|  \le 
x^{1/2} \max_{0\le\xi\le x^{1/2}} \KK^{Bessel}_{2, 2s-1}(4x, 4(\xi+x))  =  \\
=  x^{1/2} \max_{0\le\xi\le x^{1/2}} \biggl( x^{1/2}J_{2s}(2x^{1/2}) \frac{J_{2s-1}(2(\xi+x)^{1/2}) - J_{2s-1}(2x^{1/2})}{4\xi}		 - \\
- J_{2s-1}(2x^{1/2})\frac{(\xi+x)^{1/2}J_{2s}(2(\xi+x)^{1/2}) - x^{1/2}J_{2s}(2x^{1/2})}{4\xi} \biggr)  \le  \\
\le  x^{1/2} \Biggl( \frac14|J_{2s}(2x^{1/2})| \max_{0\le\xi\le x^{1/2}} \biggl| J'_{2s-1}(2(\xi+x)^{1/2}) \biggr|	 
+ \frac14|J_{2s-1}(2x^{1/2})| \times  \\
\times  \biggl( \frac12x^{-1/2} \max_{0\le\xi\le x^{1/2}} \biggl| J_{2s}(2(\xi+x)^{1/2}) \biggr|   +  
\max_{0\le\xi\le x^{1/2}} \biggl| J'_{2s}(2(\xi+x)^{1/2}) \biggr|  \biggr)\Biggr)  \le  \const(R),
\end{multline*}
and we can apply Proposition~\ref{prop:rig-det}\ref{prop:gen-rig_it-1} for this part. For the second part we have 
$$
\biggl| \frac{D_1(x,y)}{(x-y)^2} \times I_2(x,y) \biggr|  \le  \frac{\const(R)}{(y-x)^2},
$$
and we apply Proposition~\ref{prop:rig-det}\ref{prop:gen-rig_it-1} once again.

For $D_2(x,y)/(x-y) ( -1/4 + I_4(x,y) )$ we have
\begin{multline*}
2D_2(x,y) = \Biggl(y^{-1/2}J_{2s}(2x^{1/2})J_{2s-1}(y^{1/2}) -		1/2x^{-1/2}J_{2s}(2y^{1/2})J_{2s-1}(2x^{1/2}) \Biggr) + \\
			+ \Biggl( x^{1/2}y^{-1/2}J'_{2s}(2x^{1/2})J_{2s-1}(2y^{1/2}) + J_{2s}(2y^{1/2})J'_{2s-1}(2x^{1/2}) \Biggr),
\end{multline*}
and for the first term there is an easy estimate
$$
\Biggl| y^{-1/2}J_{2s}(2x^{1/2})J_{2s-1}(y^{1/2}) -		1/2x^{-1/2}J_{2s}(2y^{1/2})J_{2s-1}(2x^{1/2}) \Biggr| \le \frac{\mathrm {const}}{x\la^{1/4}}.
$$
For the second term we have 
\begin{multline*}
\Biggl| \int_a^b \biggl( \la^{-1/2}J'_{2s}(2x^{1/2})J_{2s-1}(2\la^{1/2}x^{1/2}) + J_{2s}(2\la^{1/2}x^{1/2})J'_{2s-1}(2x^{1/2}) \biggr) \times  \\
	\times	\biggl(-\frac14 - \frac12\int_{x^{1/2}}^{\la^{1/2}x^{1/2}} J_{2s-1}(2p) dp \biggr)  + \frac1{\pi \la^{1/4} x^{1/2}} \biggl( \la^{-1/2}\sin(2x^{1/2}-s\pi-\pi/4) \times  \\
	\times \cos(2\la^{1/2}x^{1/2}-(2s-1)\pi/2-\pi/4)  + \cos(2\la^{1/2}x^{1/2}-s\pi-\pi/4) \times  \\
	\times \sin(2x^{1/2}-(2s-1)\pi/2-\pi/4) \biggr)	\biggl( -\frac14 - \frac1{4\sqrt{\pi}x^{1/4}}\sin(2x^{1/2}-(2s-1)\pi/2-\pi/4) +  \\
	+ \frac1{4\sqrt{\pi}(\la x)^{1/4}}\sin(2\la^{1/2}x^{1/2}-(2s-1)\pi/2-\pi/4)  \biggr) dx \Biggr|  \le  \frac{\const(R)\log(b)}{\la^{1/4}},
\end{multline*}
and
\begin{multline*}
\frac1{\pi \la^{1/4} x^{1/2}} \biggl( \la^{-1/2}\sin(2x^{1/2}-s\pi-\pi/4)\cos(2\la^{1/2}x^{1/2}-(2s-1)\pi/2-\pi/4)  +  \\
+ \cos(2\la^{1/2}x^{1/2}-s\pi-\pi/4) \sin(2x^{1/2}-(2s-1)\pi/2-\pi/4) \biggr)  =  \\
\frac1{\pi \la^{1/4} x^{1/2}} \biggl( -\la^{-1/2}\sin(2x^{1/2}-s\pi-\pi/4)\sin(2\la^{1/2}x^{1/2}-s\pi-\pi/4)  +  \\
+ \cos(2\la^{1/2}x^{1/2}-s\pi-\pi/4) \cos(2x^{1/2}-s\pi-\pi/4) \biggr)  =  \\
\frac1{2\pi \la^{1/4} x^{1/2}} \biggl( (1+\la^{-1/2})\sin\left(2x^{1/2}(\la^{1/2}+1)-2s\pi\right) + \\
+ (1-\la^{-1/2})\cos\left(2(\la^{1/2}-1)x^{1/2}\right)  \biggr),
\end{multline*}
thus we have 
\begin{multline*}
\left| \int_a^b D_2(x,\la x) \biggl( -1/4 + I_4(x,\la x) \biggr) dx \right| \le \frac{\const(R)\log(b)}{\la^{1/4}}  +  \\
\frac{\const(R)}{\la^{1/4}}\biggl(\frac1{\sqrt{\la}+1} +\frac1{\sqrt{\la}-1} +\frac1{2\sqrt{\la}+1} +\frac1{2\sqrt{\la}-1} +\frac1{\sqrt{\la}+2} +\frac1{\sqrt{\la}-2} \biggr)  
\ \text{  for  }\ R<x<y,
\end{multline*}
and we can use Proposition~\ref{prop:Pi/(x-y)_rig}\ref{eq:var-estimate_x>R} and~\ref{eq:var-estimate_x>R_eps_2=1}.

Regarding the product $D_3(x,y) \times \biggl( -1/4 + I_1(x) + I_3(x,y) + I_4(x,y) \biggr)$, we can split the variables for all the summands.
For $D_3(x,y)I_1(x)$ we have
$$
D_3(x,y)I_1(x) = \frac{J_{2s-1}(2y^{1/2})}{2y^{1/2}} \frac{J_{2s-1}(2x^{1/2})}{2x^{1/2}} I_1(x),
$$
where
$$
\biggl| \frac{J_{2s-1}(2x^{1/2})}{2x^{1/2}} I_1(x) \biggr|  \le  \frac{\const(R)}{x^{5/4}}
$$
and
\begin{equation}\label{eq:D3_estim}
\biggl| \frac{J_{2s-1}(2y^{1/2})}{2y^{1/2}} - \frac{\cos(2y^{1/2}-(2s-1)\pi/2-\pi/4)}{2\sqrt{\pi}y^{3/4}} \biggr|  \le  \frac{\mathrm {const}}{y^{5/4}},
\end{equation}
and we can apply Corollary~\ref{cor:split-var_rig-simple}.

For $D_3(x,y)I_3(x,y)$ we write
\begin{multline*}
D_3(x,y)I_3(x,y) = \frac{J_{2s-1}(2y^{1/2})}{2y^{1/2}} \frac{J_{2s-1}(2x^{1/2})}{2x^{1/2}} \int_{x^{1/2}}^{\infty} J_{2s-1}(2p)dp \int_{x^{1/2}}^\infty J_{2s-1}(2p)dp 	-  \\
- \frac{J_{2s-1}(2y^{1/2})}{2y^{1/2}} \frac{J_{2s-1}(2x^{1/2})}{2x^{1/2}} \int_{y^{1/2}}^{\infty} J_{2s-1}(2p) dp \int_{x^{1/2}}^\infty J_{2s-1}(2p)dp
\end{multline*}
where we can use the estimate~\eqref{eq:J_int-J_estim} for $J_{2s-1}(2x^{1/2})/2x^{1/2}\times\int_{x^{1/2}}^{\infty} J_{2s-1}(2p)dp$. We also have
$$
\Biggl| \frac{J_{2s-1}(2x^{1/2})}{2x^{1/2}} \biggl(\int_{x^{1/2}}^{\infty} J_{2s-1}(2p)dp \biggr)^2 \Biggr|  \le  \frac{\const(R)}{x^{5/4}},
$$
and we apply Corollary~\ref{cor:split-var_rig-simple} once again.

For $D_3(x,y)I_4(x,y)$ we have
\begin{multline*}
D_3(x,y)I_3(x,y)  =  -\frac12\frac{J_{2s-1}(2y^{1/2})}{2y^{1/2}} \frac{J_{2s-1}(2x^{1/2})}{2x^{1/2}} \int_{x^{1/2}}^{\infty} J_{2s-1}(2p)dp  +  \\
+ \frac12 \frac{J_{2s-1}(2y^{1/2})}{2y^{1/2}} \frac{J_{2s-1}(2x^{1/2})}{2x^{1/2}} \int_{y^{1/2}}^{\infty} J_{2s-1}(2p) dp.
\end{multline*}
We use the estimates \eqref{eq:D3_estim} and
\begin{multline*}
\biggl| \frac{J_{2s-1}(2y^{1/2})}{2y^{1/2}}\int_{y^{1/2}}^{\infty} J_{2s-1}(2p) dp  -  \\
- \frac{\cos(2y^{1/2}-(2s-1)\pi/2-\pi/4)\sin(2y^{1/2}-(2s-1)\pi/2-\pi/4)}{2\pi y} \biggr|  =  \\
= \biggl| \frac{J_{2s-1}(2y^{1/2})}{2y^{1/2}}\int_{y^{1/2}}^{\infty} J_{2s-1}(2p) dp - \frac{\cos(4y^{1/2}-2s\pi)}{4\pi y} \biggr|  \le  \frac{\mathrm {const}}{y^{3/2}},
\end{multline*}
and we apply Corollary~\ref{cor:split-var_rig-simple} one more time.

Finally for $D_3(x,y)$ we use the estimate~\eqref{eq:D3_estim} and then Corollary~\ref{cor:split-var_rig-simple} for the fourth time.

We have checked that the integral \eqref{eq:int-type} tends to zero for 
$$
\Pi(x,y) = \frac{\partial }{\partial x} K^{Bessel}_s(x,y)\int_x^y K^{Bessel}_s(x,t)dt,\quad R<x<y.
$$

\begin{multline*}
\text{\bf The fourth part,\ } \\
\int_{D_{<R}} |\varphi^{(R, T)}(x)-\varphi^{(R, T)}(y)|^2 \frac{\partial }{\partial x} K^{Bessel}_s(x,y)\int_x^y K^{Bessel}_s(x,t)dt dx dy.
\end{multline*}
The last step is to show that the required conditions are satisfied for $x<R$. 

For  $x<R$, $y>R$ we have 
\begin{gather*}
|D_1(x,y)| \le \const(R) y^{-1/4},  \\
|D_2(x,y)| \le \const(R) x^{s-1} y^{-1/4},  \\
|D_3(x,y)| \le \const(R) x^{s-1} y^{-3/4}. 
\end{gather*}
We also have 
\begin{multline*}
\left| \tilde{I}_1(x) \right|  =  2 \Biggl| \int_0^{\infty} \left(\frac{x}{t+x}\right)^{1/2} \Biggl(\frac{x^{1/2}J_{2s}(2x^{1/2})J_{2s-1}(2(t+x)^{1/2})}{4t}		 - \\
- \frac{(t+x)^{1/2}J_{2s}(2(t+x)^{1/2})J_{2s-1}(2x^{1/2})}{4t} \Biggr) dt  \Biggr|  \le   2 \Biggl| \int_0^x \left(\frac{x}{t+x}\right)^{1/2} \KK^{Bessel}_{2, 2s-1}(4x, 4(t+x)) dt	\Biggr|  +  \\
 +2 \Biggl| \int_x^{\infty} \left(\frac{x}{t+x}\right)^{1/2} \KK^{Bessel}_{2, 2s-1}(4x, 4(t+x)) dt	\Biggr|,
\end{multline*}
where
\begin{multline*}
\Biggl| \int_0^x \left(\frac{x}{t+x}\right)^{1/2} \Biggl(\frac{x^{1/2}J_{2s}(2x^{1/2})J_{2s-1}(2(t+x)^{1/2})}{2t}		 - \\
- \frac{(t+x)^{1/2}J_{2s}(2(t+x)^{1/2})J_{2s-1}(2x^{1/2})}{2t} \Biggr) dt  \Biggr|  \le  \\
\Biggl| x^{1/2}J_{2s}(2x^{1/2}) \int_0^x \left(\frac{x}{t+x}\right)^{1/2} \frac{J_{2s-1}(2(t+x)^{1/2}) - J_{2s-1}(2x^{1/2})}{2t}  dt \Biggl|  + \\
\Biggl| J_{2s}(2x^{1/2})J_{2s-1}(2x^{1/2}) \int_0^x \left(\frac{x}{t+x}\right)^{1/2} \frac{(t+x)^{1/2} - x^{1/2}}{2t} dt \Biggl|  + \\  
\Biggl| x^{1/2}J_{2s-1}(2x^{1/2}) \int_0^x \frac{J_{2s}(2(t+x)^{1/2}) - J_{2s}(2x^{1/2})}{2t}  dt \Biggl|  \le \\
\le  \frac12 x|J_{2s}(2x^{1/2})| \max_{\xi\in [0,x]} |J'_{2s-1}(2(x+\xi)^{1/2})| 
+ \mathrm{const}\cdot x^{1/2} |J_{2s}(2x^{1/2})J_{2s-1}(2x^{1/2})|  +  \\
+ \frac12 x|J_{2s-1}(2x^{1/2})| \max_{\xi\in [0,x]} |J'_{2s}(2(x+\xi)^{1/2})|   \le  \const(R),
\end{multline*}
and
\begin{multline*}
\Biggl| \int_x^{\infty} \left(\frac{x}{t+x}\right)^{1/2} \Biggl(\frac{x^{1/2}J_{2s}(2x^{1/2})J_{2s-1}(2(t+x)^{1/2})}{2t}		 - \\
- \frac{(t+x)^{1/2}J_{2s}(2(t+x)^{1/2})J_{2s-1}(2x^{1/2})}{2t} \Biggr) dt  \Biggr|  \le  \\
x\Biggl| J_{2s}(2x^{1/2}) \int_x^{\infty} \frac1{2t} \frac{J_{2s-1}(2(t+x)^{1/2})}{(t+x)^{1/2}}	dt \Biggr|  + \\
+ x^{1/2} \Biggl| J_{2s-1}(2x^{1/2}) \int_x^{\infty} \frac{(t+x)}{2t}\frac{J_{2s}(2(t+x)^{1/2})}{t+x} dt  \Biggr|  \le  \\
 \mathrm{const}|J_{2s}(2x^{1/2})| +  \mathrm{const}\cdot x^{1/2}|J_{2s-1}(2x^{1/2}) |  \le  \const(R) x^{s} \le \const(R),
\end{multline*}
where we have used \eqref{eq:Abel-Dirichlet} to estimate both terms. Finally we have
$$
\left| \tilde{I}_1(x) \right|  \le  \const(R).
$$

We use the estimate \eqref{eq:I_2} for $x<R$, $y>R$ to obtain
\begin{multline*}
| I_2(x,y) | \le  \mathrm{const}\cdot x|J_{2s}(2x^{1/2})| \Biggl| \int_{y-x}^{\infty} \frac1{t(t+x)^{1/4}} \frac{J_{2s-1}(2(t+x)^{1/2})}{(t+x)^{1/4}} dt \Biggr| 	+  \\
+  \mathrm{const}\cdot x^{1/2}|J_{2s-1}(2x^{1/2})| \Biggl| \int_{y-x}^{\infty} \frac{(t+x)^{1/4}}{t} \frac{J_{2s}(2(t+x)^{1/2})}{(t+x)^{1/4}} dt \Biggr| 	\le  \\
\le  \const(R) \frac{x^{s+1}y^{-1/4}+x^{s}y^{1/4}}{y-x}  \le  \const(R) \frac{y^{1/4}}{y-x}.
\end{multline*}
For
\begin{multline*}
I_3(x,y) + I_4(x,y)  =  \\
	=  \int_{y^{1/2}}^\infty J_{2s-1}(2p) dp \int_0^{x^{1/2}} J_{2s-1}(2p)dp  -  \int_{x^{1/2}}^\infty J_{2s-1}(2p) dp \int_0^{x^{1/2}} J_{2s-1}(2p)dp
\end{multline*}
we have
$$
\int_{x^{1/2}}^\infty J_{2s-1}(2p) dp \int_0^{x^{1/2}} J_{2s-1}(2p)dp  \le  \const(R)
$$
and
$$
\int_{y^{1/2}}^\infty J_{2s-1}(2p) dp \int_0^{x^{1/2}} J_{2s-1}(2p)dp  \le  \frac{\const(R)}{y^{1/4}}.
$$

Thus we can use Proposition~\ref{prop:rig-det}\ref{prop:gen-rig_it-2} for $D_1(x,y)/(x-y)^2 \times \biggl( \tilde{I}_1(x) + I_3(x,y) + I_4(x,y) \biggr)$ and for $D_2(x,y)/(x-y) \times I_2(x,y)$. 
For $D_1(x,y)/(x-y)^2 \times I_2(x,y)$ we obtain an estimate
$$
\left| \frac{D_1(x,y)}{(x-y)^2} \times I_2(x,y) \right|  \le  \frac{\const(R)x^s y^{-1/4}}{(y-x)^2} \times \frac{x^{s}y^{1/4}}{y-x} 
\le \frac{\const(R)}{(y-x)^3}.
$$
We split the domain into two parts: $x<R<y<2R$ and $x<R<2R<y$. For the first part we have
$$
\max_{x<R<y<2R} \left| \int_0^{y-x} \left(\frac{x}{t+x}\right)^{1/2} \KK^{Bessel}_{2, 2s-1}(4x, 4(t+x)) dt \right| 	\le  \const(R),
$$
and for the second part we have 
$$
\frac{\const(R)x^{2s}}{(y-x)^3} \le \frac{\const(R)}{(y-x)^2}
$$
thus the conditions of Proposition~\ref{prop:rig-det}\ref{prop:gen-rig_it-2} are satisfied for both domains.

We use Proposition~\ref{prop:Pi/(x-y)_rig}\ref{eq:var-estimate_x<R} for $D_2(x,y)/(x-y) \times \biggl( \tilde{I}_1(x) + I_3(x,y) + I_4(x,y) \biggr)$ and $D_3(x,y) \times I_2(x,y)$. For $D_3(x,y) \times \biggl( \tilde{I}_1(x) + I_3(x,y) + I_4(x,y) \biggr)$ we split the variables and use Corollary~\ref{cor:split-var_rig-simple}.

We have proved the required convergence to zero of the intergal~\eqref{eq:int-type} for the last term \\
$\frac{\partial }{\partial x} K^{Bessel}_s(x,y)\int_x^y K^{Bessel}_s(x,t)dt$,
and therefore for all the determinant $\det(\KK^{Bessel}_{4, s}(x,y))$.

\end{proof}

\subsection{Orthogonal Bessel process}\label{sec:orthog-Bessel}

Recall the definition of the matrix kernel $\KK^{Bessel}_{1,s}(x,y)$ in the introduction. 
We start with the following proposition, similar to Proposition~\ref{prop:int-rho-2_average-cond}. 

\begin{proposition}\label{prop:int-rho-2_average-cond_beta=1}
$$
\int_0^{\infty} \Biggl( \int_0^{\infty} \det \KK^{Bessel}_{1,s}(x,y) dy - K_{1,s}(x,x) \Biggr) dx = 0.
$$
\end{proposition}
\begin{proof}
Since 
$$
 \lim_{y\to\infty}K_{1,s}(y,x) = 0,
$$
we may use Corollary~\ref{cor:int-defect} again, and we will first simplify the expressions for the defect $\Def_{\KK^{Bessel}_{1,s}}(x)$ and for the limits
$$
- K_{1,s}(y,x) \int_x^y K_{1,s}(y,t)dt \biggl\vert_{y=0}^{\infty}.
$$
%We have
%\begin{multline*}
%\int_0^{\infty} \det \KK^{Bessel}_{4,s}(x,y) dy  =   \int_0^{\infty} K^{Bessel}_s(x,y) K^{Bessel}_s(y,x) dy  -  \\
%- \int_0^{\infty} \frac{\partial }{\partial x} K_s^{Bessel}(x,y) \int_y^x K_s^{Bessel}(x,t)dt		dy  =  \int_0^{\infty} K^{Bessel}_s(x,y) K^{Bessel}_s(y,x) dy  +  \\
%+ \int_0^{\infty} \frac{\partial }{\partial y} K_s^{Bessel}(y,x) \int_y^x K_s^{Bessel}(x,t)dt		dy  =  2\int_0^{\infty} K^{Bessel}_s(x,y) K^{Bessel}_s(y,x) dy  - \\
%-  K_s^{Bessel}(y,x) \int_x^y K_s^{Bessel}(y,t)dt \biggl\vert_{y=0}^{\infty},
%\end{multline*}
%where we have used the skew-symmetry of the kernel. 
First, we see that 
$$
\KK^{Bessel}_{2, s+1}(x, y) = \frac14 \int_0^1 J_{s+1}(\sqrt{ux}) J_{s+1}(\sqrt{uy}) du
$$ 
is an orthogonal projection onto the subspace of functions $f(x)$ such that $f(x)$ has its Hankel transform supported in $[0,1]$. We see from the orthogonal relations that $J_{s+1}(\sqrt{ux})$ lies in the image of $\KK^{Bessel}_{2, s+1}(x, y)$ for $u\in [0,1]$ and is orthogonal to the image for $u>1$.
It follows that 
\begin{gather*}
\frac{1}{x^{1/2}}\int_{x^{1/2}}^\infty J_{s+1}(t)dt = \int_1^\infty J_{s+1}(x^{1/2}t)dt \text{ is orthogonal to }  
\mathrm{Ran} \biggl( \KK^{Bessel}_{2, s+1}(x, y) \biggr),
\end{gather*}
that is 
\begin{align*}
\int_0^\infty    \frac{1}{y ^{1/2}}\int_{y^{1/2}}^\infty J_{s+1}(t)dt  \cdot  \KK^{Bessel}_{2, s+1}(x, y) dy = 0.
\end{align*}
Therefore we have
\begin{multline*}
\int_0^{\infty} K^{Bessel}_{1,s}(x,y) \left(\frac{y}{x}\right)^{1/2}\KK^{Bessel}_{2, s+1}(x, y) dy =
\int_0^{\infty} \biggl( \left(\frac{x}{y}\right)^{1/2} \KK^{Bessel}_{2, s+1}(x, y) +  \\
			+  \frac{J_{s+1}(y^{1/2})}{4y^{1/2}} \int_{x^{1/2}}^\infty J_{s+1}(t)dt \biggr) \left(\frac{y}{x}\right)^{1/2} \KK^{Bessel}_{2, s+1}(x, y) dy =
\\
= \int_0^{\infty}    \KK^{Bessel}_{2, s+1}(x, y) \KK^{Bessel}_{2, s+1}(x, y) dy +  \\
			+   \frac{1}{4x^{1/2}} \int_{x^{1/2}}^\infty J_{s+1}(t)dt  \cdot  \int_0^\infty  J_{s+1}(y^{1/2})  \KK^{Bessel}_{2, s+1}(x, y) dy  = 
\\
 = \KK^{Bessel}_{2, s+1}(x, x) + \frac{1}{4x^{1/2}} \int_{x^{1/2}}^\infty J_{s+1}(t)dt  \cdot J_{s+1} (x^{1/2}) = K^{Bessel}_{1,s}(x,x)
\end{multline*}
and
$$
\int_0^{\infty} \left(\frac{x}{y}\right)^{1/2}\KK^{Bessel}_{2, s+1}(x, y) \int_{y^{1/2}}^\infty J_{s+1}(t)dt dy = 0.
$$
We also have 
$$
\int_0^{\infty} \frac{J_{s+1}(y^{1/2})}{4y^{1/2}} \int_{y^{1/2}}^\infty J_{s+1}(t)dt dy  =
- \frac14\int_0^{\infty} d\biggl(\int_{y^{1/2}}^\infty J_{s+1}(t)dt \biggr)^2 = \frac14.
$$

Thus
\begin{multline*}
\int_0^{\infty} K^{Bessel}_{1,s}(x,y) K^{Bessel}_{1,s}(y,x)  dy = \int_0^{\infty} K^{Bessel}_{1,s}(x,y) \left(\frac{y}{x}\right)^{1/2}\KK^{Bessel}_{2, s+1}(x, y) dy  +  \\
+ \frac{J_{s+1}(x^{1/2})}{4x^{1/2}} \int_0^{\infty} \left(\frac{x}{y}\right)^{1/2} \KK^{Bessel}_{2, s+1}(x, y) \int_{y^{1/2}}^\infty J_{s+1}(t)dt dy  +  \\
		+ \frac{J_{s+1}(x^{1/2})}{4x^{1/2}} \int_{x^{1/2}}^\infty J_{s+1}(t)dt \int_0^{\infty} \frac{J_{s+1}(y^{1/2})}{4y^{1/2}} \int_{y^{1/2}}^\infty J_{s+1}(t)dt dy  =  \\
		K^{Bessel}_{1,s}(x,x)  +  \frac{J_{s+1}(x^{1/2})}{16x^{1/2}} \int_{x^{1/2}}^\infty J_{s+1}(t)dt,
\end{multline*}
and 
$$
2\Def_{\KK^{Bessel}_{1,s}}(x) =  \frac{J_{s+1}(x^{1/2})}{8x^{1/2}} \int_{x^{1/2}}^\infty J_{s+1}(t)dt.
$$
Now since
$$
K_{1,s} (0,x) =  \frac{J_{s+1}(x^{1/2})}{4x^{1/2}}  \quad\text {and}\quad  \lim_{y\to\infty}K_{1,s}(y,x) = 0,
$$
we have
\begin{multline*}
- K_{1,s}(y,x) \int_x^y K_{1,s}(y,t)dt \biggl\vert_{y=0}^{\infty} =  
\frac{J_{s+1}(x^{1/2})}{4x^{1/2}}  \biggl( \lim_{y\to 0} \int_y^x \left(\frac{y}{t}\right)^{1/2} \KK^{Bessel}_{2, s+1}(y, t) dt - \\
			-  \int_0^x \frac{J_{s+1}(t^{1/2})}{4t^{1/2}} \int_0^\infty J_{s+1}(p)dp \biggr) dt  =   
				-  \frac{J_{s+1}(x^{1/2})}{8x^{1/2}} \int_0^{x^{1/2}} J_{s+1}(t) dt. 
\end{multline*}

Write
\begin{multline*}
\int_0^{\infty} \det \KK^{Bessel}_{1,s}(x,y) dy - K_{1,s}(x,x)  =  \\
2 \Def_{\KK^{Bessel}_{1,s}}(x) 	-  K_{1,s}(y,x) \int_x^y K_{1,s}(y,t)dt \biggl\vert_{y=0}^{\infty}  =  \\
=  \frac{J_{s+1}(x^{1/2})}{8x^{1/2}} \biggl( \int_{x^{1/2}}^\infty J_{s+1}(t)dt  -  \int_0^{x^{1/2}} J_{s+1}(t) dt \biggr).
\end{multline*}
We directly see that
\begin{equation}\label{no-scr-1}
 \frac{J_{s+1}(x^{1/2})}{8x^{1/2}} \biggl( \int_{x^{1/2}}^\infty J_{s+1}(t)dt  -  \int_0^{x^{1/2}} J_{s+1}(t) dt \biggr)\neq 0,
 \end{equation}
 and therefore the relation \eqref{eq:int_rho-trunc} does not hold for the kernel $\KK^{Bessel}_{1,s}$. Nonetheless, we have
\begin{multline*}
\int_0^{\infty} \Biggl( \frac{J_{s+1}(x^{1/2})}{8x^{1/2}} \biggl( \int_{x^{1/2}}^\infty J_{s+1}(t)dt  -  \int_0^{x^{1/2}} J_{s+1}(t) dt \biggr) \Biggr) dx  =  \\
=  \frac{-1}{8}\int_0^{\infty} d\biggl( \int_{x^{1/2}}^\infty J_{s+1}(t)dt \biggr)^2  -  \frac18\int_0^{\infty} d\biggl( \int_0^{x^{1/2}} J_{s+1}(t)dt \biggr)^2  =
\frac1{8} - \frac1{8} = 0.
\end{multline*}
\end{proof}

\begin{proof}[Proof of the Theorem~\ref{thm:Bessel-rigidity}\ref{thm-item:orthog-Bessel-rigidity}]

Our plan is to expand the determinant,
$$
\det \KK^{Bessel}_{1,s}(x,y)  =  K_{1,s}(x,y) K_{1,s}(y,x)  + \frac{\partial }{\partial x} K_{1,s}(x,y) \Biggl( \int_x^y K_{1,s}(x,t)dt	+ \frac{\sgn(x-y)}2 \Biggr)
$$
and then to show that all the summands in the formula~\eqref{eq:Pfaff-var_general} for the variance, with $f(x) = \varphi^{(R, T)}(x)$, tend to zero term by term.
$$
\text{\bf The first part,\ } \int_{\RR_+}|\varphi^{(R, T)}(x)|^2 \biggl( K_{1,s}(x,x) - \int_{\RR} \det \KK^{Bessel}_{1, s}(x,y) dy\biggr) dx.
$$
We set
\begin{multline*}
\Pi(x) = K_{1,s}(x,x) - \int_0^{\infty} \det \KK^{Bessel}_{1, s}(x,y) dy  =  \\
  =  - \frac{J_{s+1}(x^{1/2})}{8x^{1/2}} \biggl( \int_{x^{1/2}}^\infty J_{s+1}(t)dt  -  \int_0^{x^{1/2}} J_{s+1}(t) dt \biggr)  =  \\
  =  \frac{J_{s+1}(x^{1/2})}{8x^{1/2}} \biggl( 1 - 2\int_{x^{1/2}}^\infty J_{s+1}(t)dt \biggr),
\end{multline*}
we use Proposition~\ref{prop:int-rho-2_average-cond_beta=1} to write
$$
\int_0^{\infty}|\varphi^{(R, T)}(x)|^2 \Pi(x) dx  =  \int_0^{\infty} \biggl(|\varphi^{(R, T)}(x)|^2 - 1\biggl) \Pi(x) dx 
$$
and then we use estimates \eqref{eq:Bessel-large} and \eqref{eq:Bessel-large-int} and Corollary~\ref{cor:split-var_one-dim} to see that
$$
\Biggl| \int_0^{\infty}\biggl(|\varphi^{(R, T)}(x)|^2 - 1\biggl) \Pi(x) dx \biggl| \xrightarrow[T\to \infty]{} 0.
$$

$$
\text{\bf The second part,\ } \int_{D} |\varphi^{(R, T)}(x)-\varphi^{(R, T)}(y)|^2 K_{1,s}(x,y) \cdot K_{1,s}(y,x) dx dy.
$$
We have
\begin{multline*}
K_{1,s}(x,y) \cdot K_{1,s}(y,x) = \\
\Biggl( \left(\frac{x}{y}\right)^{1/2} \KK^{Bessel}_{2, s+1}(x, y) + \frac{J_{s+1}(y^{1/2})}{4y^{1/2}} \int_{x^{1/2}}^\infty J_{s+1}(t)dt \Biggr) \times		\\
\times  \Biggl( \left(\frac{y}{x}\right)^{1/2} \KK^{Bessel}_{2, s+1}(y, x) + \frac{J_{s+1}(x^{1/2})}{4x^{1/2}} \int_{y^{1/2}}^\infty J_{s+1}(t)dt \Biggr)  = 
\left( \KK^{Bessel}_{2, s+1}(x, y) \right)^2 	+  \\
%\end{multline*}
%\begin{multline*}
\Biggl( \left(\frac{x}{y}\right)^{1/2} \KK^{Bessel}_{2, s+1}(x, y) \frac{J_{s+1}(x^{1/2})}{4x^{1/2}} \int_{y^{1/2}}^\infty J_{s+1}(t)dt  + \\
+ \frac{J_{s+2}(x^{1/2})J^2_{s+1}(y^{1/2})}{8(x-y)} \int_{x^{1/2}}^\infty J_{s+1}(t)dt \Biggr)  + \\
+ \frac{y^{1/2}x^{-1/2}J_{s+2}(y^{1/2})J_{s+1}(x^{1/2})J_{s+1}(y^{1/2})}{8(x-y)} \int_{x^{1/2}}^\infty J_{s+1}(t)dt + \\
+  \frac{J_{s+1}(x^{1/2})}{4x^{1/2}} \int_{x^{1/2}}^\infty J_{s+1}(t)dt \frac{J_{s+1}(y^{1/2})}{4y^{1/2}} \int_{y^{1/2}}^\infty J_{s+1}(t)dt =: \\
=: \frac{S_1(x,y)}{(x-y)^2} + \frac{S_2(x,y)}{x-y} + \frac{S_3(x,y)}{x-y}+ S_4(x,y).
\end{multline*}

The integral for the first term, $S_1(x,y)/(x-y)^2 = \biggl( \KK^{Bessel}_{2, s+1}(x, y) \biggr)^2$, was estimated in \cite{AB_det-rigidity} by Proposition~\ref{prop:rig-det}.

We use \eqref{eq:Bessel-large} and \eqref{eq:Bessel-large-int} to obtain
\begin{multline*}
\biggl|S_2(x,y)\biggr| = \Biggl| \left( x^{1/2}J_{s+2}(x^{1/2})J_{s+1}(y^{1/2}) - y^{1/2}J_{s+2}(y^{1/2})J_{s+1}(x^{1/2}) \right)	\times		\\
\frac{J_{s+1}(x^{1/2})}{4y^{1/2}} \int_{y^{1/2}}^\infty J_{s+1}(t)dt + 
\frac{J_{s+2}(x^{1/2})J^2_{s+1}(y^{1/2})}{8} \int_{x^{1/2}}^\infty J_{s+1}(t)dt \Biggr|  \le  \frac{\const(R)}{\la^{1/2}x},
\end{multline*}
 for any $R > 0$ and $x,y >R$, and where we have set, as usual, $\la = y/x$. We see that the conditions of Proposition~\ref{prop:Pi/(x-y)_rig}\ref{eq:main-prop-cond1} are satisfied.
Also we easily obtain
$$
\biggl|S_2(x,y)\biggr| \le \mathrm{const(R,s)} y^{-1/2} \text{  for  } x<R, y>R,
$$
and we can apply Proposition~\ref{prop:Pi/(x-y)_rig}.

We also have
\begin{multline*}
\Biggl| S_3(x,y)  -  \frac1{4\pi\sqrt{x}}\cos(y^{1/2}-(s+2)\pi/2-\pi/4)\cos(y^{1/2}-(s+1)\pi/2-\pi/4)J_{s+1}(x^{1/2}) \times \\
\times \int_{x^{1/2}}^\infty J_{s+1}(t)dt \Biggr|  \le  \mathrm {const} \frac{\left|J_{s+1}(x^{1/2})\int_{x^{1/2}}^\infty J_{s+1}(t)dt\right|}{\sqrt{xy}},
\end{multline*}
thus we can combine this estimate with \eqref{eq:Bessel-small}, \eqref{eq:Bessel-large}, \eqref{eq:Bessel-large-int} to use Proposition~\ref{prop:Pi/(x-y)_rig} in this case. 

For the main term of $S_3(x,y)$ we have
\begin{multline*}
\Biggl| \int_a^b  \frac1{4\pi\sqrt{x}}\cos(y^{1/2}-(s+2)\pi/2-\pi/4)\cos(y^{1/2}-(s+1)\pi/2-\pi/4)J_{s+1}(x^{1/2}) \times \\
\times\int_{x^{1/2}}^\infty J_{s+1}(t)dt dx \Biggr|  = 
\Biggl| \int_a^b  \frac1{8\pi\sqrt{x}}\cos(2\la^{1/2}x^{1/2}-(s+2)\pi) J_{s+1}(x^{1/2}) \times \\
\times \int_{x^{1/2}}^\infty J_{s+1}(t)dt dx \Biggr| \le \frac{\mathrm {const}\log(b)}{\sqrt{\la}}
\end{multline*}
by Lemma~\ref{lem:sin-la-x-int-by-parts}, and we can use Proposition~\ref{prop:Pi/(x-y)_rig}\ref{eq:var-estimate_x>R}. We also have 
\begin{multline*}
 \frac{\const(R)}{(\log T)^{2}} \left| \int\limits_0^R \int\limits_R^{T} \log^2(y-R+1) \frac{\cos(2y^{1/2} - (s+2)\pi) J_{s+1}(x^{1/2}) 
\int_{x^{1/2}}^\infty J_{s+1}(t)dt}{x^{1/2}(y-x)} dy dx  \right| \le  \\
\le \frac{\const(R)}{(\log T)^{2}} \int\limits_0^R x^{s/2} dx \Biggl( \int\limits_R^{R'} \frac{\log^2(y-R+1)}{y-R} dy +
\frac{\sqrt{R'} \log^2(R'-R+1)}{R'-R} \Biggr) \le \frac{\const(R)}{(\log T)^{2}},
\end{multline*}
where the function $\sqrt{y}\log^2(y-R+1)/(y-R)$ is decreasing for $y\ge R'$. And
\begin{multline*}
\left| \int\limits_0^R \int\limits_T^\infty \frac{\cos(2y^{1/2} - (s+2)\pi) J_{s+1}(x^{1/2}) 
\int_{x^{1/2}}^\infty J_{s+1}(t)dt}{x^{1/2}(y-x)} dy dx  \right| \le  \\
\frac{\const(R)\sqrt{T}}{T-R} \int\limits_0^R x^{s/2} dx \xrightarrow[T\to \infty]{} 0,
\end{multline*}
thus we can use Corollary~\ref{cor:no_y^-eps}.

For the last term $S_4(x,y)$ we have
$$
S_4(x,y) = \frac{J_{s+1}(x^{1/2})}{4x^{1/2}} \int_{x^{1/2}}^\infty J_{s+1}(t)dt \times \frac{J_{s+1}(y^{1/2})}{4y^{1/2}} \int_{y^{1/2}}^\infty J_{s+1}(t)dt,
$$
where the variables are split. We use the estimates \eqref{eq:Bessel-small}, \eqref{eq:Bessel-large} to obtain
$$
\biggl| \frac{J_{s+1}(x^{1/2})}{4x^{1/2}} \int_{x^{1/2}}^\infty J_{s+1}(t)dt \biggr|  \le   \const(R) x^{s/2} \text{ for } x<R.
$$
Moreover,  from the estimates \eqref{eq:Bessel-large}, \eqref{eq:Bessel-large-int} we see that 
\begin{multline}\label{eq:J_int-J_estim_beta=1}
\biggl| \frac{J_{s+1}(x^{1/2})}{4x^{1/2}} \int_{x^{1/2}}^\infty J_{s+1}(t)dt  -  \\
-  \frac1{2\pi x} \sin(x^{1/2}-(s+1)\pi/2-\pi/4)\cos(x^{1/2}-(s+1)\pi/2-\pi/4) \biggr|  \le  \frac{\const(R)}{x^{3/2}},
\end{multline}
for $x>R$ thus the conditions of Corollary~\ref{cor:split-var_rig-simple} are satisfied and we have proved the required convergence to zero of the intergal~\eqref{eq:int-type} for $S_4(x,y)$, 
and therefore for all the first term $K_{1,s}(x,y) \cdot K_{1,s}(y,x)$.

\begin{multline*}
\text{\bf The third part,\ } \\
\int_{D_{>R}} |\varphi^{(R, T)}(x)-\varphi^{(R, T)}(y)|^2  \frac{\partial }{\partial x} K_{1,s}(x,y) \Biggl( \int_x^y K_{1,s}(x,t)dt	+ \frac{\sgn(x-y)}2 \Biggr) dx dy.
\end{multline*}

We will use the following notation:
\begin{multline*}
\frac{\partial }{\partial x} K_{1,s}(x,y)  =		\\
			-		\frac{xy^{-1/2}J_{s+2}(x^{1/2})J_{s+1}(y^{1/2}) - x^{1/2}J_{s+2}(y^{1/2})J_{s+1}(x^{1/2})}{2(x-y)^2} 			+ \\
 \Biggl( \frac{y^{-1/2}J_{s+2}(x^{1/2})J_{2s-1}(y^{1/2}) + x^{1/2}y^{-1/2}J'_{s+2}(x^{1/2})J_{s+1}(y^{1/2})}{2(x-y)} - \\
			-		\frac{1/2x^{-1/2}J_{s+2}(y^{1/2})J_{s+1}(x^{1/2}) + J_{s+2}(y^{1/2})J'_{s+1}(x^{1/2})}{2(x-y)} \Biggr) +		\\
			+	\frac{J_{s+1}(y^{1/2})}{4y^{1/2}} \frac{J_{s+1}(x^{1/2})}{2x^{1/2}} 		=: 
			\frac{D_1(x,y)}{(x-y)^2} + \frac{D_2(x,y)}{(x-y)} + D_3(x,y).
\end{multline*}
And
\begin{multline*}
\int_x^y K_{1,s}(x,t)dt +\frac12\sgn(x-y) = \int_0^{y-x} K_{1,s}(x,t+x)dt  - \frac12	= \\
\int_0^{\infty} \left(\frac{x}{t+x}\right)^{1/2} \KK^{Bessel}_{2, s+1}(x, t+x) dt		- 
\int_{y-x}^{\infty} \left(\frac{x}{t+x}\right)^{1/2} \KK^{Bessel}_{2, s+1}(x, t+x) dt 	+ \\
			+ \frac12 \int_{x^{1/2}}^{y^{1/2}} J_{s+1}(p) dp \int_{x^{1/2}}^\infty J_{s+1}(p)dp  - \frac12 	=:
			\tilde{I}_1(x) + I_2(x,y) + I_3(x,y) - \frac12,
\end{multline*}
where we have used that $\sgn(x-y) = -1$ for $y>x$. We will first separate the main part of $\tilde{I}_1(x)$, because the corresponding integrals are not absolutely convergent. We have
\begin{multline*}
\tilde{I}_1(x) = \int_0^{\infty} \left(\frac{x}{t+x}\right)^{1/2} \Biggl(\frac{x^{1/2}J_{s+2}(x^{1/2})J_{s+1}((t+x)^{1/2})}{2t}		 - \\
- \frac{(t+x)^{1/2}J_{s+2}((t+x)^{1/2})J_{s+1}(x^{1/2})}{2t} \Biggr) dt,
\end{multline*}
and we use \eqref{eq:Bessel-large} to estimate it as follows:
\begin{multline*}
\Biggl| \tilde{I}_1(x) - \frac2\pi \int_0^{\infty} \Biggl(\frac{x^{3/4}\cos(x^{1/2} - (s+2)\pi/2-\pi/4) \cos((t+x)^{1/2} - (s+1)\pi/2-\pi/4)}{2t(t+x)^{3/4}} 	- \\
- \frac{x^{1/4}\cos((t+x)^{1/2} - (s+2)\pi/2 -\pi/4) \cos(x^{1/2} - (s+1)\pi/2-\pi/4)}{2t(t+x)^{1/4}} \Biggr) dt \Biggr|		\le   \\ 
\int_0^{1} \frac{\mathrm{const}}{x^{1/2}} dt + \mathrm{const} x^{-1/4}\int_1^{\infty} \frac{dt}{t(t+x)^{1/4}}  \le  \frac{\mathrm{const}}{x^{1/2}}.
\end{multline*}
We also have
\begin{multline*}
\Biggl| \int_0^{\infty} \frac{x^{1/4}\cos(x^{1/2} - (s+2)\pi/2 -\pi/4) \cos((t+x)^{1/2} - (s+1)\pi/2-\pi/4)}{2t(t+x)^{1/4}} \Biggl( \frac{x^{1/2}}{(t+x)^{1/2}} - 1\Biggr) dt \Biggr| \le   \\ 
x^{1/4} \int_0^1 \frac{dt}{2(t+x)^{3/4}(x^{1/2} + (t+x)^{1/2})}  +  x^{1/4}\Biggl| \int_1^{\infty} \frac{\cos((t+x)^{1/2} - (s+1)\pi/2-\pi/4)}{2(t+x)^{3/4}(x^{1/2} + (t+x)^{1/2})} dt \Biggr| \le \\
\le \frac{\const(R)}{\sqrt{x}},	\quad\text{ for } x>R,
\end{multline*}
where we have used \eqref{eq:Abel-Dirichlet} once again to estimate the second term. We proceed now with the estimate of the main term
\begin{multline*}
\frac2\pi \int_0^{\infty} \frac{x^{1/4}}{2t(t+x)^{1/4}} \biggl( \cos(x^{1/2} - (s+2)\pi/2 -\pi/4) \cos((t+x)^{1/2} - (s+1)\pi/2-\pi/4) 	- \\
- \cos((t+x)^{1/2} - (s+2)\pi/2 -\pi/4) \cos(x^{1/2} - (s+1)\pi/2-\pi/4) \biggr) dt  =  \\ 
= \frac{-2}{\pi} \int_0^{\infty} \frac{x^{1/4}}{2t(t+x)^{1/4}} \sin((t+x)^{1/2}-x^{1/2}) dt = \frac{-2}{\pi} \int_0^{\infty} \frac{(u+1)^{1/2}}{(u+1)^2 - 1} \sin(x^{1/2}u) du = \\
\frac{-2}{\pi} \int_0^{\infty} \frac{\sin(x^{1/2}u)}{u} \frac{(u+1)^{1/2}}{u+2} du = \frac{-1}{\pi} \int_0^{\infty} \frac{\sin(u)}{u}du + O\left(\frac1{\sqrt{x}}\right)  
=  -\frac12 + O\left(\frac1{\sqrt{x}}\right),
\end{multline*}
where we have put $u = (t/x+1)^{1/2}-1$ and then used Lemma~\ref{lem:sin-la-x-int-by-parts} at the last step.

After all, we can write
$$
\tilde I_1(x) = -\frac12 + I_1(x),  \quad  \int_x^y K_{1,s}(x,t)dt - \frac12  =  -1 + I_1(x) + I_2(x,y) + I_3(x,y),
$$
where $|I_1(x)| \le \const(R)/\sqrt{x}$ for $x\ge R$.

Now we write 
\begin{multline*}
\frac{\partial }{\partial x} K_{1,s}(x,y) \Biggl( \int_x^y K_{1,s}(x,t)dt - \frac12 \Biggr)  =  \biggl( \frac{D_1(x,y)}{(x-y)^2} + \frac{D_2(x,y)}{(x-y)} + D_3(x,y) \biggr) \times \\
\times \biggl( -1 + I_1(x) + I_2(x,y) + I_3(x,y) \biggr)  =  
\Biggl( \frac{D_1(x,y)}{(x-y)^2} \times \biggl( -1 + I_1(x) + I_3(x,y)\biggr) +  \\
\frac{D_2(x,y)}{(x-y)} \times I_2(x,y)\Biggr) + \Biggl( \frac{D_2(x,y)}{(x-y)} \times \biggl( I_1(x) + I_3(x,y)\biggr) + D_3(x,y)\times I_2(x,y)\Biggr)  +  \\
+  \frac{D_1(x,y)}{(x-y)^2} \times I_2(x,y) - \frac{D_2(x,y)}{(x-y)} + D_3(x,y)\times \biggl( -1 + I_1(x) + I_3(x,y)\biggr),
\end{multline*}
and we estimate the summands one by one.

We use \eqref{eq:Bessel-large}, \eqref{eq:Bessel-large-der}, \eqref{eq:Bessel-large-int} and \eqref{eq:Abel-Dirichlet} to obtain the following simple estimates
\begin{gather*}
D_1(x,y) \le \frac{\const(R)}{\la^{1/4}},	\		D_2(x,y) \le \frac{\const(R)}{\sqrt{x}\la^{1/4}}, \	D_3(x,y) \le \frac{\const(R)}{x^{3/2}\la^{3/4}},	\  I_3(x,y) \le \frac{\const(R)}{\sqrt{x}}.
\end{gather*}

We additionally use \eqref{eq:Abel-Dirichlet} to obtain the following estimate:
\begin{multline}\label{eq:I_2_beta=1}
| I_2(x,y) | =  \Biggl| \int_{y-x}^{\infty} \left(\frac{x}{t+x}\right)^{1/2} \Biggl(\frac{x^{1/2}J_{s+2}(x^{1/2})J_{s+1}((t+x)^{1/2})}{2t}		 - \\
- \frac{(t+x)^{1/2}J_{s+2}((t+x)^{1/2})J_{s+1}(x^{1/2})}{2t} \Biggr) dt \Biggr| 		\le  \\
\le  \mathrm{const}\cdot x|J_{s+2}(x^{1/2})| \Biggl| \int_{y-x}^{\infty} \frac1{t(t+x)^{1/4}} \frac{J_{s+1}((t+x)^{1/2})}{(t+x)^{1/4}} dt \Biggr| 	+  \\
+  \mathrm{const}\cdot x^{1/2}|J_{s+1}(x^{1/2})| \Biggl| \int_{y-x}^{\infty} \frac{(t+x)^{1/4}}{t} \frac{J_{s+2}((t+x)^{1/2})}{(t+x)^{1/4}} dt \Biggr| 	\le  \\
\le  \const(R) \frac{x^{3/4}y^{-1/4}+x^{1/4}y^{1/4}}{y-x} 	\le   \const(R) \frac{\sqrt{x}\la^{1/4}}{y-x}.
\end{multline}

From the estimates above we obtain
\begin{multline*}
\biggl| \frac{D_1(x,y)}{(x-y)^2} \times \biggl( -1 + I_1(x) + I_3(x,y)\biggr) +  \frac{D_2(x,y)}{(x-y)} \times I_2(x,y) \biggr| \le 
\frac{\const(R)\la^{-1/4}}{(x-y)^2} \times  \\
\times  \biggl( 1 + \frac{\const(R)}{\sqrt{x}} + \frac{\const(R)}{\sqrt{x}} \biggr) + 
\frac{\const(R)}{\sqrt{x}\la^{1/4}} \times \frac{\sqrt{x}\la^{1/4}}{(x-y)^2} \le 
\frac{\const(R)}{(x-y)^2},
\end{multline*}
and we can apply Proposition~\ref{prop:rig-det}\ref{prop:gen-rig_it-1}.

Then we have
\begin{multline*}
\biggl| \frac{D_2(x,y)}{(x-y)} \times \biggl( I_1(x) + I_3(x,y)\biggr) + D_3(x,y)\times I_2(x,y) \biggr| 
\le \frac{\const(R)}{\sqrt{x}\la^{1/4}(y-x)} \times \\
\times \biggl( \frac{\const(R)}{\sqrt{x}} + \frac{\const(R)}{\sqrt{x}} \biggr)  +  
\frac{\const(R)}{x^{3/2}\la^{3/4}} \times \frac{\sqrt{x}\la^{1/4}}{y-x}  \le  \frac{\const(R)}{x\la^{1/4}(y-x)},
\end{multline*}
and we can apply Proposition~\ref{prop:Pi/(x-y)_rig}\ref{eq:var-estimate_x>R}.

For $D_1(x,y)/(x-y)^2 \times I_2(x,y)$ we obtain an estimate
$$
\biggl| \frac{D_1(x,y)}{(x-y)^2} \times I_2(x,y) \biggr|  \le  \frac{\const(R)\la^{-1/4}}{(y-x)^2} \times
\frac{\sqrt{x}\la^{1/4}}{y-x}  \le  \frac{\const(R)\sqrt{x}}{(y-x)^3}.
$$
The corresponding integral diverges in the neighborhood of the point $\la = y/x = 1$, so we split the domain into 
two parts: $y-x \le x^{1/2}$ and $y-x > x^{1/2}$. For the first part we have 
\begin{multline*}
\biggl| \int_0^{y-x} \left(\frac{x}{t+x}\right)^{1/2} \KK^{Bessel}_{2, s+1}(x, t+x) dt \biggr|  \le 
x^{1/2} \max_{0\le\xi\le x^{1/2}} \KK^{Bessel}_{2, s+1}(x, \xi+x)  =  \\
= x^{1/2} \max_{0\le\xi\le x^{1/2}} \biggl( x^{1/2}J_{2s}(x^{1/2}) \frac{J_{2s-1}((\xi+x)^{1/2}) - J_{2s-1}(x^{1/2})}{\xi}		 - \\
- J_{2s-1}(x^{1/2})\frac{(\xi+x)^{1/2}J_{2s}((\xi+x)^{1/2}) - x^{1/2}J_{2s}(x^{1/2})}{\xi} \biggr)  \le  \\
\le  x^{1/2} \Biggl( \frac12|J_{2s}(x^{1/2})| \max_{0\le\xi\le x^{1/2}} \biggl| J'_{2s-1}((\xi+x)^{1/2}) \biggr|	 
+ \frac12|J_{2s-1}(x^{1/2})| \times  \\
\times  \biggl( x^{-1/2}\max_{0\le\xi\le x^{1/2}} \biggl| J_{2s}((\xi+x)^{1/2}) \biggr|  +  
\max_{0\le\xi\le x^{1/2}} \biggl| J'_{2s}((\xi+x)^{1/2}) \biggr| \biggr) \Biggr)  \le  \const(R),
\end{multline*}
and we can apply Proposition~\ref{prop:rig-det}\ref{prop:gen-rig_it-1} for this part. For the second part we have 
$$
\biggl| \frac{D_1(x,y)}{(x-y)^2} \times I_2(x,y) \biggr|  \le  \frac{\const(R)}{(y-x)^2},
$$
and we apply Proposition~\ref{prop:rig-det}\ref{prop:gen-rig_it-1} once again.

For $D_2(x,y)/(x-y)$ we have
\begin{multline*}
2D_2(x,y) = \Biggl(y^{-1/2}J_{s+2}(x^{1/2})J_{s+1}(y^{1/2}) -		1/2x^{-1/2}J_{s+2}(y^{1/2})J_{s+1}(x^{1/2}) \Biggr) + \\
			+ \Biggl( x^{1/2}y^{-1/2}J'_{s+2}(x^{1/2})J_{s+1}(y^{1/2}) + J_{s+2}(y^{1/2})J'_{s+1}(x^{1/2}) \Biggr),
\end{multline*}
and for the first term there is an easy estimate
$$
\Biggl| y^{-1/2}J_{s+2}(x^{1/2})J_{s+1}(y^{1/2}) -		1/2x^{-1/2}J_{s+2}(y^{1/2})J_{s+1}(x^{1/2}) \Biggr| \le \frac{\mathrm {const}}{x\la^{1/4}}.
$$
For the second term we have 
\begin{multline*}
\Biggl| \int_a^b \la^{-1/2}J'_{s+2}(x^{1/2})J_{s+1}(\la^{1/2}x^{1/2}) + J_{s+2}(\la^{1/2}x^{1/2})J'_{s+1}(x^{1/2})  +  \\
+ \frac2{\pi \la^{1/4} x^{1/2}} \biggl( \la^{-1/2}\sin(x^{1/2}-(s+2)\pi/2-\pi/4)\cos(\la^{1/2}x^{1/2}-(s+1)\pi/2-\pi/4)  +  \\
+ \cos(\la^{1/2}x^{1/2}- (s+2)\pi/2 -\pi/4) \sin(x^{1/2}-(s+1)\pi/2-\pi/4) \biggr) dx \Biggr|  \le  \frac{\const(R)\log(b)}{\la^{1/4}},
\end{multline*}
and
\begin{multline*}
\frac2{\pi \la^{1/4} x^{1/2}} \biggl( \la^{-1/2}\sin(x^{1/2}-(s+2)\pi/2-\pi/4)\cos(\la^{1/2}x^{1/2}-(s+1)\pi/2-\pi/4)  +  \\
+ \cos(\la^{1/2}x^{1/2}-(s+2)\pi/2-\pi/4) \sin(x^{1/2}-(s+1)\pi/2-\pi/4) \biggr)  =  \\
\frac2{\pi \la^{1/4} x^{1/2}} \biggl( -\la^{-1/2}\sin(x^{1/2}-(s+2)\pi/2-\pi/4)\sin(\la^{1/2}x^{1/2}-(s+2)\pi/2-\pi/4)  +  \\
+ \cos(\la^{1/2}x^{1/2}-(s+2)\pi/2-\pi/4) \cos(x^{1/2}-(s+2)\pi/2-\pi/4) \biggr)  =  \\
\frac2{\pi \la^{1/4} x^{1/2}} \biggl( (1+\la^{-1/2})\sin\left(x^{1/2}(\la^{1/2}+1)-(s+2)\pi\right) + \\
+ (1-\la^{-1/2})\cos\left((\la^{1/2}-1)x^{1/2}\right)  \biggr),
\end{multline*}
thus we have 
$$
\left| \int_a^b D_2(x,\la x) dx \right| \le \frac{\const(R)\log(b)}{\la^{1/4}}\ \text{  for  }\ R<x<y,
$$
and we can use Proposition~\ref{prop:Pi/(x-y)_rig}\ref{eq:var-estimate_x>R}.

Regarding the product $D_3(x,y) \times \biggl( -1 + I_1(x) + I_3(x,y)\biggr)$, we can split the variables for all the summands.
For $D_3(x,y)I_1(x)$ we have
$$
D_3(x,y)I_1(x) = \frac{J_{s+1}(y^{1/2})}{4y^{1/2}} \frac{J_{s+1}(x^{1/2})}{2x^{1/2}} I_1(x),
$$
where
$$
\biggl| \frac{J_{s+1}(x^{1/2})}{2x^{1/2}} I_1(x) \biggr|  \le  \frac{\const(R)}{x^{5/4}}
$$
and
\begin{equation}\label{eq:D3_estim_beta=1}
\biggl| \frac{J_{s+1}(y^{1/2})}{4y^{1/2}} - \frac{\cos(y^{1/2}-(s+1)\pi/2-\pi/4)}{2\sqrt{2\pi}y^{3/4}} \biggr|  \le  \frac{\mathrm {const}}{y^{5/4}},
\end{equation}
and we can apply Corollary~\ref{cor:split-var_rig-simple}.

For $D_3(x,y)I_3(x,y)$ we write
\begin{multline*}
D_3(x,y)I_3(x) = \frac{J_{s+1}(y^{1/2})}{4y^{1/2}} \frac{J_{s+1}(x^{1/2})}{4x^{1/2}} \int_{x^{1/2}}^{\infty} J_{s+1}(p)dp \int_{x^{1/2}}^\infty J_{s+1}(p)dp 	-  \\
- \frac{J_{s+1}(y^{1/2})}{4y^{1/2}} \frac{J_{s+1}(x^{1/2})}{4x^{1/2}} \int_{y^{1/2}}^{\infty} J_{s+1}(p) dp \int_{x^{1/2}}^\infty J_{s+1}(p)dp
\end{multline*}
where we can use the estimate~\eqref{eq:J_int-J_estim_beta=1} for $J_{s+1}(x^{1/2})/4x^{1/2}\times\int_{x^{1/2}}^{\infty} J_{s+1}(p)dp$. We also have
$$
\Biggl| \frac{J_{s+1}(x^{1/2})}{4x^{1/2}} \biggl(\int_{x^{1/2}}^{\infty} J_{s+1}(p)dp \biggr)^2 \Biggr|  \le  \frac{\const(R)}{x^{5/4}},
$$
and we apply Corollary~\ref{cor:split-var_rig-simple} once again.

Finally for $D_3(x,y)$ we use the estimate~\eqref{eq:D3_estim_beta=1} and then Corollary~\ref{cor:split-var_rig-simple} for the third time.

We have checked that the integral \eqref{eq:int-type} tends to zero for 
$$
\Pi(x,y) = \frac{\partial }{\partial x} K_{1,s}(x,y) \Biggl( \int_x^y K_{1,s}(x,t)dt - \frac12 \Biggr),  \quad  R<x<y.
$$

\begin{multline*}
\text{\bf The fourth part,\ } \\
\int_{D_{<R}} |\varphi^{(R, T)}(x)-\varphi^{(R, T)}(y)|^2  \frac{\partial }{\partial x} K_{1,s}(x,y) \Biggl( \int_x^y K_{1,s}(x,t)dt	+ \frac{\sgn(x-y)}2 \Biggr) dx dy.
\end{multline*}
The last step is to show that the required conditions are satisfied for $x<R$. 

For  $x<R$, $y>R$ we have 
\begin{gather*}
|D_1(x,y)| \le \const(R) y^{-1/4},  \quad  |D_2(x,y)| \le \const(R) y^{-1/4},  \\
|D_3(x,y)| \le \const(R) y^{-3/4}. 
\end{gather*}
We also have 
\begin{multline*}
\left| \tilde{I}_1(x) \right|  =  \Biggl| \int_0^{\infty} \left(\frac{x}{t+x}\right)^{1/2} \Biggl(\frac{x^{1/2}J_{s+2}(x^{1/2})J_{s+1}((t+x)^{1/2})}{2t}		 - \\
- \frac{(t+x)^{1/2}J_{s+2}((t+x)^{1/2})J_{s+1}(x^{1/2})}{2t} \Biggr) dt  \Biggr|  \le   \Biggl| \int_0^x \left(\frac{x}{t+x}\right)^{1/2} \KK^{Bessel}_{2, s+1}(x, t+x) dt	\Biggr|  +  \\
 +  \Biggl| \int_x^{\infty} \left(\frac{x}{t+x}\right)^{1/2} \KK^{Bessel}_{2, s+1}(x, t+x) dt	\Biggr|,
\end{multline*}
where
\begin{multline*}
\Biggl| \int_0^x \left(\frac{x}{t+x}\right)^{1/2} \Biggl(\frac{x^{1/2}J_{s+2}(x^{1/2})J_{s+1}((t+x)^{1/2})}{2t}		 - \\
- \frac{(t+x)^{1/2}J_{s+2}((t+x)^{1/2})J_{s+1}(x^{1/2})}{2t} \Biggr) dt  \Biggr|  \le  \\
\Biggl| \frac{x^{1/2}}2 J_{s+2}(x^{1/2}) \int_0^x \left(\frac{x}{t+x}\right)^{1/2} \frac{J_{s+1}((t+x)^{1/2}) - J_{s+1}(x^{1/2})}{t}  dt \Biggl|  + \\
\Biggl| \frac12 J_{s+2}(x^{1/2})J_{s+1}(x^{1/2}) \int_0^x \left(\frac{x}{t+x}\right)^{1/2} \frac{(t+x)^{1/2} - x^{1/2}}{t} dt \Biggl|  + \\  
\Biggl| \frac{x^{1/2}}2 J_{s+1}(x^{1/2}) \int_0^x \frac{J_{s+2}((t+x)^{1/2}) - J_{s+2}(x^{1/2})}{t}  dt \Biggl|  \le \\
\le  \frac{x^{3/2}}2 |J_{s+2}(x^{1/2})| \max_{\xi\in [0,x]} \frac{|J'_{s+1}((x+\xi)^{1/2})|}{2(x+\xi)^{1/2}} 
+ \mathrm{const}\cdot x^{1/2}J_{s+2}(x^{1/2})J_{s+1}(x^{1/2}) +  \\
+ \frac{x^{3/2}}2 |J_{s+1}(x^{1/2})| \max_{\xi\in [0,x]} \frac{|J'_{s+2}((x+\xi)^{1/2})|}{2(x+\xi)^{1/2}}   \le  \const(R),
\end{multline*}
and
\begin{multline*}
\Biggl| \int_x^{\infty} \left(\frac{x}{t+x}\right)^{1/2} \Biggl(\frac{x^{1/2}J_{s+2}(x^{1/2})J_{s+1}((t+x)^{1/2})}{2t}		 - \\
- \frac{(t+x)^{1/2}J_{s+2}((t+x)^{1/2})J_{s+1}(x^{1/2})}{2t} \Biggr) dt  \Biggr|  \le  \\
\frac{x}{2} \Biggl| J_{s+2}(x^{1/2}) \int_x^{\infty} \frac1t \frac{J_{s+1}((t+x)^{1/2})}{(t+x)^{1/2}}	dt \Biggr|  + \\
+ \frac{x^{1/2}}{2} \Biggl| J_{s+1}(x^{1/2}) \int_x^{\infty} \frac{(t+x)}{t}\frac{J_{s+2}((t+x)^{1/2})}{t+x} dt  \Biggr|  \le  \\
 \mathrm{const}|J_{s+2}(x^{1/2})| +  \mathrm{const}\cdot x^{1/2}|J_{s+1}(x^{1/2}) |  \le  \\
\le  \const(R) x^{s/2+1} \le \const(R),
\end{multline*}
where we have used \eqref{eq:Abel-Dirichlet} to estimate both terms. Finally we have
$$
\left| \tilde{I}_1(x) \right|  \le  \const(R).
$$

We use the estimate \eqref{eq:I_2_beta=1} for $x<R$, $y>R$ to obtain
\begin{multline*}
| I_2(x,y) | \le  \mathrm{const}\cdot x|J_{s+2}(x^{1/2})| \Biggl| \int_{y-x}^{\infty} \frac1{t(t+x)^{1/4}} \frac{J_{s+1}((t+x)^{1/2})}{(t+x)^{1/4}} dt \Biggr| 	+  \\
+  \mathrm{const}\cdot x^{1/2}|J_{s+1}(x^{1/2})| \Biggl| \int_{y-x}^{\infty} \frac{(t+x)^{1/4}}{t} \frac{J_{s+2}((t+x)^{1/2})}{(t+x)^{1/4}} dt \Biggr| 	\le  \\
\le  \const(R) \frac{x^{s/2+2}y^{-1/4}+x^{s/2+1}y^{1/4}}{y-x}  \le  \frac{\const(R)y^{1/4}}{y-x}.
\end{multline*}
For
$$
I_3(x,y)  =  \frac12\int_{x^{1/2}}^\infty J_{s+1}(p) dp \int_{x^{1/2}}^\infty J_{s+1}(p)dp  -  \frac12 \int_{y^{1/2}}^\infty J_{s+1}(p) dp \int_{x^{1/2}}^\infty J_{s+1}(p)dp
$$
we have
$$
\int_{x^{1/2}}^\infty J_{s+1}(p) dp \int_{x^{1/2}}^\infty J_{s+1}(p)dp  \le  \const(R)
$$
and
$$
\int_{y^{1/2}}^\infty J_{s+1}(p) dp \int_{x^{1/2}}^\infty J_{s+1}(p)dp  \le  \frac{\const(R)}{y^{1/4}}.
$$

Thus we can use Proposition~\ref{prop:rig-det}\ref{prop:gen-rig_it-2} for $D_1(x,y)/(x-y)^2 \times \biggl( \tilde{I}_1(x) + I_3(x,y) \biggr)$ and for $D_2(x,y)/(x-y) \times I_2(x,y)$. For $D_1(x,y)/(x-y)^2 \times I_2(x,y)$ we have obtained an estimate
$$
\left| \frac{D_1(x,y)}{(x-y)^2} \times I_2(x,y) \right|  \le  \frac{\const(R) y^{-1/4}}{(y-x)^2} \times \frac{y^{1/4}}{y-x} 
\le \frac{\const(R)}{(y-x)^3}.
$$
We split the domain into two parts: $x<R<y<2R$ and $x<R<2R<y$. For the first part we have
$$
\max_{x<R<y<2R} \left| \int_0^{y-x} \left(\frac{x}{t+x}\right)^{1/2} \KK^{Bessel}_{2, s+1}(x, t+x) dt \right| 	\le  \const(R),
$$
and for the second part we have 
$$
\frac{\const(R)}{(y-x)^3} \le \frac{\const(R)y^{-1}}{(y-x)^2}
$$
thus the conditions of Proposition~\ref{prop:rig-det}\ref{prop:gen-rig_it-2} are satisfied for both domains.

We use Proposition~\ref{prop:Pi/(x-y)_rig}\ref{eq:var-estimate_x<R} for $D_2(x,y)/(x-y) \times \biggl( \tilde{I}_1(x) + I_3(x,y) \biggr)$ and $D_3(x,y) \times I_2(x,y)$. For $D_3(x,y) \times \biggl( \tilde{I}_1(x) + I_3(x,y) \biggr)$ we split the variables and use Corollary~\ref{cor:split-var_rig-simple}.

We have proved the required convergence to zero of the intergal~\eqref{eq:int-type} for the whole term 
$$
\Pi(x,y) = \frac{\partial }{\partial x} K_{1,s}(x,y) \left( \int_x^y K_{1,s}(x,t)dt - 1/2 \right),
$$
and therefore for all the determinant $\det\left(\KK^{Bessel}_{1, s}(x,y)\right)$.

\end{proof}

\section{Stationary case and sine processes}\label{sec-stat}

\subsection{The spectral measure of linear statistics}
In this section, we deal with the spectral measure of continuous time stationary processes. In what follows, a simple configuration $\xi$ on $\R$ will be identified with its support $\mathcal{X} = \supp(\xi) \subset \R$, which is a locally finite countable subset of $\R$.

Recall that  the spectral measure $\mu_M$ for a complex-valued continuous time $L^2$-bounded stationary (in the weak sense) stochastic process $M = (M_t)_{t \in\R}$ is defined as the unique measure $\mu_M$ on $\R$ such that 
\[
\widehat{\mu_M}(t)  = \int_\R e^{-i 2 \pi \lambda t } d\mu_M(\lambda) = \E \Big[(M_0- \E(M_0) ) (\overline{M_t} - \E(\overline{M_0}) )\Big].
\]

Consider a stationary point process $\PP$ on $\R$. Assume that the first and second correlation functions of $\PP$ exist.    Let $\mathscr{X}$ denote a random configuration on $\R$ with distribution $\PP$ and let $g \in L^2(\R)$. We consider an additive functional (linear statistics):
$$
S_g(\mathscr{X}) = \sum_{x\in \mathscr{X}} g(x).
$$
The spectral measure for the centralized linear statistics 
\begin{align}\label{c-LS}
\Big(S_g (\mathscr{X} + t)  -  \E S_g (\mathscr{X})\Big)_{t\in\R}
\end{align}
is given as follows. For avoiding the analysis of convergence of integrals, in what follows, we assume for simplicity that $g$ is bounded and compactly supported. We have 
\begin{multline*}
\E\Big[\Big(S_g (\mathscr{X})  -  \E S_g (\mathscr{X})\Big) \Big(\overline{S_g (\mathscr{X} + t)}-  \E \overline{S_g (\mathscr{X})}\Big) \Big] =  \E\Big[  \sum_{x, y \in \mathscr{X}}  g(x)  \overline{g(y+t)}\Big] - |\E(S_g(\mathscr{X}))|^2
\\ 
= \int_{\R^2} g(x) \overline{g(y+t)}  \rho^{(2)}_{\PP}(x,y) dxdy + \int_\R  g(x) \overline{g(x+t)}  \rho^{(1)}_{\PP}(x,x)dx - \Big| \int_\R  g(x)   \rho^{(1)}_{\PP}(x,x)dx \Big|^2
\\
= \int_{\R^2} g(x) \overline{g(y+t)} \rho^{(2, T)}_{\PP}(x,y)  dxdy + \int_\R g(x) \overline{g(x+t)} \rho^{(1)}_{\PP}(x) dx,
\end{multline*}
where 
\begin{align}\label{trunc-cor}
 \rho^{(2, T)}_{\PP}(x,y)   =  \rho^{(2)}_{\PP}(x,y)  -  \rho^{(1)}_{\PP}(x)   \rho^{(1)}_{\PP}(y) .
\end{align}
The stationarity of $\PP$ implies that there exists a function $F: \R\rightarrow\R$ and a constant $\rho >0$, such that 
\begin{align}\label{F-rho}
 \rho^{(2, T)}_{\PP}(x,y)  = F(x-y) \an \rho^{(1)}_{\PP}(x) = \rho.
\end{align}
Denote by $g_t(x) : = g(x+t)$, then  $\widehat{g_t} (\lambda) = e^{i 2 \pi \lambda t} \widehat{g}(\lambda)$, where we use the definition of the Fourier transform:
\[
\widehat{g}(\lambda) = \int_\R g(x) e^{- i 2 \pi \lambda x } dx. 
\] 
 By Fubini theorem and  the Plancherel identity, 
\begin{multline*}
\int_{\R^2} g(x) \overline{g(y+t)} \rho^{(2, T)}_{\PP}(x,y)  dxdy = \int_{\R^2} g(x) \overline{g_t(y)} F(x-y)  dxdy  = 
\\
=   \Big\langle g , g_t * F\Big\rangle_{L^2(\R)} 
=   \Big\langle \widehat{g} , e^{i 2 \pi \lambda t} \widehat{g}  \cdot \widehat{F} \Big\rangle_{L^2(\R)}  =  \int_\R  |\widehat{g}(\lambda)|^2  e^{ - i 2 \pi \lambda t }   \widehat{F}(\lambda) d\lambda . 
\end{multline*}
On the other hand,
\[
\int_\R g(x) \overline{g(x+t)} \rho^{(1)}_{\PP}(x) dx = \rho \langle   g, g_t\rangle_{L^2(\R)}  = \rho \langle   \widehat{g}, e^{i 2 \pi \lambda t}\widehat{g}\rangle_{L^2(\R)}  =\rho  \int_\R  |\widehat{g}(\lambda)|^2  e^{ - i 2 \pi \lambda t }  d\lambda.
\]
Thus 
\begin{align*}
\E\Big[\Big(S_g (\mathscr{X})  -  \E S_g (\mathscr{X})\Big) \Big(\overline{S_g (\mathscr{X} + t)}-  \E \overline{S_g (\mathscr{X})}\Big) \Big]  
=  \int_\R  |\widehat{g}(\lambda)|^2  e^{ - i 2 \pi \lambda t }    (\widehat{F}(\lambda)  + \rho) d\lambda.
\end{align*}

Thus we have 
\begin{proposition}\label{prop:spectral-meas}
The spectral measure for the centralized linear statistics \eqref{c-LS} is given by
\begin{align*}
 |\widehat{g}(\lambda)|^2 (\widehat{F}(\lambda)  + \rho)  d\lambda,
\end{align*}
where $F$ and $\rho$ are given by 
\eqref{F-rho}.
%\[
%\rho = \rho_\PP^{(1)}(x) = \text{constant} \an F(x) = \rho_\PP^{(2)}(x, 0)- \rho^2.
%\]
\end{proposition}

\subsection{Proof of Proposition \ref{g-stat}}

The proof of Proposition \ref{g-stat} below is similar to that of \cite[Lemma 3.3]{BQ-JHerm}. 
\begin{proof}[Proof of Proposition \ref{g-stat}]
It suffices to prove that for any fixed $R>0$ and any positive integer $n\in \N$, we can construct a  real-valued Schwartz function $\varphi_n$ such that 
\[
|\varphi_n(x)|\le 1,  \quad
\sup_{x\in [-R, R]}  |  \varphi_n(x) - 1 | \le 1/n \an  \Var_{\PP}(S_{\varphi_n}) \le 1/n.
\]
By Proposition \ref{prop:spectral-meas} and the assumption \eqref{ass-str}, we have 
\begin{align*}
 \Var_{\PP}(S_{\varphi}) =\int_\R  |\widehat{\varphi}(\lambda)|^2  (\widehat{F}(\lambda)  -\widehat{F}(0)) d\lambda \le C \int_\R  |\widehat{\varphi}(\lambda)|^2  | \lambda| d\lambda  .
\end{align*}

Fix $n\in\N$.  Let $k\ge n$ be large enough such that for any $|t| \le R k^{-1}$, we have 
$$
| e^{i 2 \pi t } -1  |\le n^{-1}.
$$
 We claim that there exists a non-negative even function $\psi_n \in C_c^\infty(\R)$ supported in a $(\frac{1}{k})$-neighbourhood of $0$, such that 
\begin{align}\label{psi-n}
 \int_\R \psi_n(\lambda) d\lambda = 1 \an \int_\R | \lambda|   \psi_n(\lambda)^2 d\lambda \le \frac{1}{ C n}.
 \end{align}
 Indeed, since the function $\frac{1}{ Cn |\lambda|} \chi_{| \lambda| \le 1/k}$ is not integrable, there exists a Schwartz function $\psi_n$ such that $ \int_\R \psi_n = 1$ and 
 $$
 \psi_n(\lambda) \le \frac{1}{Cn |\lambda|} \chi_{| \lambda| \le 1/k},
$$
 for any $\lambda \in \R$.   This inequality implies $\supp (\psi_n) \subset [ - 1/k, 1/k]$ and 
 $$
  \int_\R | \lambda|   \psi_n(\lambda)^2 d\lambda \le  \Big(\sup_{\lambda  } | \lambda| \psi_n(\lambda)\Big)  \cdot \int_\R \psi_n(\lambda)d\lambda \le \frac{1}{Cn}.  
 $$
 Now set 
 $$
 \varphi_n (x) = \check{\psi}_n (x) = \int_\R \psi_n(\lambda) e^{i 2 \pi x \lambda} d\lambda.
 $$ 
 Then  $\varphi_n \in \mathscr{S}(\R)$,  $\varphi_n(0) = 1$ and $| \varphi_n(x)| \le 1$. Since $\psi_n$ is  even and real-valued, $\varphi_n$ is real-valued.  By \eqref{psi-n},  we have
 $$
 \Var_{\PP}(S_{\varphi_n}) \le C \int_\R | \lambda| | \widehat{\varphi}_n(\lambda)|^2 d\lambda  =  C \int_\R | \lambda| | \psi_n(\lambda)|^2 d\lambda \le n^{-1}.
 $$
Moreover, by our choice of $k$, if $| \lambda | \le k^{-1}$ and $|x|\le R$, then $| e^{i 2 \pi x \lambda} -1  |\le n^{-1}$. Hence for any $| x| \le R$, 
\begin{align*}
|  \varphi_n(x) - 1 |   &=  |  \varphi_n(x) - \varphi_n(0) | \le \int_\R | e^{i 2 \pi x \lambda} -1  | |\psi_n (\lambda) |d\lambda 
\\
& = \int_{| \lambda | \le k^{-1}} | e^{i 2 \pi x \lambda} -1  | |\psi_n (\lambda) |d\lambda \le n^{-1}.
\end{align*}
This completes the proof of the proposition.
\end{proof}

\subsection{Proof of Propositions \ref{prop-sine-1} and \ref{prop-sine-4}}
\begin{lemma} 
We have the following identity 
\begin{align}\label{id-1}
 \int_{ \R}  \rho^{(2, T)}_{\sine, 1}(x,y)   dy =  - \rho^{(1)}_{\sine, 1}(x)  =- 1. 
\end{align}
\end{lemma}

\begin{proof}
By the reproducing property of the sine kernel, we have
\[
\int_\R S(x-y)^2dy = S(x -x) = 1. 
\]
Therefore, 
\begin{multline*}
   \int_{ \R}  \rho^{(2, T)}_{\sine, 1}(x,y)   dy = - \int_\R  \det K_{\sine, 1}(x, y) dy 
 \\
 =  - \int_\R S(x-y)^2 dy + \int_\R  S'(x-y) IS(x-y) dy  - \int_\R S'(x-y) \varepsilon(x-y) dy 
 \\
  = -1  + IS(y-x) S(y-x)\Big|_{y=-\infty}^\infty - \int_{\R} S(y-x)^2 dy -
\\
- \frac{1}{2}\int_x^\infty S'(y-x) dy +
 \frac{1}{2} \int_{-\infty}^x S'(y-x) dy = -1. 
\end{multline*}
\end{proof}
\begin{lemma}
The  limit 
\begin{align*}
 \lim_{R, M \to \infty} \int_{-R}^M   \rho^{(2, T)}_{\sine, 4}(x,y)   dy
 \end{align*}
 exists and we have 
\begin{align*}
  \lim_{R, M \to \infty} \int_{-R}^M   \rho^{(2, T)}_{\sine, 4}(x,y)   dy =  - \rho^{(1)}_{\sine, 4}(x)  =- \frac{1}{2}. 
\end{align*}
\end{lemma}
\begin{proof}
Indeed, 
\begin{multline*}
\lim_{R, M \to \infty} \int_{-R}^M   \rho^{(2, T)}_{\sine, 4}(x,y)   dy
\\
 = - \frac{1}{4}      \lim_{R, M \to \infty}  \left( \int_{-R}^M  S(x-y)^2dy - \int_{-R}^M IS (x-y) S'(x-y)dy \right)
 \\
 =- \frac{1}{4}    \lim_{R, M \to \infty}  \left( 2 \int_{-R}^M S(y-x)^2 dy   -  IS(y-x) S(y-x)\Big|_{y=-R}^M \right) 
 \\
  = - \frac{1}{4}   \left( 2 \int_{-\infty}^\infty S(y-x)^2 dy   -  IS(y-x) S(y-x)\Big|_{y=-\infty}^\infty\right)  = - \frac{1}{2}. 
\end{multline*}
\end{proof}
Denote 
\[
 \int_{ \R}  \rho^{(2, T)}_{\sine, 4}(x,y)   dy :=  \lim_{R, M \to \infty} \int_{-R}^M   \rho^{(2, T)}_{\sine, 4}(x,y)   dy.
 \]
We have 
\begin{align}\label{id-2}
 \int_{ \R}  \rho^{(2, T)}_{\sine, 4}(x,y)   dy =  - \rho^{(1)}_{\sine, 4}(x).
\end{align}

In the orthogonal and symplectic cases, the equalities \eqref{id-1} and \eqref{id-2} can be interpreted as $\widehat{F}(0) = -\rho$. Therefore, in these two cases,  the spectral measure for the centralized linear statistics \eqref{c-LS} is given by
\begin{align}\label{s-meas}
 |\widehat{g}(\lambda)|^2  (\widehat{F}(\lambda)  -\widehat{F}(0)) d\lambda.
\end{align}

\section{Appendix.}

Another proof of rigidity for pfaffian sine-processes can be given by considering stationary processes of occupation numbers of consecutive intervals and applying the Kolmogorov criterion, in the spirit of \cite{BDQ} for the determinantal case. 

Let $\PP$ be a stationary point process on $\R$.  
For any bounded Borel subset $B \subset \R$, we denote by $\#_B$ the function that associates any configuration $\xi \in \Conf(\R)$ to  $\xi(B)$.  For any $\lambda> 0$ and any $n\in \Z$, set 
\[
I_n^{(\lambda)} : = [n \lambda- \lambda/2, n \lambda+ \lambda/2). 
\] 
The sequence 
\begin{align}\label{X-n-lambda}
(X_n^{(\lambda)})_{n\in\Z}: = (\#_{I_n^{(\lambda)}})_{n\in\Z}
\end{align}
defines a stationary stochastic process on the probability space $(\Conf(\R), \PP)$.  Assume that $\PP$ admits up to second order correlation measures. Then in particular,  $X_n^{(\lambda)}$ are square integrable for any $\lambda>0$. The number rigidity of $\PP$ directly follows once there exists a sequence of positive real numbers $(\lambda_k)_{k\in\N}$ such that $\lambda_k \to +\infty$ and for any $k\in\N$, we have 
\begin{align}\label{linear-rig}
X_0^{(\lambda_k)} - \E(X_0^{(\lambda_k)}) \in \overline{\spann\Big\{X_n^{(\lambda_k)} - \E(X_n^{(\lambda_k)}): n \in \Z \setminus \{0\}\Big\}}^{L^2}. 
\end{align}

\begin{proposition}[{\cite[Theorem 3.1]{BDQ}}]\label{prop-l-rig}
Let  $Z  = (Z_n)_{n\in \Z}$ be a stationary stochastic process. If
\begin{align}\label{summable}
\sup_{N \ge 1 }  \left ( N \sum_{|n| \ge N}  | \Cov(Z_0, Z_n)| \right) < \infty,
\end{align}
and 
\begin{align}\label{cov-cond}
\sum_{n \in \Z}  \Cov(Z_0, Z_n) = 0.
\end{align}
Then 
\begin{align*}
Z_0 - \E(Z_0) \in \overline{\spann\Big\{Z_n - \E(Z_n): n \in \Z \setminus \{0\}\Big\}}^{L^2}. 
\end{align*}
\end{proposition}

\begin{remark}\label{rem:summable}
The absolute summability $\sum_{n \in \Z}  | \Cov(Z_0, Z_n)| <\infty$ is clear from~\eqref{summable}.
\end{remark}

 Consider now a Pfaffian point process $\PP_K$ induced by a matrix kernel $K$  such that \eqref{eq:skew-symm} is satisfied. We have
\begin{align}\label{1-cor}
\rho^{(1)}_{\PP_K}(x) = \Pf
\begin{bmatrix}
0										& 	K_{11}(x,x) \\
-K_{11}(x,x)		&		0
\end{bmatrix} = K_{11}(x,x), 
\end{align}
and also
\begin{align}\label{2-cor}
\begin{split}
\rho^{(2)}_{\PP_K}(x,y)&= \Pf
\begin{bmatrix}
0										& 	K_{11}(x,x)	&	-K_{12}(x,y)		&	K_{11}(x,y)	\\
-K_{11}(x,x)		&		0									&	-K_{22}(x,y)		&	K_{21}(x,y)	\\
-K_{12}(y,x)		&	K_{11}(y,x)		&		0									&	K_{11}(y,y)	\\
-K_{22}(y,x)		&	K_{21}(y,x)		&	-K_{11}(y,y)		&	0
\end{bmatrix} 
\\
 &= K_{11}(x,x)K_{11}(y,y) - K_{11}(x,y)K_{22}(x,y) + K_{12}(x,y)K_{21}(x,y)  
\\
&= \rho^{(1)}_{\PP_K}(x) \rho^{(1)}_{\PP_K}(y) - \det  K(x, y).
\end{split}
\end{align}
In what follows, we denote the truncated second correlation function
\begin{align}\label{tr-cor}
\rho^{(2, T)}_{\PP_K}(x,y): =\rho^{(2)}_{\PP_K}(x,y) - \rho^{(1)}_{\PP_K}(x) \rho^{(1)}_{\PP_K}(y) =  - \det  K(x, y). 
\end{align}

\begin{lemma}\label{lem-cov}
Fix $\lambda> 0$.  Then for $(X_n^{(\lambda)})_{n\in\Z}$ defined by~\eqref{X-n-lambda} we have
\begin{align*}
\Cov(X_0^{(\lambda)}, X_n^{(\lambda)}) = \left\{  \begin{array}{lc}  \displaystyle{\int_{I_0^{(\lambda)}}} \rho_{\PP_K}^{(1)}(x)  dx + 
							\int_{I_0^{(\lambda)} \times I_0^{(\lambda)}}  \rho^{(2, T)}_{\PP_K}(x,y)   dx dy,   & n = 0; \vspace{3mm} \\  
\displaystyle{\int_{I_0^{(\lambda)}\times I_n^{(\lambda)}}}  \rho^{(2, T)}_{\PP_K}(x,y)  dx dy, &  n \ne 0.   \end{array} \right. . 
\end{align*}
\end{lemma}
\begin{proof}
Fix $\lambda> 0$. For brevity, in this proof, we denote $X_n^{(\lambda)}, I_n^{(\lambda)}$  by $X_n, I_n$ respectively.    Let $\mathscr{X}$ denote a random configuration on $\R$ with distribution $\PP_K$. 
If $n\ne 0$, then 
\begin{align*}
\E(X_0X_n)& = \E\Big(\sum_{x\in \mathscr{X}} \ch_{I_0}(x) \sum_{y\in \mathscr{X}} \ch_{I_n}(y)  \Big)  = \E\Big(\sum_{x, y\in \mathscr{X}, x \ne y} \ch_{I_0}(x)  \ch_{I_n}(y)  \Big)
\\
 &= \int_{I_0 \times I_n}  \rho_{\PP_K}^{(2)}(x, y) dx dy  = \int_{I_0 \times I_n}  \Big[\rho^{(2, T)}_{\PP_K}(x,y) + \rho^{(1)}_{\PP_K}(x) \rho^{(1)}_{\PP_K}(y) \Big] dxdy
 \\
 & = \int_{I_0 \times I_n}  \rho^{(2, T)}_{\PP_K}(x,y) dxdy  + \E(X_0) \E(X_n). 
 \end{align*}
 It follows that 
 \[
 \Cov(X_0, X_n) = \E(X_0 X_n) - \E(X_0) \E(X_n) = \int_{I_0 \times I_n}  \rho^{(2, T)}_{\PP_K}(x,y) dxdy. 
 \]
If $n=0$, then 
\begin{align*}
\E(X_0^2)&  = \E\Big(\sum_{x, y\in \mathscr{X} } \ch_{I_0}(x)  \ch_{I_0}(y)  \Big)= \E\Big(\sum_{x \in \mathscr{X} } \ch_{I_0}(x)   \Big) + \E\Big(\sum_{x, y\in \mathscr{X}, x\ne y } \ch_{I_0}(x)  \ch_{I_0}(y)  \Big).
\end{align*}
By similar computation as above, we obtain 
\begin{align*}
\E(X_0^2) = \int_{I_0} \rho^{(1)}_{\PP_K}(x)  dx +  \int_{I_0 \times I_0}  \rho^{(2, T)}_{\PP_K}(x,y) dxdy  + \E(X_0) \E(X_0) 
\end{align*}
and thus 
\[
\Cov(X_0, X_0) = \int_{I_0} \rho^{(1)}_{\PP_K}(x)  dx +  \int_{I_0 \times I_0}  \rho^{(2, T)}_{\PP_K}(x,y) dxdy. 
\]
\end{proof}

\subsection{The orthogonal sine process}

\begin{proof}[Proof of Proposition \ref{prop-sine-1}]
It suffices to show that for any $\lambda> 0$, the hypothesis of Proposition \ref{prop-l-rig} is satisfied for the stochastic process $(X_n^{(\lambda)})_{n\in\Z}$ defined in  \eqref{X-n-lambda}. For brevity, in this proof, we will omit the superscript $\lambda$ in the notation $X_n^{(\lambda)}, I_n^{(\lambda)}$. 

{\flushleft Claim A1: }  We have 
\begin{align}\label{sum-decay}
\sup_{N\in\N} \Big(N \sum_{|n| \ge N } | \Cov(X_0, X_n) |\Big) <\infty.
\end{align}
 We will use the estimate of $\rho^{(2, T)}_{\sine, 1}(x,y)$ in   Forrester \cite[formula (7.135)]{Forrester-log}: 
\[
 \rho^{(2, T)}_{\sine, 1}(x, 0)   =  O\Big(\frac{1}{x^2}\Big)  \quad \text{as $|x| \to\infty$.} 
\]
In other words,  there exists $C> 0$, such that 
\[
| \rho^{(2, T)}_{\sine, 1}(x,0) |  \le \frac{C}{x^2}. 
\]
By Lemma \ref{lem-cov} and note that $\rho^{(2, T)}_{\sine, 1}(x, y) = \rho^{(2, T)}_{\sine, 1}(x-y,0)$, we have 
\begin{multline*}
\sum_{|n|\ge N}| \Cov(X_0, X_n)|  \le \sum_{|n| \ge N} \int_{I_0 \times I_n }  | \rho^{(2, T)}_{\sine, 1}(x,y) | dx dy  = \int_{I_0} dx \int_{|y| \ge N\lambda  - \lambda/2 } | \rho^{(2, T)}_{\sine, 1}(x-y, 0) | dy 
\\
\le  \int_{I_0}  \sup_{x \in I_0}  \Big( \int_{|y| \ge N\lambda  - \lambda/2 } | \rho^{(2, T)}_{\sine, 1}(x-y, 0) | dy\Big)   dx 
\\
\le \int_{I_0}    \Big( \int_{|z| \ge N\lambda  - \lambda } | \rho^{(2, T)}_{\sine, 1}(z, 0) | dz \Big)   dx  \le \lambda \int_{|z| \ge N\lambda  - \lambda }  \frac{C}{z^2} dz   =  \frac{2C}{N-1}. 
\end{multline*}
We thus obtain the inequality \eqref{sum-decay}. 

{\flushleft Claim B1: } We have 
\begin{align}\label{sum-0}
\sum_{n\in\Z} \Cov(X_0, X_n) = 0. 
\end{align}
Indeed, by \eqref{sum-decay}, we already know that the above series converges absolutely.   And by Lemma \ref{lem-cov},  we have 
\begin{align*}
\sum_{n\in\Z} \Cov(X_0, X_n)  =  \int_{I_0} \rho^{(1)}_{\sine, 1}(x)  dx +  \int_{I_0 \times \R}  \rho^{(2, T)}_{\sine, 1}(x,y) dxdy.  
\end{align*}
Then by using the equality \eqref{id-1}, we obtain the desired equality \eqref{sum-0}. 

The proof of Proposition \ref{prop-sine-1} is complete. 
\end{proof}

\subsection{The symplectic sine process}

In what follows, we will use the following estimate of $\rho^{(2, T)}_{\sine, 4}(x,y)$ in   Forrester \cite[formula (7.94)]{Forrester-log}: 
\[
 \rho^{(2, T)}_{\sine, 4}(x, 0)   =  \frac{\cos (\pi x)}{ 8 x} + O\Big(\frac{1}{x^2}\Big)  \quad \text{as $|x| \to\infty$.} 
\]
That is,   there exists $C> 0$, such that 
\begin{align}\label{es-large-d}
\Big| \rho^{(2, T)}_{\sine, 4}(x,0)  -  \frac{\cos (\pi x)}{ 8 x}  \Big|  \le \frac{C}{x^2}. 
\end{align}
Note that 
\[
 \int_{ \R}  |\rho^{(2, T)}_{\sine, 4}(x,y) |  dy =   \int_\R |\rho^{(2, T)}_{\sine, 4}(t,0)|dt = \infty. 
\]

\begin{proof}[Proof of Proposition \ref{prop-sine-4}]
It suffices to show that for any  positive even integer $2k \in 2 \N$, the hypothesis of Proposition \ref{prop-l-rig} is satisfied for the stochastic process $(X_n^{(2 k)})_{n\in\Z}$ defined in  \eqref{X-n-lambda}.  Let us now fix the  positive even integer $ 2k \in 2\N$.  For brevity, in this proof, we will omit the superscript $k$ in the notation $X_n^{(2k)}, I_n^{(2k)}$. 

{\flushleft Claim A2: }  We have 
\begin{align}\label{sum-decay-4}
\sup_{N\in\N} \Big(N \sum_{|n| \ge N } | \Cov(X_0, X_n) |\Big) <\infty.
\end{align}
By \eqref{es-large-d} and Lemma \ref{lem-cov} and note that $\rho^{(2, T)}_{\sine, 4}(x, y) = \rho^{(2, T)}_{\sine, 4}(x-y,0)$, we have 
\begin{multline*}
\sum_{|n|\ge N}| \Cov(X_0, X_n)|  =  \sum_{|n| \ge N}  \Big| \int_{I_0 \times I_n }   \rho^{(2, T)}_{\sine, 4}(x,y)  dx dy\Big| 
\\
\le   \sum_{|n| \ge N}  \Big|   \underbrace{ \int_{I_0 \times I_n }  \frac{\cos(\pi(x-y))}{8 (x-y)}  dx dy}_{\text{denoted $\sigma_n$}} \Big| +  C  \sum_{|n| \ge N}   \int_{I_0 \times I_n }   \frac{1}{(x-y)^2}  dx dy. 
\end{multline*}
For the term $\sigma_n$, we have 
\begin{multline*}
\sigma_n= \int_{I_0} dx \int_{I_n}  \frac{\cos(\pi(x-y))}{8 (x-y)}  dy  = \int_{I_0} dx \int_{I_n}  \frac{ d \sin(\pi(y-x))}{8 \pi (x-y)}  = 
\\
= \int_{I_0} dx \left[ \frac{  \sin(\pi(y-x))}{8 \pi (x-y)} \Big|_{y=2nk- k}^{2nk+k}   - \int_{I_n}\frac{\sin(\pi (y-x))}{8\pi (x-y)^2} dy\right]. 
\end{multline*}
Note that since the function $ y \mapsto \sin(\pi(y-x))$ is $2 \pi$-periodic, we have 
\[
\Big| \frac{  \sin(\pi(y-x))}{8 \pi (x-y)} \Big|_{y=2nk- k}^{2nk+k}  \Big|= \frac{ | \sin (\pi (2nk+k -x)| }{8\pi}  \int_{I_n} \frac{1}{(y-x)^2}dy \le \int_{I_n} \frac{1}{8 \pi (y-x)^2}dy.
\]
Therefore, we have 
\[
|\sigma_n| \le  \int_{I_0 \times I_n }   \frac{1}{8 \pi (x-y)^2}  dx dy  + \int_{I_0 \times I_n}\frac{|\sin(\pi (y-x))|}{8\pi (x-y)^2} dxdy \le  \int_{I_0 \times I_n }   \frac{1}{4 \pi (x-y)^2}  dx dy. 
\] 
Consequently, 
\begin{align*}
\sum_{|n|\ge N}| \Cov(X_0, X_n)| & \le  (C+ \frac{1}{4 \pi})  \sum_{|n| \ge N}   \int_{I_0 \times I_n }   \frac{1}{(x-y)^2}  dx dy
\\
&=  (C+ \frac{1}{4 \pi})\int_{I_0} dx \int_{|y| \ge Nk   - k /2 }  \frac{1}{(x-y)^2}  dy 
\\
& \le (C+ \frac{1}{4 \pi}) \int_{I_0}    \Big( \int_{|z| \ge N k   - k } \frac{1}{z^2}  dz \Big)   dx    =  \frac{C+ \frac{1}{4 \pi}}{N-1}. 
\end{align*}

{\flushleft Claim B2: } We have 
\begin{align}\label{sum-0-4}
\sum_{n\in\Z} \Cov(X_0, X_n) = 0. 
\end{align}
Indeed, by \eqref{sum-decay-4}, we already know that the above series converges absolutely.   And by Lemma \ref{lem-cov},  we have 
\begin{align*}
\sum_{n\in\Z} \Cov(X_0, X_n)  &=  \int_{I_0} \rho^{(1)}_{\sine, 4}(x)  dx +  \sum_{n\in \Z} \int_{I_0 \times I_n}  \rho^{(2, T)}_{\sine, 4}(x,y) dxdy
\\
& = \int_{I_0} \rho^{(1)}_{\sine, 4}(x)  dx +   \int_{I_0} dx \int_\R  \rho^{(2, T)}_{\sine, 4}(x,y)dy.  
\end{align*}
Then by using the equality \eqref{id-2}, we obtain the desired equality \eqref{sum-0-4}. 

The proof of Proposition \ref{prop-sine-4} is complete. 
\end{proof}

\subsubsection{Comments}
For the symplectic sine process,  let us now consider the stationary stochastic process $(X_n^{(1)})_{n\in\Z}$. We show that 
\begin{align}\label{div-coef}
\sum_{n\in\Z} | \Cov(X_0^{(1)}, X_n^{(1)}) | = \infty. 
\end{align}
If follows that the assumptions of Proposition~\ref{prop-l-rig} doesn't hold in this case, see Remark~\ref{rem:summable}.

Indeed, by Lemma \ref{lem-cov} and the inequality \eqref{es-large-d}, for $n\ne 0$,  we have 
\begin{multline}\label{low-bdd}
 | \Cov(X_0^{(1)}, X_n^{(1)}) |  =  \Big| \int_{I_0^{(1)} \times I_n^{(1)} }   \rho^{(2, T)}_{\sine, 4}(x,y)  dx dy\Big|   \ge 
 \\
  \ge    \Big|   \underbrace{ \int_{I_0^{(1) } \times I_n^{(1)}}  \frac{\cos(\pi(x-y))}{8 (x-y)}  dx dy}_{\text{denoted $\sigma_n^{(1)}$}} \Big| -  C   \int_{I_0^{(1)} \times I_n^{(1)} }   \frac{1}{(x-y)^2}  dx dy. 
\end{multline}
Note that by integration by parts, we obtain
\[
\sigma_n^{(1)}  = \int_{I_0^{(1)}} dx \left[ \frac{  \sin(\pi(y-x))}{8 \pi (x-y)} \Big|_{y=n- 1/2}^{ n + 1/2 }   - \int_{I_n^{(1)}}\frac{\sin(\pi (y-x))}{8\pi (x-y)^2} dy\right]. 
\]
Using the identity 
\[
\sin \Big[ \pi ( n  + 1/2 -x ) \Big]   = -  \sin \Big[ \pi ( n   - 1/2 -x ) \Big] = (-1)^n \cos(\pi x) , 
\] 
we obtain that 
\[
\frac{  \sin(\pi(y-x))}{8 \pi (x-y)} \Big|_{y=n- 1/2}^{ n + 1/2 }   = \frac{ (-1)^n  \cos  ( \pi x ) }{8 \pi }  \Big[  \frac{1}{x - n -  1/2  } + \frac{1}{x - n  +  1/2 } \Big]. 
\]
For any positive integer  $ n > 0$, we have 
\begin{multline*}
\Big| \int_{I_0^{(1)}}   \frac{  \sin(\pi(y-x))}{8 \pi (x-y)} \Big|_{y=n - 1/2}^{n + 1/2 }  dx \Big|   =  \int_{-1/2}^{1/2} \frac{\cos (\pi x) }{ 8 \pi}  \Big[  \frac{1}{ n +  1/2- x  } + \frac{1}{ n  -  1/2-x } \Big] dx  \ge 
\\
\ge \int_{0}^{1/3} \frac{\cos (\pi x) }{ 8 \pi} \cdot   \frac{1}{ n  -  1/2-x }  dx \ge \int_{0}^{1/3} \frac{1}{ 16 \pi}  \cdot  \frac{1}{ n  -  1/2-x }  dx \ge  \frac{1}{48\pi (n  - 1/2 ) }.  
\end{multline*} 
Therefore, 
\begin{multline*}
|\sigma_n^{(1)}| \ge   \Big| \int_{I_0^{(1)}}   \frac{  \sin(\pi(y-x))}{8 \pi (x-y)} \Big|_{y=n- 1/2}^{ n + 1/2 } dx \Big|    -  \Big|  \int_{I_0^{(1)}}  dx \int_{I_n^{(1)}}\frac{\sin(\pi (y-x))}{8\pi (x-y)^2} dy \Big|  \ge 
\\
\ge \frac{1}{48\pi (n  - 1/2 ) } -   \int_{I_0^{(1)}}  dx \int_{I_n^{(1)}}\frac{1}{8\pi (x-y)^2} dy. 
\end{multline*}
Combining the above inequality with \eqref{low-bdd}, we get
\begin{multline*}
 | \Cov(X_0^{(1)}, X_n^{(1)}) |  \ge  \frac{1}{48\pi (n  - 1/2 ) } -  (C + \frac{1}{8 \pi})  \int_{I_0^{(1)}}  dx \int_{I_n^{(1)}}\frac{1}{(x-y)^2} dy
\end{multline*}
By the argument in the proof of Proposition \ref{prop-sine-4}, we know that 
\[
\sum_{n\in\Z} \int_{I_0^{(1)}}  dx \int_{I_n^{(1)}}\frac{1}{(x-y)^2} dy < \infty. 
\]
Therefore, we get the claimed divergence
\begin{multline*}
\sum_{n\in\Z} | \Cov(X_0^{(1)}, X_n^{(1)}) |  \ge  \sum_{n \ge 1}\frac{1}{48\pi (n  - 1/2 ) } -  (C + \frac{1}{8 \pi})  \sum_{n \ge 1} \int_{I_0^{(1)}}  dx \int_{I_n^{(1)}}\frac{1}{(x-y)^2} dy = \infty. 
\end{multline*}

\noindent {\bf{Acknowledgements.}} 
We are deeply grateful to Alexei Klimenko for useful discussions.
The research of A. Bufetov on this project has received funding from the European Research Council (ERC) under the European Union's Horizon 2020 research and innovation programme under grant agreement No 647133 (ICHAOS).
Yanqi Qiu's research is supported by the National Natural Science Foundation of China, grants NSFC Y7116335K1, NSFC 11801547 and NSFC 11688101.
The research of P. Nikitin is supported by the RFBR grant 17-01-00433.

%\bibliography{mybib}

\begin{thebibliography}{1}

\bibitem{AS} M. Abramowitz and I. Stegun, Handbook of Mathematical Functions,  National Bureau of Standards,
Department of Commerce of the United States of America, Tenth Printing, 1972.

\bibitem{AB_det-rigidity} Alexander I. Bufetov, Rigidity of determinantal point processes with the Airy, the Bessel and the Gamma kernel, Bull. Math. Sci., 6:1 (2016), 163–172.

\bibitem{BDQ}
Alexander~I. Bufetov, Yoann Dabrowski, and Yanqi Qiu.
\newblock Linear rigidity of stationary stochastic processes.
\newblock {\em Ergodic Theory Dynam. Systems}, 38 (2018), no.7, 2493--2507.
  
  \bibitem{BQ-JHerm}
Alexander~I. Bufetov and Yanqi Qiu.
\newblock  $J$-{H}ermitian determinantal point processes: balanced rigidity and balanced {P}alm equivalence.
\newblock{\em }
\newblock {\em Math. Ann.}, 371 (2018), no. 1-2, 127--188.  

\bibitem{rena} Reda Chhaibi, Joseph Najnudel, 
\newblock Rigidity of the $Sine_{\beta}$  process.
\newblock{\em  arXiv:1804.01216.}

\bibitem{Forrester-log}
Peter~J. Forrester.
\newblock {\em Log-gases and random matrices}, volume~34 of {\em London
  Mathematical Society Monographs Series}.
\newblock Princeton University Press, Princeton, NJ, 2010.

\bibitem{Ghosh-sine}
Subhroshekhar Ghosh.
\newblock Determinantal processes and completeness of random exponentials: the
  critical case.
\newblock {\em Probab. Theory Related Fields}, pages 1--23, 2014.

\bibitem{Ghosh-rigid}
Subhroshekhar Ghosh and Yuval Peres.
\newblock Rigidity and tolerance in point processes: {G}aussian zeros and
  {G}inibre eigenvalues.
\newblock {\em Duke Math. J.} 166 (2017), no.10, 1789--1858.

\bibitem{TW-Bessel}
C. A. Tracy and H. Widom.
\newblock Level spacing distributions and the {B}essel kernel. 
\newblock {\em Comm. Math. Phys.} 161, no. 2 (1994), 289–309.
\end{thebibliography}
%\bibliographystyle{plain}
%

\def\cprime{$'$} \def\cydot{\leavevmode\raise.4ex\hbox{.}}

\end{document}